\documentclass[10pt]{amsart}
\usepackage{latexsym, amssymb, amsfonts, amscd}
\usepackage[cmtip, all]{xy}
\usepackage{enumerate}
\usepackage{colonequals}

\theoremstyle{plain}
\newtheorem{Thm}{Theorem}[subsection]
\newtheorem{Cor}[Thm]{Corollary}

\newtheorem{Prop}[Thm]{Proposition}
\newtheorem{Lem}[Thm]{Lemma}

\newtheorem{Cl}[Thm]{Claim}

\newtheorem{Thm'}{Theorem}[section]
\newtheorem{Cor'}[Thm']{Corollary}
\newtheorem{Prop'}[Thm']{Proposition}
\newtheorem{Lem'}[Thm']{Lemma}
\newtheorem{Cl'}[Thm']{Claim}
\theoremstyle{definition}
\newtheorem{Def}[Thm]{Definition}

\newtheorem{Emp}[Thm]{}

\newtheorem{Not}[Thm]{Notation}

\newtheorem{Def'}[Thm']{Definition}
\newtheorem{Rem'}[Thm']{Remark}
\newtheorem{Rem1'}[Thm']{Remarks}
\newtheorem{Emp'}[Thm']{}
\newtheorem{Ex'}[Thm']{Example}
\newtheorem{Exs'}[Thm']{Examples}
\newtheorem{Con'}[Thm']{Construction}
\newtheorem{Not'}[Thm']{Notation}
\newtheorem{Q'}[Thm']{Question}
\numberwithin{equation}{section}

\newcommand{\ql}{\B{Q}_{\ell}}
\newcommand{\qlbar}{\overline{\B{Q}_l}}

\newcommand{\La}{\Lambda}

\newcommand{\ov}{\overline}

\newcommand{\fq}{\B{F}_q}

\newcommand{\B}[1]{\mathbb#1}
\newcommand{\cal}[1]{\mathcal{#1}}
\newcommand{\form}[1]{(\ref{Eq:#1})}
\newcommand{\C}[1]{\cal#1}

\newcommand{\isom}{\overset {\thicksim}{\to}}

\newcommand{\si}{\sigma}

\newcommand{\surj}{\twoheadrightarrow}
\newcommand{\lra}{\longrightarrow}
\newcommand{\lla}{\longleftarrow}
\newcommand{\hra}{\hookrightarrow}
\newcommand{\wt}{\widetilde}

\newcommand{\Gm}{\Gamma}

\newcommand{\p}{\partial}

\newcommand{\dt}{\delta}
\newcommand{\Dt}{\Delta}

\newcommand{\al}{\alpha}

\newcommand{\rl}[1]{Lemma \ref{L:#1}}
\newcommand{\rn}[1]{Notation \ref{N:#1}}
\newcommand{\rcl}[1]{Claim \ref{C:#1}}
\newcommand{\rp}[1]{Proposition \ref{P:#1}}

\newcommand{\re}[1]{\ref{E:#1}}
\newcommand{\rco}[1]{Corollary \ref{C:#1}}

\newcommand{\rt}[1] {Theorem \ref{T:#1}}
\newcommand{\rd}[1]{Definition \ref{D:#1}}
\newcommand{\sm}{\smallsetminus}

\newcommand{\Spec}{\operatorname{Spec}}
\newcommand{\Aut}{\operatorname{Aut}}

\newcommand{\Tr}{\operatorname{Tr}}

\newcommand{\on}{\operatorname}
\newcommand{\red}{\operatorname{red}}

\newcommand{\Bl}{\operatorname{Bl}}
\newcommand{\Irr}{\operatorname{Irr}}

\newcommand{\ram}{\operatorname{ram}}

\newcommand{\Id}{\operatorname{Id}}

\newcommand{\gdeg}{\operatorname{gdeg}}
\newcommand{\sL}{{\mathcal L}}

\newcommand{\sO}{{\mathcal O}}

\newcommand{\F}{{\mathbb F}}

\renewcommand{\P}{{\mathbb P}}
\renewcommand{\bar}{\overline}

\def\tilde{\widetilde}

\begin{document}
%Topmatter

\title[The Hrushovski--Lang--Weil estimates]{The Hrushovski--Lang--Weil estimates}
%\author{K.~V.~Shuddhodan \& Yakov Varshavsky}

\author{K.~V.~Shuddhodan}
\address{Department of Mathematics,
Purdue University, West Lafayette, IN 47907, USA}
\email{kvshud@purdue.edu}

\author{Yakov Varshavsky}
\address{Institute of Mathematics\\
Hebrew University\\
Givat-Ram, Jerusalem,  91904\\
Israel}
\email{vyakov@math.huji.ac.il }

\thanks{The work was supported by
THE ISRAEL SCIENCE FOUNDATION (Grant No. 1017/13)}
%\date{February 2006}
%\keywords{Lefschetz trace formula, Deligne's conjecture}
%\subjclass[2000]{Primary: 14F20; Secondary: 11G25, 14G15}
\date{\today}

%\abstract

\begin{abstract}
In this work we give a geometric proof of Hrushovski's generalization of the Lang-Weil estimates on the number of
points in the intersection of a correspondence with the graph of Frobenius.
\end{abstract}

\maketitle

\tableofcontents

\section*{Introduction}

In this work we give a purely geometric proof of a generalization of the Lang--Weil estimates \cite[Theorem 1]{LW54} obtained by Hrushovski \cite[Theorem 1.1]{Hr}. In order to formulate it we need to introduce some notations.

\begin{Emp'} \label{E:geom deg}
{\bf Geometric degree.}
Let $k$ be an algebraically closed field. Recall (see \cite[Ex 2.5.2]{Fu}) that to every closed subscheme $X\subseteq\B{P}_k^n$  one can associate its {\em degree} $\on{\deg}(X)$.

(a) For a closed reduced subscheme $X \subseteq\B{P}^n_k$, we denote by $\on{gdeg}(X)$ the sum of the degrees of \textit{all} irreducible components of $X$, and call it the {\em geometric degree} of $X$ (compare \cite{BM93}).

(b) More generally, for a (locally closed) reduced subscheme $X\subseteq\B{P}^n_k$, we denote by $\on{gdeg}(X)$ the minimum $\on{min}\{\on{gdeg}(X_1) + \on{gdeg}(X_2)\}$, taken over all pairs $X_1,X_2$ of closed reduced subschemes of $\B{P}^n_k$ with $X=X_1 \sm X_2$.
\footnote{Note that if $X\subseteq\B{P}^n_k$ is closed, then the geometric degree $\gdeg(X)$ defined in (b) coincides with that of (a).}

(c) For a reduced subscheme  $C$ of $\B{P}^m_X\subseteq\B{P}^m_k \times \B{P}^n_k$, we denote by $\on{gdeg}(C)$ the geometric degree of $C$, viewed as a subscheme of $\B{P}^{mn+m+n}_k$ via the Segre embedding.
\end{Emp'}

\begin{Emp'} \label{set up}
{\bf Set up.} (a) Let $k$ be an algebraically closed field of characteristic $p>0$, $q$ a power of $p$, and $\si_{q}:\on{Spec} k\to \on{Spec} k$ the absolute $q$-th Frobenius automorphism.

(b) Let $X^0\subseteq\B{P}^n_k$ be a locally closed reduced subscheme, $(X^0)^{(q)}$ the base change of $X$ along $\si_{q}$, and
$C^0\subseteq X^0 \times (X^0)^{(q)}$ be a closed reduced subscheme. In particular, $C^0$ is a reduced subscheme of $\B{P}^n_k \times \B{P}^n_k$,
thus we can talk about $\on{gdeg}(C^0)$.

(c) We denote by $c^0=(c^0_1,c^0_2):C^0\hra X^0 \times (X^0)^{(q)}$ the inclusion map.
%and always assume that $c^0_2$ is quasi-finite.

(d) Let $\on{F}_{X^0/k,q}$ be the induced relative Frobenius map  $X^0\to (X^0)^{(q)}$ over $k$, and denote by $\Gm^0:=\Gamma_{X^0,q}\subseteq X^0 \times (X^0)^{(q)}$ the graph of $\on{F}_{X^0/k,q}$.
\end{Emp'}

Now we are ready to formulate the main result of this work.
%In this work we give a purely geometric proof the following generalization of the Lang-Weil estimates \cite[Theorem 1]{LW54},~first obtained by Hrushovski \cite[Theorem 1.1]{Hr} using techniques from model theory and intersection theory.
%$C$ can be realized as a subscheme of $\B{P}^{n^2+n}_k$ (via the Segre embedding),~thus giving meaning to degree of $C$.

\begin{Thm'}\label{T:main}\cite[Theorem 1.1]{Hr} For every pair $n,\dt\in\B{N}$ there exists a constant $M=M(n,\dt)$ satisfying the following properties:

\vskip 4truept
%\begin{enumerate}[(a)]
(a) Let $(k,q,X^0,C^0)$ be a quadruple as in \ref{set up} such that $c^0_2$ is quasi-finite,
%\footnote{It is enough to assume that either  $c^0_1$ or $c^0_2$ is quasi-finite.},
$\on{gdeg}(X^0)\leq\dt$, $\on{gdeg}{C^0}\leq \dt$ and $q\geq M$.

Then the intersection $(C^0 \cap \Gamma^0)(k)$ is finite, and its cardinality satisfies
\begin{equation} \label{Main1}
\#(C^0 \cap \Gamma^0)(k) \leq M q^{\on{dim}(C^0)}.
\end{equation}

\vskip 4truept

(b) Let $(k,q,X^0,C^0)$ be a quadruple as in (a) such that $C^0$ is irreducible and $c^0_1$ is dominant.
%\footnote{It is enough to assume that either $c^0_1$ is quasi-finite and $c^0_2$ is dominant or $c^0_1$ is quasi-finite and $c^0_2$ is dominant.}.

Then the finite number $\#(C^0 \cap \Gamma^0)(k)$ satisfies
\footnote{Our assumptions imply that $X^0$ is irreducible, $\dim(X^0)=\dim(C^0)$, and both $c^0_1$ and $c^0_2$ are generically finite. In particular, we can talk about the degree $\on{deg}(c^0_1)$ of $c^0_1$ and the inseparable degree $\on{deg}(c^0_2)_{\on{insep}}$ of $c^0_2$.}
\begin{equation} \label{Main2}
\left|\#(C^0 \cap \Gamma^0)(k)-\frac{\on{deg}(c^0_1)}{\on{deg}(c^0_2)_{\on{insep}}}q^{\on{dim}(C^0)}\right | \leq M q^{\on{dim}(C^0)-\frac{1}{2}}.
\end{equation}
%
%Here $\alpha_1$ and $\beta_2$ are respectively the degree and the inseparable degree of $c_1$ and $c_2$.
%\end{enumerate}
\end{Thm'}

\begin{Emp'}
{\bf Remark.} The result can be slightly generalized. Namely, it suffices to assume that either $c^0_1$ or $c^0_2$  is quasi-finite
in \rt{main}(a) and that either $c^0_1$ is quasi-finite and $c^0_2$ is dominant or
$c^0_1$ is quasi-finite and $c^0_2$ is dominant in \rt{main}(b).
\end{Emp'}

\begin{Emp'}
{\bf A particular case.}
Assume now that $k$ is an algebraic closure $\B{F}$ of a finite field $\fq$. For any scheme $X^0$ over $\fq$, we have a natural isomorphism
$(X^0)^{(q)}\cong X^0$. Then one obtains the following consequence of Theorem \ref{T:main}.
\end{Emp'}

\begin{Cor'} \label{non_uniform_estimates} (\cite[Theorem 1.3]{Shu19a}).
Let $c^0:C^0 \to X^0 \times X^0$ be a morphism of schemes finite type over $k$ such that $X^0$ is defined over $\fq$ and either $c^0_1$ or $c^0_2$ is quasi-finite. Then there exists an integer $M$ such that for every $n \geq M$,

(a) the set $c^{-1}(\Gm_{X^0,q^n})(k)$ is finite and $\#c^{-1}(\Gm_{X^0,q^n})(k) \leq Mq^{n\on{dim}(C^0)}$.

(b) if in addition $C^0$ is irreducible with $c_1$ and $c_2$ dominant, then we have an inequality

\begin{equation*}\label{main_theorem_estimate}
\left |\#(c^0)^{-1}(\Gm_{X^0,q^n})(k)-\frac{\on{deg}(c_1)}{\on{deg}(c_2)_{\on{insep}}}q^{n \on{dim}(C^0)}\right| \leq M q^{n(\on{dim}(C^0)-\frac{1}{2})}.
\end{equation*}

%\noindent where $\alpha_1$ and $\beta_2$ are the degree and inseparable degree of $c_1$ and $c_2$ respectively.

\end{Cor'}

We also have the following very useful consequence of Corollary \ref{non_uniform_estimates} with applications to algebraic dynamics (\cite{Fa},\cite{Amerik}),~group theory (\cite{BS}) and algebraic geometry (\cite{EM},~\cite{Esnault_Srinivas_Bost}).

\begin{Cor'}\label{dense_correspondence}\cite[Corollary 1.2]{Hr}, \cite[Theorem 0.1]{Va}. Let $c^0:C^0 \to X^0 \times X^0$ be a morphism of schemes finite type over $k$ such that $C^0$ is irreducible, $c_1$ and $c_2$ are dominant,
$c_1$ or $c_2$ is quasi-finite, and $X^0$ is defined over $\fq$. Then for every sufficiently large $n$, the preimage
$(c^0)^{-1}(\Gm_{X^0,q^n})(k)$ is non-empty .
\end{Cor'}

\begin{Emp'}
{\bf Our strategy.} To prove the result, we follow the strategy of \cite{Va} rather closely. %In fact, most of the arguments just {\em twisted} analogues of those carried out in \cite{Va}.
The only (but essential) difference is that unlike \cite{Va}, the refined intersection theory \`a la Fulton is used.
\footnote{An interesting question is whether it is possible to prove our result by extending the method of the proof of
Corollary \ref{non_uniform_estimates} in \cite{Shu19a}. Unlike \cite{Hr} and \cite{Va}, the method of \cite{Shu19a} does not make use of
alterations, and uses techniques from $\ell$-adic cohomology, properties of the intersection complex and theory of weights instead.
One of the obstacles we don't know how to resolve at the moment is how to define a suitable ``norm'' on cohomological correspondences,
compatible with the norm of algebraic cycles, given in Appendix \ref{uniform_Gysin_Appendix}.}
\end{Emp'}

\begin{Emp'}
{\bf Outline of the proof.}
Our proof of \rt{main} goes as follows. Using Noetherian induction, we reduce \rt{main} to \rt{main}(b). Next, replacing $X^0$ by an open subscheme,  we reduce to a particular case of \rt{main}(b), where $X^0$ is smooth, $c^0_2$ decomposes as a composition of a flat universal homeomorphism and an \'etale morphism. In this case, the intersection $C^0\cap\Gm^0$ is finite when $q>\on{deg}(c_2)_{\on{insep}}$, which we assume from now on.

Let $X\subseteq\B{P}^n$ and $C\subseteq X\times X^{(q)}$ be the closures $X^0$ and $C^0$, respectively, and
let $c=(c_1,c_2):C\to X\times X^{(q)}$ be the inclusion map. We say that $\p X:=X\sm X^0$ is {\em locally $c$-invariant},
if every point $x\in X$ has an open neighborhood $U\subseteq X$ such that $c_2^{-1}(\p X^{(q)}\cap U^{(q)})\cap
c_1^{-1}(U)\subseteq c_1^{-1}(\p X\cap U)$. By a {\em twisted} version of a argument of \cite{Va}, we reduce to the situation where
$\p X$ is locally $c$-invariant. Namely, this happens after we replace $X^0$ by an open subscheme and $X$ by a blow-up.

Our main observation that in this case it is possible to write a formula for the cardinality $\#(C^0\cap\Gm^0)(k)$ in cohomological terms.

By a theorem of de Jong \cite{dJ}, there exists an alteration $f:X'\to X$, where
$X'$ is smooth, and the boundary $\p X':=f^{-1}(\p X)_{\red}$ is a union of smooth
divisors $\{X'_i\}_{i\in I}$ with  normal crossings. Mimicking (\cite{Pi}, \cite{Va}), we consider the blowup $\wt{Y}:=\Bl_{\cup_{i\in I}(X'_i\times
(X'_i)^{(q)})}(X'\times X'^{(q)})$, and denote by $\pi:\wt{Y}\to X'\times X'^{(q)}$ the projection map. Let $\Gm\subseteq X'\times X'^{(q)}$
be the graph of $\on{F}_{X'/k,q}$, and denote by $\wt{C}\subseteq\wt{Y}$ and $\wt{\Gm}\subseteq \wt{Y}$
the strict preimages of $C$ and $\Gm$, respectively.

Set $X'^0:=f^{-1}(X^0)\subseteq X'$ and let $g^0:X'^0\times (X'^0)^{(q)}\to X^0\times (X^0)^{(q)}$ be the restriction
of $f\times f^{(q)}$. Then $\pi:\wt{Y}\to Y$ in an isomorphism over $X'^0\times (X'^0)^{(q)}$, therefore
there exists an canonical open embedding $X'^0\times (X'^0)^{(q)}\hra\wt{Y}$, which induces an open embedding
$(g'^{0})^{-1}(C^0)\hra\wt{C}\subseteq\wt{Y}$.

Denote by $[C^0]\in A_d([C^0])$ the class of $C^0$, where $A_d(-)$ denotes the group of cycle classes, let
$\al^0:=(g^0)^*([C^0])\in  A_d((g'^{0})^{-1}(C^0))$ be the refined Gysin pullback of $[C^0]$, and let $\wt{\al}\in  A_d(\wt{Y})$
be the ``closure" of $\al^0$. Since $\wt{Y}$ is smooth of dimension $2d$, we can consider the intersection number $\wt{\al}\cdot[\wt{\Gm}]\in\B{Z}$.

Our strategy is to calculate the intersection number $\wt{\al}\cdot[\wt{\Gm}]$ in two ways: local and global.
Arguing as in \cite{Va}, we see that if $q$ greater than the ramification of $c'_2$ along $\p X'^{(q)}$, then we have
$\wt{C}\cap\wt{\Gm}\subseteq\pi^{-1}(X'^0\times (X'^0)^{(q)})$. Combining this with the refined projection formula and a local calculation,
we get the equality
\begin{equation} \label{Eq:int number}
\wt{\al}\cdot[\wt{\Gm}]=\deg(c^0_2)_{\on{insep}}\deg(f) \#(C^0\cap\Gm^0)(k).
\end{equation}

To get a second expression we proceed as follows, for every subset $J\subseteq I$, we denote by $X'_J$ the intersection
$\cap_{i\in J}X'_i$. In particular, $X'_{\emptyset}=X'$. For every $J\subseteq I$, the cycle class
$\wt{\al}\in A_d(\wt{Y})$ gives rise to the cycle class
$\wt{\al}_J\in A_{d-|J|}(X'_J\times (X'_J)^{(q)})$.

For every prime $\ell\neq\on{char}(k)$ and $i\in\B{Z}$, the class $\wt{\al}_J$
gives rise to a homomorphism $H^i(\wt{\al}_J):H^i(X'_J,\ql)\to H^i((X'_J)^{(q)},\ql)$ between $\ell$-adic cohomology groups,
hence to an endomorphism $\on{F}_{X'_J,q}^*\circ H^i(\wt{\al}_J)$ of $H^i(X'_J,\ql)$.
We then show the following equality %(compare \cite{Va} or \cite[Ch IV]{Laf})
\begin{equation} \label{Eq:altsum}
\wt{\al}\cdot[\wt{\Gm}]=\sum_{J\subseteq I}(-1)^{|J|}\sum_{i=0}^{2(d-|J|)}(-1)^i\Tr(\on{F}_{X'_J,q}^*\circ H^i(\wt{\al}_J), H^i(X'_J,\ql)).
\end{equation}

It is well-known that each $\Tr(\on{F}_{X'_J,q}^*\circ H^i(\wt{\al}_J))\in\B{Z}$ and is independent of $\ell$.
Combining equalities \form{int number}, \form{altsum} and identity $\Tr(\on{F}_{X',q}^*\circ H^{2d}(\wt{\al}_{\emptyset}))=\deg(c^0_1)\deg(f)$, to show \rt{main} it suffices to show that there exists a constant $N=N(n,\dt)$ such that for every $J\subseteq I$ and $i\in\B{N}$ we have  an inequality $|\Tr(\on{F}_{X'_J,q}^*\circ H^i(\wt{\al}_J))|\leq N q^{i/2}$.

To prove the result, we use a combination of the Deligne's purity theorem and a variant of the argument of \cite{Hr}. Namely, we introduce
a version of ``Hrushovski's norm'' on cycle classes, and show that it is ``submultiplicative". Actually, we deduce the submultiplicativity
from a general claim asserting that the refined Gysin pullback is ``uniformly bounded".
\end{Emp'}

%Now the assertion of the corollary follows from the theorem of Deligne which asserts that
%for every embedding $\ql\hra\B{C}$, all eigenvalues of the action of
%$\phi^*_{q^n}$ on $H^i(X_J,\ql)$ have absolute values $q^{\frac{ni}{2}}$.

\begin{Emp'}
{\bf Plan of the paper.}
The paper is organized as follows:

In Section 1 we introduce twisted versions of constructions and results of \cite{Va}. Namely,
we study correspondences with locally invariant boundary in subsection 1.1
and a version of Pink's construction in subsection 1.2.

In Section 2 we give a cohomological expression for the number of points in the intersection. Namely, we review definitions, constructions and results from intersection theory and $\ell$-adic cohomology in subsection 2.1\footnote{Note that though we follow Fulton's book \cite{Fu} rather closely, our notation is slightly different.}, introduce our main construction and formulate our main formula in subsection 2.2, and finally prove the formula in subsection 2.3.

In Section 3 we carry out the proof of \rt{main}. Namely, in subsection 3.1 we review ``boundness results'' shown in the appendices; in subsection
3.2 we reduce the result to a particular case of \rt{main}(b), while in subsection 3.3 we deduce the final estimate from the formula for the number of points in the intersection, obtained in Section 2.

We finish the paper by two appendixes in which we show ``boundness results'' used in Section~3. %In particular, the reader interested only in the non-uniform estimates (Corollary \ref{non_uniform_estimates}) may entirely skip the appendices.

In Appendix A we introduce a notion of a bounded collection and show that the notion of ``boundness" is preserved by various standard algebro-geometric operations. In the case of quasi-projective varieties we also relate this notion with a much more naive notion of a boundness of geometric degree\footnote{In particular, we give an elementary proof of theorem of Kleiman, following a suggestion of Hrushovski.}. %All results in this section seem to be well known to experts.

In Appendix B we introduce a norm on the Chow groups of quasi-projective schemes and study its properties.
Namely, we show that various standard operations, that is, proper pushforwards, smooth pullbacks and refined Gysin pullbacks are
``uniformly bounded", and deduce bounds on the spectral radius. We would like to stress that though our notion of a norm is much
more ``naive'' than the one used by Hrushovski, our results are strongly motivated by his results. In particular,
the last subsection  is almost identical to the corresponding part of \cite{Hr}.

Note that though results in Appendix A seem to be well known to experts, some of the results of Appendix B could be of independent interest.
\end{Emp'}

\begin{Emp'}
{\bf Acknowledgements.}
This work would not be possible without a masterpiece \cite{Hr} of Ehud Hrushovski. We would like also to thank him for enlightening conversations over the years. In addition, the idea of the proof of \rt{boundedness_degree} belongs to him. We also thank H\'el\`ene Esnault, Mark Sapir, and Luc Illusie for their interest and stimulating conversations. KVS would also like to thank Vasudevan Srinivas for his constant support and helpful discussions, especially around Kleiman's boundedness theorem.

This work has been conceived when the first named author visited the Hebrew University of Jerusalem and has beed supported by the ISF grant 1017/13.
\end{Emp'}

\begin{Emp'}
{\bf Conventions.} In this work $k$ will be always an algebraically closed field,
all schemes over $k$ will be assumed to be of finite type. By a variety we mean an integral scheme (of finite type) over $k$.
For two schemes $X,Y$ over $k$ we will write $X\times Y$ instead of $X\times_k Y$. For a scheme $X$ over $k$ and an automorphism $\si$ of $k$ we denote by ${}^{\si}X$ the base change of $X$ along $\sigma:\Spec k\to \Spec k$.
\end{Emp'}

%\begin{Emp}
%{\bf Notation.}
%(a) For a morphism of schemes $f \colon X \to Y$ and any point $y \in Y$,~$X_y$ (resp $f_y$) will denote the base change of $X$ (resp. $f$) along $\Spec(k(y))$.~Similarly when $\bar{y}$ is a geometric point of $Y$,~$X_{\bar{y}}$ and $f_{\bar{y}}$ will denote the corresponding base change to $\bar{y}$.
%(b) For a quasi-projective scheme $X \hookrightarrow \B{P}^n_k$,~by $\partial X$ we mean the closed subscheme $\bar{X} \sm X \hookrightarrow \B{P}^{n}_k$ with the reduced induced structure.
%(c) For a closed subscheme $X \subseteq \B{P}^n_k$,~we denote by $\C{I}_X$ the ideal sheaf of $X$.
%(d) Let $Y \subseteq X$ be a closed subscheme defined by an ideal sheaf $\C{I}$.~We use $Y_r$ to denote the closed subscheme of $X$ defined by the ideal sheaf $\C{I}^{r+1}$.
%\end{Emp}

\section{Twisted versions of results of \cite{Va}}
Let $k$ be an algebraically closed  field, and let $\si\in\Aut(k)$ be an automorphism.

\subsection{Correspondences with locally invariant boundary}

In this subsection we recall twisted versions of the
constructions and results of \cite{Va}, where the case $\si=\Id$ is treated.

\begin{Not} \label{N:cor}
Let $k$ be an algebraically closed  field. %In this work we only will be interested in the case when $k$ is either algebraically closed or finite.

(a) By a {\em correspondence}, we mean a morphism of schemes of the form $c=(c_1,c_2):C\to X\times {}^{\si}X$, where
$X$ and $C$ are schemes of finite type over $k$.

(b) For a correspondence $c:C\to X\times {}^{\si}X$ and an open subscheme
$U\subseteq X$, we denote by
$c|_U:c_1^{-1}(U)\cap c_2^{-1}({}^{\si}U)\to U\times {}^{\si}U$ the
restrictions of $c$.

(c) Let $c:C\to X\times X$ and $\wt{c}:\wt{C}\to\wt{X}\times{}^{\si}\wt{X}$ be
two correspondences. By a {\em morphism} from $\wt{c}$ to $c$,
we mean a pair of morphisms $[f]=(f,f_C)$, making
the following diagram commutative
\begin{equation} \label{Eq:funct}
\CD
        \wt{X}  @<{\wt{c}_1}<<   \wt{C}       @>{\wt{c}_2}>>        {}^{\si} \wt{X}\\
        @V{f}VV                        @V{f_C}VV                       @VV{}^{\si}{f}V\\
        X @<{c_1}<<                   C    @>{c_2}>>           {}^{\si}X.
\endCD
\end{equation}

(d) Suppose that we are given correspondences $\wt{c}$ and $c$ as in (c) and a morphism $f:\wt{X}\to X$.
We say that $\wt{c}$ {\em lifts} $c$, if there exists a morphism $f_C:\wt{C}\to C$ such that
$[f]=(f,f_C)$ is a morphism from $\wt{c}$ to $c$.
\end{Not}

%\begin{Not} \label{N:closed}
%(a) For a closed subscheme $Z\subseteq X$, we denote by $\C{I}_Z\subseteq\C{O}_X$ the sheaf of
%ideals of $Z$, and let
%\end{Not}

\begin{Def} \label{D:inv}
Let $c:Y\to X\times {}^{\si}X$ be a correspondence, and let $Z\subseteq X$ be a
closed subset, viewed as a closed reduced subscheme.

(a) We say that $Z$ is {\em $c$-invariant}, if $c_1(c_2^{-1}({}^{\si}Z))$
is set-theoretically contained in $Z$.

%(b) We say that $Z$ is {\em $c$-invariant in a neighborhood of
%fixed points}, if there exists an open an neighborhood $W\subseteq
%Y$ of $Fix(c)$ such that $Z$ is $c|_W$-invariant.

(b) We say that $Z$ is {\em locally $c$-invariant}, if for every point
$x\in Z$ there exists an open neighborhood $U\subseteq X$ of $x$
such that $Z\cap U\subseteq U$ is
$c|_U$-invariant.
\end{Def}

%\begin{Emp} \label{E:qf}
%{\bf Example.} Let $c:Y\to X\times X$ be a correspondence, and let
%$x\in X$ be a closed point such that $c_2^{-1}(x)$ is finite. Then
%$\{x\}$ is $c$-invariant in a neighborhood of fixed points.
%Indeed, $W:YC\sm [c_2^{-1}(x)\cap c_1^{-1}(X\sm x)]$ is the
%required open subset.
%\end{Emp}

%The following lemma lists simple properties of this notion.

\begin{Lem} \label{L:locinv}
Let $[f]=(f,f_C)$ be a morphism from a correspondence  $\wt{c}:\wt{C}\to\wt{X}\times{}^{\si}\wt{X}$ to
$c:C\to X\times {}^{\si}X$. If $Z\subseteq X$ is a locally $c$-invariant
closed subset, then $f^{-1}(Z)_{\red}\subseteq\wt{X}$ is locally $\wt{c}$-invariant.

%(a) If $Z$ is locally $c$-invariant, then it is $c$-invariant in a
%neighborhood of fixed points.

%(b) If $Z_1,Z_2\subseteq X$ are two closed locally $c$-invariant subsets, then the union
%$Z_1\cup Z_2$ is also locally $c$-invariant.
\end{Lem}

\begin{proof}
Repeat the argument of \cite[Lem 1.3]{Va}, where the case $\si=\Id$ is shown.
\end{proof}

\begin{Not} \label{N:fg}
Let $c:Y\to X\times {}^{\si}X$ be a correspondence, and $Z\subseteq X$ a
closed subset. Following \cite{Va}, to this data we associate closed subsets
\begin{equation} \label{Eq:f}
F(c,Z):=c_2^{-1}({}^{\si}Z)\cap c_1^{-1}(X\sm Z)\subseteq Y,
\end{equation}
\begin{equation} \label{Eq:g}
G(c,Z):=\cup_{S\in \Irr(F(c,Z))}[\ov{c_1(S)}\cap {}^{\si^{-1}}\ov{c_2(S)}]\subseteq X,
\end{equation}
where $\Irr(F(c,Z))$ denotes the set of (reduced) irreducible components of $F(c,Z)$, and
$\ov{c_1(S)}\subseteq X$ and $\ov{c_2(S)}\subseteq {}^{\si}X$ are the closures of $c_i(S)$.
\end{Not}

\begin{Emp} \label{E:dim}
{\bf Remark.} Arguing as in \cite[Remark 1.5(a) and Cor. 1.7(a)]{Va} one sees that a closed subset $Z\subseteq X$ is $c$-invariant
(resp. locally $c$-invariant) if and only if $F(c,Z)=\emptyset$ (resp. $G(c,Z)=\emptyset$).
\end{Emp}

\begin{Emp}
{\bf Notation.}
For a closed subscheme $Z\subseteq X$, we denote by  $\C{I}_{Z,X}\subseteq\C{O}_X$ the sheaf of ideals of $Z$.
\end{Emp}

The following result is a twisted variant of \cite[Prop 2.3]{Va}.

\begin{Prop} \label{P:blowup}
Let  $c:C\to X\times {}^{\si}X$ be a correspondence over $k$
such that $X$ and $C$ are irreducible, $c_2$ is dominant quasi-finite. Let $X^0 \subseteq X$ be any non-empty open subset of $X$.

Then there exists a non-empty open subset $V\subseteq X^0$ and a blow-up $\pi:\wt{X}\to X$, which is an
isomorphism over $V$, such that for every correspondence $\wt{c}:\wt{C}\to \wt{X}\times\wt{X}$ lifting $c$, the closed
subset $\wt{X}\sm \pi^{-1}(V)\subseteq \wt{X}$ is locally $\wt{c}$-invariant.
\end{Prop}

\begin{proof}
The proof is a twisted analog of that of \cite[Lem 2.2 and Prop 2.3]{Va}.

%Since $\dim C=\dim X$ and $c_2$ is dominant, there exists a non-empty maximal open subset
%$U_0\subseteq X^0$ such that $c_2|_{c_2^{-1}({}^{\si}U_0)}$ is quasi-finite.
We set $U_0:=X^0$, and by induction, we define for every $j\geq 0$ open subsets
$V_j\subseteq U_j\subseteq X^0$ by the rules
\begin{equation} \label{Eq:vj}
V_j:=U_j\sm \ov{c_1(c_2^{-1}({}^{\si}X\sm {}^{\si} U_j))},
\end{equation}
\begin{equation} \label{Eq:zj}
Z_{j}:=G(c|_{U_{j}},U_{j}\sm V_{j}), \text{ and }U_{j+1}:=
U_{j}\sm Z_{j}\subseteq U_j.
\end{equation}

Arguing as in \cite[Lem 2.2]{Va} we see that each $V_j$ is non-empty and set $V:=V_{\dim X}$ and
$F:=F(c,X\sm V)=c_2^{-1}({}^{\si}X\sm {}^{\si}V)\cap c_1^{-1}(V)$.

For every $S\in\Irr(F)$, we denote by $Z_S:=\ov{c_1(S)}\cap {}^{\si^{-1}}\ov{c_2(S)}\subseteq X$ the schematic intersection,
let $Z\subseteq X$ be a closed subscheme such that  the sheaf of ideals
$\C{I}_{Z,X}\subseteq\C{O}_X$ equals the product $\C{I}_{Z,X}=\prod_{S\in \Irr(F)}\C{I}_{Z_S,X}\subseteq\C{O}_X$,
let $\wt{X}$ be the blow-up $\Bl_{Z}(X)$, and denote by
$\pi:\wt{X}\to X$ the canonical projection.

Then the argument of \cite[Prop 2.3]{Va} and \cite[Lem 2.2]{Va} implies that $V$ and $\pi$ satisfy the required properties.
\footnote{Alternatively, one can observe that in the arguments of \cite[Lem 2.2 and Prop 2.3]{Va} one does not use the assumption that
the morphisms $c_1$ and $c_2$ are $k$-linear. Therefore our assertion follows from the proof \cite[Lem 2.2 and Prop 2.3]{Va} in the case of correspondence $X\overset{c_1}{\lla}C\overset{c'_2}{\lra}X$, where $c'_2$ is the composition $C\overset{c_2}{\lra}{}^{\si}X\overset{\si}{\lra}X$.}
\end{proof}

\subsection{A version of Pink's construction}

\begin{Emp} \label{E:pink}
{\bf The construction.}
(a) Let $X'$ be a smooth scheme of relative dimension $d$ over a field $k$, and let $\{X'_{i}\}_{i\in I}$ be a
finite collection of smooth divisors with normal crossings.
We set $\p X':= \cup_{i\in I}X'_{i}\subseteq X'$ and $X'^0:=X'\sm \p X'$.

(b) For every $J\subseteq I$, we set $X'_J:=\cap_{i\in J} X'_{i}$. In particular, $X'_{\emptyset}=X'$.
Then every $X'_{J}$ is either empty, or smooth over $k$ of relative dimension
$d-|J|$. We also set $\p X'_J:=\cup_{i\in I\sm J}X'_{J\cup\{i\}}$, and $(X'_J)^0:=X'_J\sm \p X'_J$.

(c) For every $i\in I$, we set $Y_i:=X'_i\times {}^{\si}(X'_i)$. We also set $Y:=X'\times {}^{\si}(X')$, $\p Y:=\cup_{i\in I} Y_i$,
$Y^0:=Y\sm \p Y$. In particular, we have $X'^0\times {}^{\si}(X'^0)\subseteq Y^0$.

(d) We denote by $\C{K}:=\prod_{i\in I}\C{I}_{Y_i,Y}\subseteq\C{O}_Y$ the product of the sheaves of ideals of $Y_i$, let $\wt{Y}$ be the blow-up $\Bl_{\C{K}}(Y)$, and let $\pi:\wt{Y}\to Y$ be the projection map.

(e) By construction, $\pi$ is an isomorphism over $Y^0$, hence over $X'^0\times {}^{\si}(X'^0)$.
Therefore we have an open embedding $j:X'^0\times {}^{\si}(X'^0)\simeq \pi^{-1}(X'^0\times {}^{\si}(X'^0))\hra \wt{Y}$.

(f) For every  $J\subseteq I$, we set $Y_J:=X'_J\times {}^{\si}(X'_J)\subseteq Y$ and
$E_J:=\pi^{-1}(Y_J)\subseteq \wt{Y}$, denote by $i_J$ the inclusion
$E_J\hra\wt{Y}$, and by $\pi_J$ the projection $E_J\to Y_{J}$.

%(e) (IS IT NEEDED?) We also set $\p X'_J:=\cup_{i\in I\sm J}X'_{J\cup\{i\}}$, and
%$(X'_J)^0:=X'_J\sm \p X'_J$, $\p Y_J:=\cup_{i\in I\sm J}Y_{J\cup\{i\}}$ and
%$Y_J^0:=Y_J\sm \p Y_J$.  Explicitly, a point $y\in Y$ belongs to $Y_J^0$
% if and only if $y\in Y_j$ for every $j\in J$, and
%$y\notin Y_j$ for every $j\in I\sm J$.
%and $E_J^0:=p_J^{-1}(Y_J^0)$.
\end{Emp}

Following Lemma is an analogue of \cite[Lem 3.4]{Va} and the proof there goes over verbatim.

\begin{Lem} \label{L:smooth pink}
In the situation of \re{pink}, for every $J\subseteq I$ the closed subscheme
$E_J\subseteq\wt{Y}$ is smooth of dimension $2d-|J|$.
\end{Lem}
%\begin{proof}
%In the basic case
%\re{basic}, the assertion follows from the explicit description in \re{basic}(c) and the observation that $\wt{\B{A}}^2$ is smooth of dimension two. Since the assertion is
%local on $Y$, the general case follows from this and \re{loc}. Namely, it suffices to show that for every closed point $a\in Y$ there exists an open neighbourhood $U$ such that $\pi^{-1}(U)\cap E_J$ is smooth of dimension $2d-|J|$.

%Choose $J\subseteq I$ such that $a\in Y_J^0$. Then $a\in Y\sm (\cup_{i\in I\sm J}Y_i)$,
%thus it follows from \re{loc}(a), that we can replace $I$ by $J$, thus assuming that
%$a\in Y_I=X_I\times X_I$. In this case, by \re{loc}(c), there exists an open neighbourhood $U\subseteq Y$ of $a$ and a smooth morphism $U\to (\B{A}^2)^I$, which induces an isomorphism $\wt{Y}\times_Y U\to(\wt{\B{A}^2})^I\times_{(\B{A}^{2})^I}U$. Thus the assertion in general follows from the basic case.
%\end{proof}

\begin{Emp} \label{E:prpreim}
{\bf Notation.} For every morphism $c'=(c'_1,c'_2):C'\to Y=X\times X^{(q)}$, we denote by
$\wt{c}:\wt{C}\to\wt{Y}$ the {\em strict preimage} of $c'$. Explicitly,
$\wt{C}$ is the schematic closure of $c'^{-1}(Y^0)\times_Y
\wt{Y}$ in $C'\times_Y\wt{Y}$. Notice that $\wt{c}$ is a closed embedding (resp. finite),  if $c'$ is such.
%in which case we say that $\wt{c}(\wt{C})\subseteq\wt{Y}$ is a strict preimage of $c(C)\subseteq Y$.
\end{Emp}

\begin{Emp} \label{E:pinkfq}
{\bf Set-up.} In the situation of \re{pink}, assume that $k$ is an algebraically closed field of characteristic $p$, $q$ a power of $p$, and
$\si=\si_q$, the $q^{\mathrm{th}}$ Frobenius automorphism, $\si_q(x)=x^q$ for all $x\in k$.

As in the introduction, for a scheme $X$ over $k$, we write $X^{(q)}$ instead of ${}^{\si_q}X$, denote by $\on{F}_{X,q}=\on{F}_{X/k,q}:X\to X^{(q)}$ the geometric Frobenius morphism, and by $\Gm_{X,q}\subseteq X\times X^{(q)}$ the graph of $\on{F}_{X/k,q}$.
\end{Emp}

\begin{Emp} \label{E:ramification}
{\bf The ramification degree.}
(a) Recall that for a Noetherian scheme $X$ the sheaf of ideal $\C{I}_{X_{\red},X}\subseteq\C{O}_X$ is nilpotent. By a {\em thickness} of $X$,
we mean the smallest $r\in\B{N}$ such that $(\C{I}_{X_{\red},X})^r=0$.

(b) Let $f:Y\to X$ be a morphism of Noetherian schemes, and let $Z$ be a closed subset of $X$,  or, what is the same, a closed reduced subscheme.
We denote by  $\ram(f,Z)$ the thickness of the schematic preimage $f^{-1}(Z)\subseteq Y$. %and call it {\em the ramification of $f$ along $Z$}.

(c) For a morphism $f:Y\to X$ of Noetherian schemes we denote $\ram(f)\in\B{N}\cup\{\infty\}$ the supremum $\sup_{x}\ram(f,x)$, taken over all closed points $x\in X$. Actually, $\ram(f)$ is always finite (see \rl{ram}).
\end{Emp}

The following result is a slight generalization of \cite[Lem 3.7]{Va}.

\begin{Lem} \label{L:trans}
Let $c':C'\to X'\times X'^{(q)}$ be a correspondence such that the closed subset $\p X'\subseteq X'$ is locally $c'$-invariant and $q>\on{ram}(c'_2,\p X'^{(q)})$, let $\wt{\Gm}\subseteq \wt{Y}$ be the strict preimage of $\Gm:=\Gm_{X',q}$. Then  $\wt{c}(\wt{C})\cap\wt{\Gm}\subseteq\pi^{-1}(X'^0\times (X'^0)^{(q)})$.
\end{Lem}

\begin{proof}
 Note that the assertion is local on $X'$ in the \'etale topology. Moreover, as observed in the proof of \cite[Lemma 3.7]{Va} we may assume that $X'=\B{A}^n$ and $X'_i=Z(x_i)$. In this case, $X'$ and $X'_i$ are defined over $\fq$, thus the assertion is proven in \cite[Lemma 3.7]{Va}.
\end{proof}

The following Lemma is a natural generalization of \cite[Lemma 3.9]{Va} and its proof there carries over without any changes.

\begin{Lem} \label{L:grfrob}
In the situation of \re{pinkfq}, for every $J\subseteq I$, we set $\Gm_J:=\Gm_{X'_J,q}\subseteq X'_J\times (X'_J)^{(q)}$
and $\Gm^0_J:=\Gm_{X'_J,q}\subseteq X^0_J\times (X_J^0)^{(q)}$.

Then the schematic closure $\ov{\pi_J^{-1}(\Gm^0_J)}\subseteq E_J$ is smooth of dimension $d$, and the
schematic preimage $\pi_J^{-1}(\Gm_J)\subseteq E_J$ is a schematic
union of $\ov{\pi_{J'}^{-1}(\Gm_{J'})}$ with $J'\supset J$.
\end{Lem}

\section{The formula for the number of points in the intersection}\label{the_formula}

 In this section we will show a formula for the number of points in the intersection $\#(C^0\cap\Gamma^0)(k)$ appearing in
  Theorem \ref{T:main}.

\subsection{Preliminaries from intersection theory and $\ell$-adic cohomology}

All schemes are assumed to be of finite type over a fixed field $k$. An integral scheme will be called a variety.

\begin{Emp} \label{E:basic}
{\bf Notation.} Let $X$ be a scheme and $i\in \B{N}$.

(a) We denote by $Z_i(X)$ the group of $i$-cycles on $X$, that is, the free abelian group with generated by the symbols $[Z]$, where
$Z$ runs over all closed subvarieties of $X$ of dimension $i$.

(b) We denote by $A_i(X)$ the group of $i$-cycles on $X$ modulo rational equivalence (see \cite[1.3]{Fu}).
We denote by $A_*(X)$ (resp. $Z_*(X)$) the direct sum of the $A_i(X)$'s (resp. $Z_i(X)$'s).

(c) For a closed equidimensional subscheme $Z\subseteq X$ of dimension $i$, we denote by $[Z]_X$ or simply $[Z]$ its class in $Z_i(X)$ or $A_i(X)$ (see \cite[1.5]{Fu}).

(d) If $X$ is an equidimensional of dimension  $d$ with irreducible components $X_a$, the then $A_d(X)=Z_d(X)$ is a free abelian group with basis $\{[X_a]\}$.
\end{Emp}

\begin{Emp} \label{E:pushforward.}
{\bf Proper pushforward.}
(a) For a proper morphism $f:X\to Y$ of schemes, we denote by $f_*:A_i(X)\to A_i(Y)$ the induced morphism (see \cite[1.4]{Fu}). It is characterized by the property that for every closed subvariety $Z\subseteq X$, the pushforward  $p_*([Z])$ equals $\deg(p|_Z)[p(Z)]$, if the induced map $p|_Z:Z\to p(Z)$  is generically finite, and is zero otherwise.

(b) If $X$ is proper over $k$, then the projection $p_X: X\to \Spec k$ gives rise to the degree map $\deg:=(p_X)_*:A_0(X)\to A_0(\Spec k)=\B{Z}$.

(c) For every closed subscheme $X'\subseteq X$ we denote by $f_{X'}:X'\to f(X')$ the induced morphism
and denote the pushforward $(f|_{X'})_*:  A_i(X')\to A_i(f(X'))$ simply by $f_*$.
\end{Emp}

\begin{Emp} \label{E:pullback}
{\bf Flat pullback.}
(a) For a flat morphism $f:X\to Y$ of schemes of relative dimension $n$, we denote by $f^*:A_i(Y)\to A_{i+n}(X)$ the pullback map
(see \cite[Theorem 1.7]{Fu}). It is characterized by condition, that for every closed subvariety $Z\subseteq Y$ we have an equality
$f^*([Z])=[f^{-1}(Z)]$, where $f^{-1}(Z)$ denotes the schematic pullback.

(b) Flat pullbacks are compatible with proper pushforwards. In other words, for every Cartesian diagram
\[
\begin{CD}
X' @>g_X >>X\\
@Vf'VV  @VfVV\\
Y' @>g_Y>> Y,
\end{CD}
\]
where $f$ and $f'$ are proper and $g_X$ and $g_Y$ are flat of relative dimension $n$ we have an equality $(g_Y)^*\circ f_*=f'_*\circ (g_X)^*$ of morphisms
$A_i(X)\to A_{i+n}(Y')$.
\end{Emp}

\begin{Emp} \label{E:gysin}
{\bf (Refined) Gysin map.}
(a) Let $f:X\to Y$ be a morphism between smooth connected schemes of dimensions $d_X$ and $d_Y$, respectively, let
$Y'\subseteq Y$ be a closed subscheme, and $X':=f^{-1}(Y')\subseteq X$ its preimage. Then we have the refined Gysin map
$f^*:A_i(Y')\to A_{i+d_X-d_Y}(X')$ (see \cite[Definition 8.1.2]{Fu}, which Fulton sometimes denotes by $f^!$)\footnote{More generally, \cite[6.6]{Fu} defines the refined Gysin pullback when $f$ is an lci morphisms, that is, $f=g\circ i$, where $i$ is a closed regular embedding and $g$ is a smooth morphism.} for every $i$.

(b) The refined Gysin maps are compatible with compositions (by \cite[Propostion 8.1.1(a) and Corollary 8.1.3]{Fu}).

(c) Assume in addition that $f$ is flat. Then the induced map $f':X'\to Y'$ is flat as well, that the refined Gysin map $f^*$ coincide with
the pullback $f'^*$ from \re{pullback}(a) (by  \cite[Proposition 8.1.3(a)]{Fu}).
\end{Emp}

\begin{Emp} \label{E:ref inters}
{\bf (Refined) intersection pairing.} Let $X$ be a smooth connected scheme of dimension $d$.

(a) For a pair of closed subschemes $Y_1,Y_2\subseteq X$ we have an intersection pairing
\[
\cap: A_i(Y_1)\times A_j(Y_2)\to A_{i+j-d}(Y_1\cap Y_2)
\]
(compare \cite[Definition 8.1.1]{Fu} but note that our notation is different). Namely, let $\Dt:X\to X\times X$ be the diagonal map. By \re{gysin}, we then have the Gysin map
$\Dt^*:A_{i+j}(Y_1\times Y_2)\to A_{i+j-d}(Y_1\cap Y_2)$, and  the intersection pairing is characterized  by the condition
that for every closed subvarieties $Z_1\subseteq Y_1$ and $Z_2\subseteq Y_2$, we have a formula
$[Z_1]_{Y_1}\cap [Z_1]_{Y_2}=\Dt^*([Z_1\times Z_2])$.

(b) In the situation of (a) assume that $Y_1\cap Y_2$ is proper over $k$.  Then we have an intersection pairing
$\cdot:=\deg\circ\cap:A_i(Y_1)\times A_{d-i}(Y_2)\to\B{Z}$.

(c) In the situation of (b), let $Y'_1\subseteq Y_1$
and $Y'_2\subseteq Y_2$, and let $i_1:Y'_1\hra  Y_1$ and $i_2:Y'_2\hra  Y_2$ be the inclusion maps.
Then for every $\al_1\in A_i(Y'_1)$ and $\al_2\in A_{d-i}(Y'_2)$, we have an equality
$(i_1)_*(\al_1)\cdot (i_2)_*(\al_2)=\al_1\cdot\al_2$ (combine \cite[Proposition 8.1.3(a) and Theorem 6.2(a)]{Fu}).

(d) An open embedding $j:X^0\hra X$ induces open embeddings
$j_1:Y^0_1:=X^0\cap Y_1\hra Y_1$, $j_2:Y^0_2:=X^0\cap Y_2\hra Y_2$ and
$j_{12}:Y_1^0\cap Y_2^0\hra Y_1\cap Y_2$. Then it follows from \re{gysin}(c)
that for every $\al_1\in A_i(Y_1)$ and $\al_2\in A_j(Y_2)$ we have an equality $j_1^*(\al_1)\cap j_2^*(\al_2)=
j_{12}^*(\al_1\cap\al_2)$.

(e) In the situation of (d), assume that $Y_1\cap Y_2$ is proper over $k$ and $Y_1\cap Y_2\subseteq j(X^0)$. Then, by (d), we
have an equality $j_1^*(\al_1)\cdot j_2^*(\al_2)=\al_1\cdot\al_2$.

(f) Let $f:X'\to X$ be a proper map, where $X'$ is a smooth connected scheme of dimension $d'$, and let $Y'\subseteq X'$ and
$Y\subseteq X$ be closed subschemes such the intersection  $f(Y')\cap Y$ is proper over $k$.  Then the intersection $Y'\cap f^{-1}(Y)$ is proper
over $k$ as well, and for every $\al\in A_i(X')$ and $\beta\in A_{d-i}(Y')$ we have an equality $\al\cdot f^*(\beta)=f_*(\al)\cdot\beta$,
called the {\em projection formula} (see \cite[Example 8.1.7)]{Fu}).
\end{Emp}

\begin{Emp} \label{E:examples}
{\bf Example.}  Let $X$ be a smooth connected scheme of dimension $d$, and let $Y,Z\subseteq X$ be closed subvarieties such that
such that $Y\cap Z$ is proper over $k$ and $\dim(X)+\dim(Y)=d$. Then we can form the intersection $[Y]_Y\cdot [Z]_Z$ (see \re{ref inters}(b)).

(a) If $X$ is proper, then we have an equality $[Y]_Y\cdot [Z]_Z= [Y]_X\cdot [Z]_X$ (by \re{ref inters}(c)).

(b) If $Y\cap Z$ is finite, then we have an equality $[Y]_Y\cap [Z]_Z=\sum_{x\in Y\cap Z}l(\C{O}_{Y\cap Z,x})[x]$, hence
$[Y]_Y\cdot [Z]_Z=\sum_{x\in Y\cap Z}l(\C{O}_{Y\cap Z,x})$. Indeed, since $X\subseteq X\times X$ is locally given by $d$ equations,
the assumption that the intersection $Y\cap Z=X\cap(Y\times Z)\subseteq X\times X$ is finite implies that $Y\times Z$ is Cohen-Macaulay at every
$x\in Y\cap Z\subseteq Y\times Z$. Therefore the assertion follows from \cite[Proposition 8.2(b)]{Fu}.
\end{Emp}

%\begin{Emp} \label{E:expb}
%{\bf Example.} Let $f:X\to Y$ be a morphism between smooth connected schemes,
%and let $C\subseteq Y$ be a closed subscheme such that both
%inclusions $i:C\hra Y$ and $i':f^{-1}(C):=C\times_Y X\hra X$ are
%regular imbeddings of codimension $d$. Then
%$f^*([C])=[f^{-1}(C)]$.

%Indeed, $f$ is a composition $X\overset{(\Id,f)}{\lra}X\times Y\overset{\pr_2}{\lra}Y$.
%Since the assertion for $\pr_2$ is clear, and $(\Id,f)$ is a regular embedding,
%we may assume that $f$ is a regular embedding. In this case, the
%induced morphism $f^{-1}(C)\to C$ is a regular embedding as well,
%so the assertion follows, for example, from \cite[Thm 6.2 (a) and Rem
%6.2.1]{Fu}.
%\end{Emp}

From now on we assume that $k$ is an algebraically closed field, and $\ell$ is a prime,
different from the characteristic of $k$.

\begin{Emp} \label{E:end}
{\bf Morphism of the cohomology.} Let $X_1$ and $X_2$ be smooth connected
proper schemes over $k$ of dimension $d$.

(a) Let $c=(c_1,c_2):C\to X_1\times X_2$ is a correspondence,
where $C$ is a scheme of dimension $d$ as well.
Then $c$ gives rise to the morphism
\[
H^i(c):H^i(X_1,\ql)\overset{(c_1)^*}{\lra}H^i(C,\ql)\overset{(c_2)_*}{\lra}H^i(X_2,\ql),
\]
where  $(c_2)_*: H^i(C,\ql)\to H^i(X_2,\ql)$ is the pushforward map, corresponding by the Poincare duality to the pairing $ H^i(C,\ql)\times
H^{2d-i}(X_2,\ql(d))\to\ql$, defined by  $(x,y)\mapsto
\Tr_{C/k}(x\cup c_2^*(y))$.

(b) As is explained for example in \cite[4.5]{Va}, every class $\al\in A_d(X_1\times X_2)$ induces
a morphism $H^i(\al):H^i(X_1,\ql)\to H^i(X_2,\ql)$. Namely, it is characterised by the condition that the map $\al\mapsto H^i(\al)$
is a group homomorphism and that if $C\subseteq X_1\times X_2$ is a closed subvariety of dimension $d$, and $c:C\hra X_1\times X_2$ is the inclusion map, then the morphism $H^i([C])$ coincides with the map $H^i(c)$ from (a).

(c) A morphism $f:X_1\to X_2$ induces the pullback map $f^*:H^i(X_2,\ql)\to H^i(X_1,\ql)$.
Then for every $\al\in A_d(X_1\times X_2)$, one can form an endomorphism  $f^*\circ H^i(\al)$ of $H^i(X_1,\qlbar)$.
Let $\Gm_f\subseteq X_1\times X_2$ be the graph of $f$. Then (see, for example, the discussion in \cite[4.6]{Va}) we have an equality

\begin{equation} \label{Eq:traces}
\al\cdot[\Gm_f]=\sum_{i=0}^{2d}(-1)^i\Tr(f^*\circ H^i(\al),H^i(X_1,\ql)).
\end{equation}

\end{Emp}

\begin{Emp} \label{E:calculation}
{\bf A calculation.} In the situation of \re{end}, to $\al\in A_d(X_1\times X_2)$ we associate $\deg(p_1|_{\al})\in\B{Z}$ characterized by the condition $(p_1)_*(\al)= \deg(p_1|_{\al})[X_1]$ (see \re{basic}(d)). We claim that for every $\al\in A_d(X_1\times X_2)$ and a generically finite morphism $f:X_1\to X_2$ we have an equality
\[
\Tr(f^*\circ H^{2d}(\al), H^{2d}(X_1,\ql))=\deg(f)\deg(p_1|_{\al}).
\]

 Indeed, since  $\Tr(f^*\circ H^{2d}(\al))=\Tr(H^{2d}(\al)\circ f^*)$ and $H^{2d}(X_2,\ql)$ is one dimensional, we need to show that $H^{2d}(\al)\circ f^*=\deg(f)\deg(p_1|_{\al})\Id$.
By additivity, we can assume that $\al=[C]$, where $C\subseteq X_1\times X_2$ is a closed subvariety of dimension $d$. Let $c:C\hra X_1\times X_2$ is the inclusion map.

In this case, by definition of $(p_1)_*$, the degree $\deg(p_1|_{\al})$ equals $\deg(c_1)$, if $c_1$ is generically finite, and zero, otherwise.
Then combining \re{end}(b) and basic properties of the trace map, we see that the endomorphism $H^{2d}(\al)\circ f^*=(c_2)_*\circ (c_1)^*\circ f^*$ of $H^{2d}(X_1,\qlbar)$ equals $\deg(f\circ c_1)\Id=\deg(f)\deg(c_1)\Id$, if $c_1$ is generically finite, and zero, otherwise.
\end{Emp}

\begin{Emp} \label{E:purity}
{\bf Independence of $\ell$.}
Let $X$ be a smooth proper variety of over  $k$. It is known that for every $\ell\neq\on{char}(k)$ and $i\in\B{N}$, the cohomology $H^i(X,\ql)$ is a finite-dimensional $\ql$-vector space.  It follows from results of Katz--Messing \cite{KM} that for every $\al\in A_d(X\times X)$ the characteristic polynomial of the endomorphism $H^i(\al)$ of $H^i(X,\ql)$ has coefficients in $\B{Z}$ and is independent of $\ell\neq\on{char}(k)$.

Namely, in \cite{KM} (see also \cite[3.3.3]{An}) this is shown under an assumption that $k$ is an algebraic closure of a finite field (using purity \cite{De}), while the general case follows by a standard speading out argument (see \rcl{indep} and the ``general case'' in the proof of \rp{spectral}).
\end{Emp}

%(e)  For every morphism $f:X\to Y$ between smooth connected schemes, we
%have a pullback map $f^*:A^i(Y)\to A^i(X)$ (see \cite{Fu},~Section 8.1). Moreover, if $f$ is
%proper, then we have the equality $f^*(x)\cdot
%y=x\cdot f_*(y)$ for every $x\in A^i(Y)$ and $y\in A_i(X)$, called
%the projection formula (see \cite{Fu},~Proposition 8.3~(c)).

%\subsubsection{Example} \label{expb}
%Let $f:X\to Y$ be a morphism between smooth connected schemes,
%and let $C\subseteq Y$ be a closed subscheme such that both
%inclusions $i:C\hra Y$ and $i':f^{-1}(C):=C\times_Y X\hra X$ are
%regular imbeddings of codimension $d$. Then
%$f^*([C])=[f^{-1}(C)]$.
%
%Indeed, $f$ is a composition $X\overset{(\text{Id},f)}{\lra}X\times Y\overset{\text{pr}_2}{\lra}Y$.
%Since the assertion for $\text{pr}_2$ is clear, and $(\text{Id},f)$ is a regular embedding,
%we may assume that $f$ is a regular embedding. In this case, the
%induced morphism $f^{-1}(C)\to C$ is a regular embedding as well,
%so the assertion follows, for example, from \cite{Fu},~Theorem 6.2 (a) and Remark
%6.2.1.

%\begin{Emp'}{Purity} \label{purity}
%Let $X$ be a smooth projective variety of dimension $d$ over $\bar{\F_{q}}$, defined over $\F_q$.
%The vector spaces $H^i(X,\ql)=0$ vanish for $i>2d$ \cite[Expos\'e VI,~Section 2.2.3 (A)]{SGA5}. Moreover,~for every $i$ and
%every embedding $\iota \colon \B{Q}_{\ell}\hookrightarrow\B{C}$,~
%all eigenvalues $\alpha$ of $\on{F}_{X/\F_q}^* \colon H^i(X,\ql)\to H^i(X,\ql)$ satisfy $|\iota(\alpha)|=q^{i/2}$ \cite[Th\'eor\`eme 1.6]{De}.
%\end{Emp'}

\subsection{The formula}

The finiteness of the intersection is guaranteed by the following simple lemma.

\begin{Lem} \label{L:isolated}
Let $c=(c_1,c_2):C \to X \times X^{(q)}$ be a correspondence, let $\Gm\subseteq X \times X^{(q)}$ be the graph of $\on{F}_{X/k,q}$, and
let $y\in c^{-1}(\Gamma)(k)\subseteq C(k)$ be a point such that $c_2^{-1}(c_2(y))$ is finite and
$q>\ram(c_2,c_2(y))$.

Then $y\in c^{-1}(\Gamma)$ is an isolated point, and there exists an open neighborhood $C^0\subseteq C$ of $y$ such that
the closed subschemes $C^0\cap c^{-1}(\Gm)$ and $c_2^{-1}(c_2(y))\cap C^0$ of $C^0$ are
equal.

%local rings $\C{O}_{C \cap \Gamma,y}$ and $\C{O}_{c_2^{-1}(c_2(y)),y}$ are canonically isomorphic.
\end{Lem}

\begin{proof}
Set $x:=c_1(y)\in X(k)$. Since assertion is local, so we can assume that $X$ and $C$ are affine.
Note that $c_2(y)=x^{(q)}$ and consider ideals $I:=(\C{I}_{c_2^{-1}(x^{(q)}),C})_y\subseteq \C{O}_{C,y}$ and $J:=(\C{I}_{C \cap \Gamma,C})_y\subseteq \C{O}_{C,y}$, where $\C{O}_{C,y}$ denotes the local ring of $C$ at $y$. Since $y$ is an isolated point in $c_2^{-1}(x^{(q)})$, the radical $\sqrt{I}$ is the maximal ideal $\frak{m}_{C,y}\subseteq \C{O}_{C,y}$,
and is suffices to show the equality $I=J$.

Let $(f_1,\ldots,f_n)$ be generators of $\frak{m}_{X,x}\subseteq \C{O}_{X,x}$, let $(f^{(q)}_1,\ldots,f^{(q)}_n)$ be the corresponding generators of $\frak{m}_{X^{(q)},x^{(q)}}\subseteq \C{O}_{X^{(q)},x^{(q)}}$, and let $a_i:=c_1^{\cdot}(f_i)\in  \C{O}_{C,y}$ and $b_i:=c_2^{\cdot}(f^{(q)}_i)\in  \C{O}_{C,y}$ be their pullbacks. Then we have equalities $I=(b_1,\ldots,b_n)$ and $J=(b_1-a_1^q,\ldots,b_n-a_n^q)$.

Since $\ram(c_2,x^{(q)})<q$ and $a_i\in \frak{m}_{C,y}=\sqrt{I}$ for all $i$, we have $a_i^{q-1}\in I$, hence $a_i^{q}\in \frak{m}_{C,y}I\subseteq  I$. Therefore we have  $b_i-a_i^{q}\in I$ and $b_i=(b_i-a_i^q)+a_i^q\in J+\frak{m}_{C,y}I$ for all $i$. This implies inclusions
$J\subseteq I\subseteq J+\frak{m}_{C,y}I$, from which the equality $I=J$ follows from the Nakayama lemma.
\end{proof}

The following notion is not standard, but it is convenient for our purposes.

\begin{Emp} \label{E:uh-etale}
{\bf Uh-\'etale morphisms.} We call a morphism of schemes $f:X\to Y$ {\em uh-etale}, where {\em uh} stands for {\em universal homeomorphism},
if it decomposes as $X\overset{f_i}{\lra}{Z}\overset{f_s}{\lra}Y$, where $f_i$ is a flat universal homeomorphism and $f_s$ is \'etale.
Notice that if $f$ is uh-\'etale and $U\subseteq X$ is an open subscheme, then $f|_U:U\to X$ is uh-\'etale.
\end{Emp}

%The following lemma summarizes basic properties of this notion.

\begin{Lem} \label{L:uh-etale}
Let $f:X\to Y$ be a generically finite morphism of finite type between integral Noetherian schemes.

(a) There exists an open dense subscheme $U\subseteq X$ such that $f|_U:U\to Y$ is uh-etale.

(b) If $f$ is uh-\'etale, then we have an equality $l(\C{O}_{f^{-1}(f(x)),x})=\deg(f)_{\on{insep}}$ for every $x\in X$.
\end{Lem}

\begin{proof}
(a) The induced map on rational functions $f^*:\on{Rat}(Y)\hra\on{Rat}(X)$ decomposes
as a composition $\on{Rat}(Y)\hra L\hra\on{Rat}(X)$ such that $L/\on{Rat}(Y)$ is separable and $\on{Rat}(X)/L$ is purely inseparable.
Replacing $Y$ by an open dense subscheme and $X$ by its preimage, we can assume that $X$ any $Y$ are affine, while $f$ is a finite.

In this case, $f$ decomposes as $X\overset{f_i}{\lra}{Z}\overset{f_s}{\lra}Y$, where $Z$ is the affine scheme $\Spec(L\cap \Gm(X,\C{O}_X))$.
Then $\on{Rat}(Z)\simeq L$, morphisms $f_i$ and $f_s$ are finite, $f_i$ is generically a universal homeomorphism, while
$f_s$ is generically \'etale. Therefore there exists an open dense subscheme $V\subseteq Z$ such that $f_s|_V:V\to Y$ is \'etale and $f_i|_V:=f_i^{-1}(V)\to V$ is a flat universal homeomorphism (use \cite[Tag 0559]{Stacks}).
Then $U:=f_i^{-1}(V)$ satisfies the required property.

(b) Notice first that in the notation \re{uh-etale}, the scheme $Z$ is irreducible, because $X$ is irreducible and $f_i$ is a universal homeomorphism, and $Z$ is reduced, because $Y$ is reduced, while $f_s$ is \'etale. Thus $Z$ is integral, so we can talk about $\deg(f_i)$.
Then for every $x\in X$ we have an equality
\[
l(\C{O}_{f^{-1}(f(x)),x})=l(\C{O}_{f_i^{-1}(f_i(x)),x})=\deg(f_i)=\deg(f)_{\on{insep}}.
\]
Indeed, the first equality holds since $f_s$ is unramified, the second holds since $f_i$ is flat, while the third one is clear.
\end{proof}

\begin{Emp} \label{E:assumptions}
{\bf Assumptions.} Let $(k,q,X^0,C^0)$ be a quadruple as in \rt{main}(b), let
$X\subseteq\B{P}^n_k$ be the closure of $X^0$, let $C\subseteq X\times X^{(q)}$ be the closure of $C^0$, and let $c:C\hra X\times X^{(q)}$ be the inclusion map. Assume that

% $C^0$ and $X^0$ are irreducible of dimension $d$, and $c^0_1$ and $c^0_2$ are dominant (hence generically finite)
\begin{enumerate}[(a)]

\item $X^0$ is smooth over $k$;

\item morphism $c^0_2$ is uh-\'etale (see \re{uh-etale});

%$c^0_2:C^0 \to (X^0)^{(q)}$ factorizes as a $c^i_2:C^0\to C'$ and $c^{s}_2:C' \to (X^0)^{(q)}$, where $c^i_2$ is finite flat universal
%homeomorphism (of degree $\on{deg}(c_2^0)_{\on{insep}}$, and $c^{s}_2$ is finite \'etale;

\item the boundary $\partial X:=X\sm X^0$ is locally $c$-invariant (see \rd{inv}).

%\item Also assume that $q>\max\{\on{ram}(c^0_2),\on{ram}(c_2,\partial X^0)\}$ (see Definition \ref{ramification_morphism}).
\end{enumerate}
\end{Emp}

\begin{Emp} \label{E:construction}
{\bf Construction.}
(a) By \cite[Theorem 4.1]{dJ}, there exists an alteration $f:X'\to X$ such that $X'$ is connected and smooth over $k$ and
the reduced preimage $\partial X':=f^{-1}(\partial X)_{\on{red}}\subseteq X'$ is a strict normal crossing divisor, that is, a union of smooth divisors $\{X'_i\}_{i\in I}$ with normal crossings. From now on we fix such a $f$. In particular, the twisted Pink construction \re{pink} applies.

(b) Set $X'^0:=X'\sm\partial X'$, and let $f^0:X'^0\to X^0$ be the restriction of $f$.
We denote the projection $f\times f^{(q)}:Y=X'\times X'^{(q)}\to X\times X^{(q)}$ by $g$ and denote its restriction
$f^0\times (f^0)^{(q)}:X'^0\times (X'^0)^{(q)}\to X^0\times (X^0)^{(q)}$ by $g^0$.

(c) We set  $C'^0:=(g^0)^{-1}(C^0)\subseteq X'^0\times (X'^0)^{(q)}$, let  $C'\subseteq Y$ be the closure of $C'^0$, and let
$c'=(c'_1,c'_2):C'\hra Y=X'\times X'^{(q)}$ be the inclusion map. As in \re{prpreim}, let $\wt{C}\subseteq\wt{Y}$ be the strict preimage of $C'$. By construction, the open embedding $j:X'^0\times (X'^0)^{(q)}\hra\wt{Y}$ (see \re{pink}(e)) restricts to an open embedding
$j_C:C'^0\hra \wt{C}$.

(d) Since $g^0$ is a morphism between smooth varieties,
we can form the refined Gysin pullback $\al^0:=(g^0)^*([C^0]_{C^0}) \in A_d(C'^0)$ of $[C^0]_{C^0}\in  A_d(C^0)$ (see \re{basic}(c) and \re{gysin}(a)).

(e) Choose a cycle $\wt{\al}^0:=\sum_i n_i [V_i]\in Z_d(C'^0)$ representing $\al^0$, where each $V_i\subseteq C'^0$ is a closed irreducible subvariety. We denote by $\tilde{\al}_{\wt{C}}\in A_d(\tilde{C})$ the class of $\sum_i n_i [\bar{V}_i]\in Z_d(\tilde{C})$, where $\bar{V}_i\subseteq\tilde{C}$ is the closure of $j_C(V_i)$.

(f) We denote by $\tilde{\al}\in  A_d(\tilde{Y})$ the image of $\tilde{\al}_{\wt{C}}\in A_d(\tilde{C})$ with respect to the closed embedding
$\wt{C}\hra\wt{Y}$. For every $J\subseteq I$, recall that $i_J:E_J\hra \wt{Y}$ is a closed embedding of codimension $|J|$ between smooth schemes (see \rl{smooth pink}), and we set $\wt{\al}_{J}:=\pi_{J*}i_{J}^*(\wt{\al})\in A_{d-|J|}(Y_J)$. In particular, we have  $\wt{\al}_J=0$, if $Y_J=\emptyset$.
%We also set $\al:=\tilde{\al}_{\emptyset}\in A_{d}(Y)$.

(g) By \re{end}, every class $\wt{\al}_J\in A_{d-|J|}(X'_J\times (X'_J)^{(q)})$, gives rise to endomorphism
$\on{F}_{X'_J,q}^*\circ H^i(\wt{\al}_J)$ of $H^i(X'_J,\ql)$ for all $i=0,\ldots,2(d-|J|)$.
\end{Emp}

\begin{Emp} \label{E:rem construction}
{\bf Remark.} Notice that in the construction \re{construction} we made two choices. Namely, we choose an alteration $f:X'\to X$
in (a) and a cycle $\wt{\al}^0$ in (e).
%The choice of alterations is the only \textit{non-canonical} choice we make.
\end{Emp}

%Note that $X \times X^{(q)}$ is a smooth variety over $k$,~and hence there is a well defined intersection product on $A_*(X \times X^{(q)})$ (see \cite{Fu},~Section 8.1).

%\begin{Lem}\label{formula_for_intersection}
%In the given set-up as above,~the scheme $\Gamma_X \cap C$ is \'etale (and hence finite) over $k$.~Moreover $[C] \cdot [\Gamma_X]=\#\Gamma_X \cap C(k)$.
%\end{Lem}
%
%\begin{proof}
%Since $C/k$ is also smooth,~by \cite{Fu},~Proposition 7.1 (b) it suffices to show that $\Gamma_X \cap C$ is \'etale over $k$.
%\end{proof}

%\begin{Emp}
%{\bf Some cycle classes of interest}\label{cycle_classes_interest}

The goal of this section is to show the following formula

\begin{Prop} \label{P:basic formula}
In the situation of \re{assumptions} and notation \re{construction} assume that $q>\on{ram}(c'_2,\partial X'^0)$ and $q>\on{ram}(c^0_2)$ (see \re{ramification}). Set $\Gm^0:=\Gm_{X^0,q}$. Then the intersection $(C^0\cap\Gm^0)(k)$ is finite, and the difference
\[
\#(C^0 \cap \Gamma^0)(k)-\frac{\deg(c^0_1)}{\deg(c^0_2)_{\on{insep}}}q^{d}
\]
equals the product of $\frac{1}{\deg(f)\deg(c^0_2)_{\on{insep}}}$ and
\begin{equation} \label{Eq:main estimate}
\sum_{i=0}^{2d-1}(-1)^{i} \on{Tr}(F^*_{X',q} \circ H^{i}(\wt{\al}_{\emptyset}))+
\sum_{J \neq \emptyset}(-1)^{|J|}\left (\sum_{i=0}^{2d-|J|}(-1)^{i} \on{Tr}(F^*_{X'_J,q} \circ H^{i}(\wt{\al}_{J}))\right)
\end{equation}
\end{Prop}

%\begin{Emp}
%{\bf Remark.} the conclusion of the proposition would become true if we replace assumptions \re{assumptions}(a),(b) by a weaker assumptions
%that $X^0$ is smooth and for every $y\in C^0$ we have an equality $l(\C{O}_{(c_2^0)^{-1}(c_2^0(y)),y})=\deg(c_2^0)_{\on{insep}}$
%(compare \re{decomposition}(b),(c)).
%\end{Emp}

\subsection{Calculation of the intersection numbers}

Let $[\wt{\Gm}] \in A_d(\wt{Y})$ be the class of $\wt{\Gamma}\subseteq\wt{Y}$ (see \rl{trans}).  To show \rp{basic formula}, we are going to calculate the intersection number $\wt{\al}\cdot [\wt{\Gm}]$ in $\wt{Y}$ in two different ways.

The following result is an analogue of \cite[Lemma 4.9]{Va} in our set-up (which in its turn was motivated by \cite[Proposition IV.6]{Laf}).

\begin{Lem} \label{L:inters}
Let $[\Gm_J] \in A_{d-|J|}(Y_J)$ be the class of $\Gm_J \subseteq Y_J$ (see \rl{grfrob}).
Then we have an equality

\begin{equation} \label{Eq:classes}
\wt{\al}\cdot[\wt{\Gm}]=\sum_{J\subseteq I}(-1)^{|J|}\wt{\al}_J\cdot[\Gm_J].
\end{equation}
\end{Lem}

\begin{proof}
By the projection formula (see \re{ref inters}(f)), we have equalities $\wt{\al}_J\cdot [\Gamma_J]=\wt{\al} \cdot (i_{J*}\pi^*_{J}[\Gamma_J])$. Thus it suffices to show the identity $[\wt{\Gamma}]=\sum_{J\subseteq I} (-1)^{|J|}i_{J*}\pi^*_{J}[\Gamma_J]$.
To prove the result, we argue as in \cite[Lemma 4.9]{Va}, where an untwisted version of this result is deduced from an analogue of \rl{grfrob}.
\end{proof}

\begin{Lem}\label{L:proj formula}
In the situation of \re{assumptions} and notation of \re{construction} assume that $q>\on{ram}(c'_2,\partial X'^0)$ and $q>\on{ram}(c^0_2)$.  Then we have an equality
\begin{equation} \label{Eq:local}
\tilde{\al} \cdot [\tilde{\Gamma}]= \on{deg}(f)([C^0]_{C^0}\cdot [\Gm^0]_{\Gm^0})=
\on{deg}(f)\on{deg}(c_2^0)_{\on{insep}}\cdot \#(C^0\cap\Gm^0)(k).
\end{equation}
\end{Lem}

\begin{Emp}
{\bf Remark.}
Since $q>\on{ram}(c^0_2)$, the intersection $C^0\cap\Gm^0$ is finite (by \rl{isolated}). Therefore the intersection number
$[C^0]_{C^0}\cdot [\Gm^0]_{\Gm^0}$ is defined (see \re{ref inters}(b)).
\end{Emp}

\begin{proof}
We set $\Gm'^0:=\Gm_{X'^0,q}$. We claim that we have a sequence of equalities
\[
\tilde{\al} \cdot [\tilde{\Gamma}]\overset{(1)}{=}\tilde{\al}_{\wt{C}} \cdot [\tilde{\Gamma}]_{\wt{\Gm}}\overset{(2)}{=}\al^0 \cdot [\Gamma'^0]_{\Gm'^0}
\overset{(3)}{=}(g^0)^*([C^0]_{C^0}) \cdot [\Gamma'^0]_{\Gm'^0}\overset{(4)}{=}\]\[\overset{(4)}{=}[C^0]_{C^0}\cdot (g^0)_*([\Gamma'^0]_{\Gm'^0})\overset{(5)}{=} \deg(f)([C^0]_{C^0}\cdot [\Gamma^0]_{\Gm^0}).
\]
Indeed,

$\bullet$ equality $(1)$ follows from \re{examples}(a);

$\bullet$ since $q>\on{ram}(c'_2,\partial X'^0)$, we have an inclusion $\tilde{C} \cap \tilde{\Gamma} \subseteq \pi^{-1}(X'^0\times X'^{(q)})$ (by \rl{trans}),  hence $\tilde{C} \cap \tilde{\Gamma} \subseteq j(X'^0\times X'^{(q)})$
(see \re{pink}(e)). So equality $(2)$ follows from observation \re{ref inters}(e) and equality $j_C^*(\tilde{\al}_{\wt{C}})=\al^0$
(use \re{construction}(c),(e));

$\bullet$ equality $(3)$ follows from equality $\al^0= (g^0)^*([C^0]_{C^0})$ (see \re{construction}(d));

$\bullet$  equality $(4)$ follows from the projection formula (see \re{ref inters}(f));

$\bullet$  equality $(5)$ would follow, if we show that
$(g^0)_*([\Gamma'^0]_{\Gm'^0})=\deg(f)[\Gamma^0]_{\Gm^0}$. To see it, notice that the restriction $g^0|_{\Gm'^0}:\Gm'^0\to\Gm^0$ is isomorphic to the projection $f^0:X'^0\to X^0$. In particular, it is generically finite of degree $\deg(f)$.

This shows the left equality of \form{local}. To see the right one, it suffices to show that
\begin{equation} \label{Eq:int}
[C^0]_{C^0}\cdot [\Gm^0]_{\Gm^0}=\sum_{y\in (C^0\cap\Gm^0)(k)}l(\C{O}_{C^0\cap\Gm^0,y})=\on{deg}(c_2^0)_{\on{insep}}\cdot \#(C^0\cap\Gm^0)(k).
\end{equation}
The first equality of \form{int} follows from \re{examples}(b), while for the second one it suffices to show that
\begin{equation} \label{Eq:int1}
l(\C{O}_{C^0\cap\Gm^0,y})=l(\C{O}_{(c^0_2)^{-1}(c^0_2(y)),y})=\on{deg}(c_2^0)_{\on{insep}}.
\end{equation}
Finally, the first equality in \form{int1} follows from \rl{isolated}, while the second one follows from \rl{uh-etale}(b).
\end{proof}

\begin{Lem}  \label{L:degree proj}
Let $\wt{\al}_{\emptyset}=\pi_*(\wt{\al})\in A_d(Y)$ be the class defined in \re{construction}(f). Then
we have an equality $g_*(\wt{\al}_{\emptyset})=\deg(f)^2[C]$ in $A_d(X\times X^{(q)})$.
\end{Lem}

\begin{proof}
We have to show that $(g\circ\pi)_*(\wt{\al})=\deg(f)^2[C]$. Denote by $\wt{g}_C:\wt{C}\to C$ be the restriction of $g\circ\pi$. It suffices to show the equality $(\wt{g}_C)_*([\wt{\al}]_{\wt{C}})=\deg(f)^2 [C]_C$ in $A_d(C)=Z_d(C)$. Therefore it suffices to show the equality of pullbacks to any open non-empty subscheme $V$ of $C$.

Thus, taking pullback to $C^0$, using \re{pullback}(b) and equality $j_C^*(\wt{\al})=(g^0)^*([C^0]_{C^0})$, it suffices to show that $(g^0)_*(g^0)^*([C^0]_{C^0})=\deg(f)^2 [C^0]_{C^0}$.

Let $U\subseteq X^0$ be an open dense subscheme such that $f_U:f^{-1}(U)\to U$ is finite flat. Since $c^0_1$ and $c^0_2$ are dominant, the open subscheme $C_U:=C^0\cap (U\times U^{(q)})$ of $C^0$ is non-empty. Set $g_U:=f_U\times (f_U)^{(q)}$. Taking pullback to $C_U$, using \re{pullback}(b) and \re{gysin}(b),(c), it thus suffices to show the identity $(g_U)_*(g_U)^*([C_U]_{C_U})=\deg(f)^2 [C_U]_{C_U}$. But this is clear, because $(g_U)^*([C_U]_{C_U})=[g^{-1}(C_U)]_{g^{-1}(C_U)}$ (use \re{gysin}(b)), while the map $(g^0)^{-1}(C_U)\to C_U$ is finite flat of degree $\deg(f)^2$.
\end{proof}

\begin{Cor} \label{C:top cohom}
In the situation of \re{assumptions} and notation of \re{construction}, we have an equality
\[
\on{Tr}(F^*_{X',q} \circ H^{2d}(\wt{\al}_{\emptyset}))=\deg(f)\deg(c_1)q^{d}.
\]
\end{Cor}

\begin{proof}
By \re{calculation}, we have to show that
\[
\deg(F^*_{X',q})\deg(p_1|_{\wt{\al}_{\emptyset}})=\deg(f)\deg(c_1)q^{d}.
\]
Since $\deg(F_{X',q})=q^{\dim X'}=q^n$, we thus have to show that $\deg(p_1|_{\wt{\al}_{\emptyset}})=\deg(f)\deg(c_1)$, that is, to show the equality $(p_1)_*(\wt{\al}_{\emptyset})= \deg(f)\deg(c_1)[X']$ in  $A_d(X')$.

Since $f_*([X'])=\deg(f)[X]$, it thus suffices to show an equality
\[
f_*(p_1)_*(\wt{\al}_{\emptyset})= (p_1)_*g_*(\wt{\al}_{\emptyset})=\deg(f)^2\deg(c_1)[X]
\]
in $A_d(X)$. Since $f\circ p_1= p_1\circ g:X'\times X'^{(q)}\to X$, the last equality follows from \rl{degree proj} and identity
$(p_1)_*[C]=\deg(c_1)[X]$.
\end{proof}

Now we are ready to show \rp{basic formula}.

\begin{Emp}
\begin{proof}[Proof of \rp{basic formula}]
Combining Lemmas \ref{L:inters} and \ref{L:proj formula}, we have an equality
\[
\on{deg}(f)\on{deg}(c_2^0)_{\on{insep}}\cdot \#(C^0\cap\Gm^0)(k)= \sum_{J\subseteq I}(-1)^{|J|}\wt{\al}_J\cdot[\Gm_J].
\]
Therefore the assertion follows from a combination of equalities
\[
\wt{\al}_J\cdot[\Gm_{J}]=\sum_{i=0}^{2(d-|J|)}(-1)^i\Tr(\on{F}_{X'_J,q}^*\circ H^i(\wt{\al}_J), H^i(X'_J,\ql))
\]
(see \form{traces}) and  \rco{top cohom}.
\end{proof}
\end{Emp}

\section{Proof of the main theorem}

\subsection{Review of results from Appendices}
In this subsection we will review results from Appendices, we need for the proof of \rt{main}.
%First we list ``standard boundness results" we need for the proof. For completeness, we provide proofs in the Appendix (see Lemma \ref{L:gdeg} for assertions (a),(b) and \re{pfapp} for the rest).

\begin{Emp} \label{E:app}
{\bf Standard boundness results} (see Lemma \ref{L:gdeg} for assertions (a),(b) and section \re{pfapp} for the remaining assertions).
Let $X\subseteq\B{P}_k^n$, $X'\subseteq \B{P}^{m}_X$, and $Z\subseteq X$ be reduced  subschemes, and let
$f:X'\to X$ be the composition $X'\hra\B{P}^{m}_X\surj X$.

(a) We have an equality $\gdeg(X^{(q)})=\gdeg(X)$.

(b) Let $\ov{X}\subseteq\B{P}^n_k$ be the closure of $X$, and let $X_1,\ldots,X_r$ be all irreducible components of $X$. Then we have inequalities $\gdeg(\ov{X})\leq\gdeg(X)$, $r\leq\gdeg(X)$ and $\gdeg(X_i)\leq\gdeg(X)$ for all $i$.

(c) Let $Y\subseteq \B{P}_k^n$ be a reduced subscheme. Then $\gdeg((X\cap Y)_{\red})$  is bounded in terms of $n,\gdeg(X)$ and $\gdeg(Y)$.

(d) Assume that $Z\subseteq X$ either open or closed. Then $\gdeg((X\sm Z)_{\red})$  is bounded in terms of $n,\gdeg(X)$ and $\gdeg(Z)$.

(e) The ramification  $\ram(f,Z)$ and the geometric degree of $f^{-1}(Z)_{\red}\subseteq X'$ are bounded in terms of $n,m,\gdeg(X),\gdeg(X')$ and $\gdeg(Z)$.

(f) The geometric degree of $\ov{f(X')}\subseteq X$ is bounded in terms of $n,m,\gdeg(X)$ and $\gdeg(X')$.

(g) Assume that $X$ and $X'$ are irreducible and $f$ is generically finite. Then $\deg(f)$ and $\ram(f)$ are bounded in terms of $n,m,\gdeg(X)$ and $\gdeg(X')$.

(h) Let $U\subseteq X$ be the largest open subscheme such that $U$ is smooth over $k$. Then $\gdeg(U)$  is bounded in terms of $n,\gdeg(X)$.

(i) In the situation of (g), there exists an open dense subscheme $U\subseteq X'$ such
$f|_U$ is uh-\'etale (see \re{uh-etale}) and $\gdeg(U)$  is bounded in terms of $n,m,\gdeg(X)$ and $\gdeg(X')$.

(j) Assume that $Z\subseteq X$ is closed. Then there exists a closed embedding $\Bl_Z(X)\hra\B{P}^{n'}_X$ over $X$ such that
$n'$ and geometric degree of $\Bl_Z(X)\subseteq\B{P}^{n'}_X$ are bounded in terms of $n,\gdeg(X)$ and $\gdeg(Z)$.

(k) Assume that $X$ is irreducible and $Z\subseteq X$ is closed. Then there exists an alteration $\wt{f}:\wt{X}\to X$
and a closed embedding $\wt{X}\hra\B{P}^{n'}_X$ over $X$ such that $\wt{X}$ is smooth over $k$, $\wt{f}^{-1}(Z)_{\red}\subseteq\wt{X}$
is a union of smooth divisors with normal crossings such that $n'$ and $\gdeg(\wt{X})$ are bounded in terms of $n,\gdeg(X)$ and $\gdeg(Z)$.
\end{Emp}

\begin{Emp} \label{E:norm}
{\bf Norm of a cycle class.} Suppose that we are in the situation of \re{app}.

(a) For every $\al\in A_*(X)$, we denote by $|\al|$ the minimum
of $\sum_j |n_j|\gdeg(Z_j)$ taken over all representatives $\al'=\sum_j n_j [Z_j]\in Z_*(X)$ of $\al$, where  $Z_j\subseteq X$ are irreducible subvarieties and  $n_j\in\B{Z}$.

(b) Assume that $X\subseteq\B{P}^n_k$ and $X'\subseteq \B{P}^{m}_X$ are closed. Then $f$ is proper, and  it follows from \rl{pullback}(a) that for every $\al\in A_*(X')$ we have an inequality $|f_*(\al)|\leq|\al|$.

(c) Assume that $X$ and $X'$ are smooth and connected. Then we have a refined Gysin pullback
map $f^*:A_*(Z)\to A_{*}(f^{-1}(Z))$ (see \re{gysin}). Moreover, it follows from \rco{bound gysin} that there exists
a constant $M$ depending on $n,m,\gdeg(X),\gdeg(X')$ and $\gdeg(Z)$ such that for every class $\al\in A_*(Z)$, we have an inequality
$|f^*(\al)|\leq M|\al|$.
\end{Emp}

\begin{Emp} \label{E:spectral}
{\bf Main estimate.}
Let $k$ be an algebraically closed field of characteristic $p>0$, $q$ a power of $p$, and $X\subseteq\B{P}^n_k$ a smooth closed subvariety of dimension $d$. By \rco{bound trace}, there exists a constant $M$ depending on $n$ and $\deg(X)$ such that for every $\al\in A_{d}(X \times X^{(q)})$ and $i\in\B{N}$ the trace $\Tr(F^*_{X,q} \circ H^{i}(\al))\in\B{Z}$ (see \re{purity}) satisfies
\[
|\Tr(F^*_{X,q} \circ H^{i}(\al))|\leq M|\al| q^{i/2}.
\]
\end{Emp}

%\begin{Emp} \label{E:main obs}
%{\bf Main observation.} In each step of the proof we are going to use \rco{bounded} asserting that a collection  $\{X_{\al}\subseteq \B{P}^n_{k_{\al}}\}_{\al\in A}$ of reduced subschemes is bounded in the sense of \re{???} if and only the their geometric degrees  $\gdeg(X_{\al})\}_{\al\in A}$ are bounded.
%\end{Emp}

\subsection{Main reduction}

In this subsection we are going to reduce Theorem \ref{T:main} to a particular case of Theorem \ref{T:main}(b).

%(for all possible choices of $C$ and $X$) implies part $(1)$.~Subsequently our goal will be to reduce the proof of Theorem \ref{T:main} to the case where $X$ is smooth over $k$,~and the correspondence $c$ has a locally invariant compactification (Corollary \ref{Final_reduction}).

%\begin{emp}
%{\bf Convention.}
%Suppose we are in the situation of Theorem \ref{T:main},~in particular we fix integers $n,d_1$ and $d_2$.~In this section a collection of reduced quasi-projective schemes (resp. numbers) are said to be $(n,d_1,d_2)$-bounded if the geometric degree and the dimension of embedding of the quasi-projective schemes (or the numbers) can be bounded in terms of $n,d_1$ and $d_2$.~In particular if a collection of quasi-projective schemes is $(n,d_1,d_2)$-bounded then there exists a finite type scheme over $\B{Z}$ parametrising them (see Corollary \ref{C:bounded gdeg}),~and thus is \textit{bounded} is the sense of Appendix \ref{Uniform_Projective} (see Notation \ref{N:bounded}(a)).
%\end{emp}

\begin{Prop}\label{P:reduction}
To prove Theorem \ref{T:main} it suffices to show that the ``the conclusion of Theorem \ref{T:main}(b) holds for quadruples satisfying assumptions  \re{assumptions}".

More precisely, in order to prove \rt{main} it suffices to show that for every $n,\dt\in\B{N}$ there exists a constant $M$ such that the
conclusion of Theorem \ref{T:main}(b) holds for all quadruples  $(k,q,X^0,C^0)$ satisfying assumptions \re{assumptions}.

%such that in addition satisfy
%\begin{enumerate}[(a)]

%\item $C$ and $X$ are irreducible,

%\item $c_1$ and $c_2$ are dominant,

%\item $X$ is smooth,

%\item $c_2:C \to X^{(q)}$ factorizes as a composition of $c^i_2:C \to C'$ and $c^{s}_2:C' \to X^{(q)}$,~where $c^i_2$ is a finite flat universal homeomorphism and $c^{s}_2$ is finite \'etale.
%\end{enumerate}
\end{Prop}

\begin{proof}
For convenience of the reader, we divide the proof into steps. First we are going to show that Theorem \ref{T:main}(a) follows from Theorem \ref{T:main}(b).

\begin{Emp} \label{E:step1}
Note that it suffices to show that for every fixed $d\in\B{N}$ the conclusion of Theorem \ref{T:main}(a) holds for quadruples such that $\on{\dim}(C^0)=d$. Moreover, by induction, we can assume that the conclusion of Theorem \ref{T:main}(a) holds for quadruples such that $\on{\dim}(C^0)<d$.
\end{Emp}

\begin{Emp} \label{E:step2}
We claim that it suffices to show Theorem \ref{T:main}(a) for quadruples such that $C^0$ is irreducible.

\vskip 4truept

\noindent Indeed, let $(k,q,X^0,C^0)$ be a quadruple as in Theorem \ref{T:main}(a), and let
$C_1,\ldots,C_r$ be all irreducible components of $C^0$. Since $r\leq\on{gdeg}(C)$, $\on{gdeg}(C_i)\leq\on{gdeg}(C^0)$ (see
\re{app}(b)), $\on{dim}(C_i)\leq\on{dim}(C^0)$ and $(C^0\cap \Gamma^0)=\cup_i(C_i \cap \Gamma^0)$, we see that
if each quadruple $(k,q,X^0,C_i)$ satisfies inequality (\ref{Main1}) with constant $M$, then quadruple
$(k,q,X^0,C^0)$ satisfies inequality (\ref{Main1}) with constant $\dt M$.
\end{Emp}

\begin{Emp} \label{E:step3}
We claim that it suffices to show the conclusion of Theorem \ref{T:main}(a) for quadruples satisfying assumptions of  Theorem \ref{T:main}(b).

\vskip 4truept

\noindent Indeed, by \re{step1} and \re{step2}, it suffices to show the conclusion of Theorem \ref{T:main}(a) for quadruple  $(k,q,X^0,C^0)$ such that $C^0$ is irreducible and $\on{dim}(C^0)=d$. Consider closed reduced subschemes $Z:=\bar{c_1(C)}\subseteq X$ and $C_Z:=c_2^{-1}(Z^{(q)})_{\red}\subseteq C^0$. Then $C_Z\subseteq Z\times Z^{(q)}$ is a closed reduced subscheme, and we have an equality
\begin{equation} \label{reduction step}
C^0\cap\Gamma_{X^0,q}=C_Z\cap\Gamma_{Z,q}.
\end{equation}
Moreover, it follows from a combination of \re{app}(a),(e),(f) that geometric degrees of $Z$ and $C_Z$ are bounded in terms of $n$ and $\dt$.

If $C_Z\subsetneq C^0$, then $\dim(C_Z)<d$, thus the assertion follows from equality (\ref{reduction step}) and induction hypothesis (see \re{step1}). Therefore we can assume that $C_Z=C^0$, and morphism $c_1:C^0\to Z$ is dominant. Thus quadruple $(k,q,Z,C^0)$ satisfies all the assumptions of Theorem \ref{T:main}(b), thus the assertion follows from equality (\ref{reduction step}).
% thus $Z$ is irreducible of dimension $\leq d$, and $c_2:C^0\to Z^{(q)}$ is quasi-finite
%(because $c_2$ is such), hence $(c_0)_2$ is dominant. Thus quadruple $(k,q,X_0,C_0)$ satisfies all the assumptions of Theorem \ref{T:main}(b), thus the assertion follows from equality (\ref{reduction step}).

\end{Emp}

\begin{Emp} \label{E:step4}
Now we claim that \rt{main}(b) implies \rt{main}(a).

\vskip 4truept

\noindent Indeed, by \re{step3}, it suffices to show the conclusion of Theorem \ref{T:main}(a) for quadruples $(k,q,X^0,C^0)$ from  Theorem \ref{T:main}(b). Since $\deg(c_1^0)$ is bounded in terms of $n$ and $\dt$ (see \re{app}(g)), the assertion follows.
\end{Emp}

Now we are going to show that it suffices to show the conclusion of Theorem \ref{T:main}(b) for quadruples satisfying assumptions  \re{assumptions}(a)-(c).

\begin{Emp} \label{E:step5}
First we claim that suffices to show the conclusion of Theorem \ref{T:main}(b) for quadruples $(k,q,X^0,C^0)$ such that $X^0$ is smooth.

\vskip 4truept

\noindent Indeed, let $(k,q,X^0,C^0)$ be a quadruple as in Theorem \ref{T:main}(b), let $U\subseteq X^0$ be the
largest smooth open dense subscheme, and set
\[
Z:=(X^0\sm U)_{\on{red}}, C_U:=C\cap (U\times U^{(q)})\text{ and }
C_Z:=(C^0\cap (Z\times Z^{(q)}))_{\on{red}}.
\]
%Next, we let $C_U^0\subseteq C_U$ be the largest smooth open subscheme, and set $\p C_U:=C_U\sm C_U^0$.

Then it follows from a combination of \re{app}(a),(c),(d),(e),(h) that geometric degrees of $U, Z, C_U$ and  $C_Z$  are bounded in terms of $n$ and $\dt$. Since $\on{dim}(C_Z)<\dim(C^0)$, the inequality (\ref{Main1}) for quadruple $(k,q,Z,C_Z)$ holds by the induction hypothesis (see \re{step1}). Finally, since  $C^0\cap \Gamma_{X^0,q}$ decomposes as a union of $C_U\cap \Gamma_{U,q}$ and $C_Z \cap \Gamma_{Z,q}$, the inequality (\ref{Main2}) for $(k,q,X^0,C^0)$ follows from that for  $(k,q,U,C_U)$.
\end{Emp}

\begin{Emp} \label{E:step6}
Next we claim that suffices to show the conclusion of Theorem \ref{T:main}(b) for quadruples $(k,q,X^0,C^0)$ satisfying the assumptions \re{assumptions}(a),(b).

\vskip 4truept

\noindent Indeed, let $(k,q,X^0,C^0)$ be a quadruple satisfying \re{assumptions}(a). Then by \re{app}(i) there exists an open dense subscheme
$V\subseteq C^0$ subset such that $c_2^0|_{V}$ is uh-\'etale (see \re{uh-etale}) and $\gdeg(V)$ is bounded in terms of $n$ and $\dt$.
We set $U:=X^0\sm\ov{c^0_1(C^0\sm V)}\subseteq X^0$ and $C_U:=C^0\cap (U\times U^{(q)})$. Then $\dim(C^0\sm V)<\dim(C^0)=\dim(X^0)$, thus $U\subseteq X^0$ is open dense. By construction, we have inclusions $C_U\subseteq c_1^{-1}(U)\subseteq V$, thus $c_2^0|_{C_U}$ is uh-\'etale.
Then arguing as in \re{step5}, we can replace $(k,q,X^0,C^0)$ by $(k,q,U,C_U)$, thus assuming that $c_2^0$ is uh-\'etale.
\end{Emp}

\begin{Emp} \label{E:step7}
Finally, we claim that it suffices to show the conclusion of Theorem \ref{T:main}(b) for quadruples $(k,q,X^0,C^0)$ satisfying the assumptions \re{assumptions}(a)-(c).

\vskip 4truept

\noindent Indeed, by \re{step6}, it suffices to show the conclusion of Theorem \ref{T:main}(b) for quadruples satisfying the assumptions \re{assumptions}(a),(b). Let $(k,q,X^0,C^0)$ be any such quadruple, and let $V\subseteq X^0$  and  $\pi:\wt{X}\to X$ be the open subscheme and the blow-up constructed in Proposition \ref{P:blowup}. The proof of the following assertion will be given in the Appendix (see \re{pfclapp}).

\begin{Cl} \label{C:app}
(a) The geometric degree of $V\subseteq X^0$ is bounded in terms of $n$ and $\dt$.

(b) There exist $n'\in\B{N}$ and a closed embedding $\wt{X}\hra\B{P}^{n'}_X\subseteq\B{P}^{nn'+n+n'}_k$ over $X$ such that $n'$ and $\gdeg(\wt{X})$ are bounded in terms of  $n$ and $\dt$.
\end{Cl}

Set $C_V:=C\cap (V\times V^{(q)})$. Using \rcl{app}(a) and arguing as in \re{step5} one sees that the conclusion of \rt{main}(b) for $(k,q,X^0,C^0)$ follows from that for $(k,q,V,C_V)$. Next, since $\pi|_V:\pi^{-1}(V)\to V$ is an isomorphism, we have an open embedding $V\cong \pi^{-1}(V)\hra \wt{X}$ with dense image, which induces the embedding $C_V\hra V\times V^{(q)}\hra \wt{X}\times \wt{X}^{(q)}$.
Then combining \rcl{app} and \re{app}(a),(c),(e) we see that the geometric degrees of $V\subseteq \wt{X}$ and $C_V\subseteq\wt{X}\times \wt{X}^{(q)}$ are bounded in terms of $n$ and $\dt$. In other words, the conclusion of \rt{main}(b) for $(k,q,V\subseteq\B{P}^n_k,C_V)$ follows from that for $(k,q,V\subseteq\B{P}^{nn'+n+n'}_k,C_V)$.

Now the reduction is immediate. Indeed, let $\wt{C}$ be the closure of the image of the embedding $C_V\hra V\times V^{(q)}\hra \wt{X}\times \wt{X}^{(q)}$, and let $\wt{c}:\wt{C}\hra\wt{X}\times \wt{X}^{(q)}$ be the inclusion map. By \rp{blowup}, the closed subset $\wt{X}\sm V\subseteq \wt{X}$ is locally $\wt{c}$-invariant, hence the proof of \rp{reduction} is complete.
\end{Emp}
%This completes the proof of \rp{reduction}.
\end{proof}

\subsection{Completion of the proof}

\begin{Emp}
{\bf Reformulation of the problem.}
It follows from a combination \rp{reduction} and \rp{basic formula} that it suffice to show that for every $n,\dt\in\B{N}$ there exists
a constant $M$ satisfying the following property:

For every quadruple satisfying assumptions \re{assumptions}, there exists an alteration $f:X'\to X$ as in \re{construction}(a) and a cycle $\wt{\al}^0\in A_d(C'^0)$ as in \re{construction}(e) such that
\begin{equation} \label{Eq:first}
\on{ram}(c'_2,\partial X'^0)\leq M, \on{ram}(c^0_2)\leq M, |I|\leq M;
\end{equation}
and for all $J\subseteq I$ and $i\in\B{N}$, the trace $\on{Tr}(F^*_{X'_J,q} \circ H^{i}(\wt{\al}_{J}))\in\B{Z}$ satisfies
\begin{equation} \label{Eq:second}
\left|\on{Tr}(F^*_{X'_J,q} \circ H^{i}(\wt{\al}_{J}))\right|\leq Mq^{i/2}.
\end{equation}
\end{Emp}

%To show the assertion we are going combine the results, shown in the Appendixes B and C, the main of which is \rco{spectral}.

\begin{Emp} \label{E:Step 1}
Using \re{app}(k), there exist $n',\dt'\in\B{N}$ such that for every quadruple satisfying assumptions \re{assumptions}, there exists an alteration $f:X'\to X$ as in \re{construction}(c) and a closed embedding $X'\hra \B{P}^{n'}_{X}$ over $X$ such that
the (geometric) degree of $X'$ is at most $\dt'$.
\end{Emp}

From now on we assume that we are given a quadruple satisfying assumptions \re{assumptions}, and a pair $(n',\dt')$ and a closed embedding $X'\hra \B{P}^{n'}_{X}$ as in \re{Step 1}.

\begin{Emp} \label{E:Step 2}
(a) By \re{app}(d),(e), the geometric degree of $\p X'\subseteq X'$ is bounded in terms of $n'$ and $\dt'$, hence in terms of $n$ and $\dt$.
Then, by \re{app}(b),(c) the cardinality $|I|$ and geometric degrees of $X'_i$ and $X'_J$ are bounded in terms of $n$ and $\dt$ as well.

(b) Using \re{app}(b),(e), we see that the geometric degree of $C'=\ov{C'^0}\subseteq X'\times X'^{(q)}$ is bounded in terms of $n$ and $\dt$. Hence combining (a) with \re{app}(e),(g) we see that $\on{ram}(c^0_2)$ and $\on{ram}(c'_2,\partial X'^0)$ are bounded in terms of $n$ and $\dt$. This proves existence of $M$ satisfying \form{first}.

(c) By a combination of (a) and \re{app}(j) there exists a pair $n'',\dt''\in\B{N}$ and
a closed embedding $\wt{Y}\hra \B{P}^{n''}_Y$ over $Y$ such that $\gdeg(\wt{Y})\leq \dt''$.

(d) Combining (c) with \re{app}(e), the geometric degrees of $X'^0\times (X'^0)^{(q)}\subseteq\wt{Y}$ and
$E_J\subseteq\wt{Y}$ are bounded in terms of $n$ and $\dt$.
\end{Emp}

From now on we assume that we have chosen pair $(n'',\dt'')$ and a closed embedding $\wt{Y}\hra \B{P}^{n''}_{Y}$ satisfying \re{Step 2}(c).

\begin{Emp} \label{E:Step 3}
(a) Combining \re{Step 2}(d) and \re{norm}(c), we see that the norm $|\al^0|$ of the class $\al^0\in A_d(C'^0)$ (see \re{construction}(d)) is bounded in terms of $n$ and $d$.

(b) Choose a cycle $\wt{\al}^0=\sum_i n_i [V_i]_{C'^0}$ representing $\al^0$ such that
$|\al^0|=\sum_i|n_i| \on{gdeg}(V_i)$ (see \re{norm}(a) and \re{construction}(d)). We claim that this cycle satisfies the required properties.

(c) Using \re{app}(b), the cycle class
$\wt{\al}=\sum_i n_i [\ov{V_i}]\in A_d(\wt{Y})$ (see \re{construction}(f)) satisfies
\[
|\wt{\al}|\leq \sum_i |n_i| \gdeg(\ov{V_i})\leq \sum_i |n_i| \gdeg (V_i)=|\al^0|.
\]
Hence, by (a), the norm $|\wt{\al}|$ is bounded in terms of $n$ and $\dt$.

(d) Combining (c), \re{Step 2}(d) and \re{norm}(b),(c), the norm of $\wt{\al}_J:=(\pi_J)_*(i_J)^*\wt{\al}\in A_d(X'_J\times (X'_J)^{(q)})$
is therefore bounded in terms of $n$ and $\dt$.

(e) Since $\gdeg(X'_J)$ is bounded in terms of $n$ and $\dt$ (see \re{Step 2}(a)), the existence
of constant $M$ satisfying \form{second} follows from a combination of (d) and \re{spectral}.
\end{Emp}

\appendix

\section{Bounded collections}\label{Uniform_Projective}

\subsection{Stability of bounded collections under standard operations}.

%\addtocontents{toc}{\protect\setcounter{tocdepth}{2}}

In this section we introduce {\em bounded collections} and show that standard operations
preserve {\em boundness}.

%\subsection{Some notations}\label{notation_boundedness}
%
%We will be using the following notions through this section.

%\begin{Not}\label{belonging_family}
%Let $S$ be a scheme of finite type over $\B{Z}$.
%
%
%
%%The base scheme $S$ will be called the realizing scheme.
%
%\end{Not}

\begin{Not}\label{N:bounded}
Let $\Lambda$ be a set (or a class).
%
%\item A \textit{family of schemes} (resp.  a \textit{family of morphisms}) over $S$ is a scheme $\C{X}$ of finite type over $S$ (resp. a $S$-morphism $\C{X} \to \C{Y}$ of schemes of finite type over $S$).~Occasionally we also deal with a \textit{family of correspondences/diagrams} over $S$,~whose definition the reader can easily formulate.
%
%\item A scheme $X/k$ (resp. morphism of schemes $f \colon X \to Y$ over $k$) is said to \textit{belong to} a family of schemes $\C{X}$ (resp. a family of morphisms $\pi: \C{X} \to \C{Y}$) over $S$,~if there exists a $k$-point of $S$ such that,~the base change of $\C{X}/S$ (resp. $\pi$) along this $k$-point is $X$ (resp. $f$).~There is also analogous notion of belonging to a family of correspondences,~diagrams etc.

(a) We say that a collection of morphisms $\{f_{\al}:X_{\al}\to Y_{\al}\}_{\al\in\La}$, where $f_{\al}$ is a morphism of schemes of finite type over an algebraically closed field $k_{\al}$, is {\em bounded}, if there exists a scheme $S$ of finite type over $\B{Z}$ and a morphism of schemes $f:X\to Y$ of finite type over $S$ such that for every $\al\in \La$ there exists a geometric point $s_{\al}\in S(k_{\al})$ such that $f_{\al}$ is isomorphic to the geometric fiber $f_{s_{\al}}$ of $f$ over $s_{\al}$. In this case, we say that a family $f:X\to Y$ {\em parameterizes} collection $\{f_{\al}\}_{\al}$ over $S$.

(b) Similarly (but easier),  we say that collection $\{X_{\al}\}_{\al\in\La}$ of schemes of finite type over $k_{\al}$ is {\em bounded},
if there exists a scheme $S$ as in (a) and a scheme $X$ of finite type over $S$ such that each $X_{\al}$ is isomorphic to a geometric fiber $X_{s_{\al}}$ over some $s_{\al}\in S(k_{\al})$. Again we say that family $X$ {\em parameterizes} collection  $\{X_{\al}\}_{\al}$ over $S$.

(c) Let $\C{P}$ be a class of morphisms of schemes, which is stable under pullbacks. We say that a bounded collection $\{f_{\al}\}_{\al}$ of morphisms $f_{\al}\in\C{P}$ is {\em $\C{P}$-bounded}, if there is a parameterizing family $f$, which belongs to $\C{P}$.
In particular, we can talk about bounded collections of open (resp. closed, resp. locally closed) embeddings.

(d) We say that a collection $\{Y_{\al}\subseteq X_{\al}\}_{\al}$, where $Y_{\al}\subseteq X_{\al}$ is an open (resp. closed, resp. locally closed) subscheme, is bounded, if the inclusions $\{Y_{\al}\hra X_{\al}\}_{\al}$ form a bounded collection of open (resp. closed, resp. locally closed)
embeddings.

(e) More generally, we define bounded collections of more complicated diagrams. For example,

\quad (i) we say that a collection
$\{X_{\al}\subseteq\B{P}^n_{k_{\al}}\}_{\al}$, where $X_{\al}\subseteq\B{P}^n_{k_{\al}}$ is a (locally) closed subscheme, is bounded, if it is parameterized by a (locally) closed subscheme $X\subseteq\B{P}^n_S$ for  $S$ as in (a).

\quad (ii) we say that a collection
$\{X_{\al}\subseteq\B{P}^n_{k_{\al}},X'_{\al}\subseteq\B{P}^m_{X_{\al}}\}_{\al}$, where $X_{\al}\subseteq\B{P}^n_{k_{\al}}$ and
$X'_{\al}\subseteq\B{P}^m_{X_{\al}}$ are (locally) closed subschemes  is bounded, if it is parameterized by a pair (locally) closed subschemes $X\subseteq\B{P}^n_S, X'\subseteq\B{P}^m_X$ for  $S$ as in (a).
\end{Not}

The following standard lemma is central for what follows.

 \begin{Lem}\label{L:reduced irreducible}
 Let $f:X\to S$ be between a morphism of schemes of finite type over $\B{Z}$. Then there exists a surjective morphism $S'\to S$, where $S'$ is a scheme of finite type over $\B{Z}$ that the pullback $X':=X\times_S S'$ satisfies the following properties:

 (a) All geometric fibers of $X'_{\red}\to S'$ are reduced.

 (b) Let $X'_1,\ldots,X'_m$ be all the irreducible components of $X'_{\red}$. Then all non-empty geometric fibers of $X'_i\to S'$ are integral.
 \end{Lem}

 \begin{proof}
 By the Noetherian induction, we can assume that for every closed subscheme $Z\subseteq S$ the assertion holds for the morphism
 $f|_{Z}:f^{-1}(Z)\to Z$. Therefore we can replace $S$ by an open non-empty subscheme. In other words, we can assume that
 $S$ is integral, and it suffices to show the existence a dominant morphism $S'\to S$ satisfying the required properties. Then assertion (a) follows from \cite[Tag 0550 and Tag 0578]{Stacks}, while assertion (b) from
 \cite[Tag 0551 and Tag 0559]{Stacks}. Alternatively, both results can be deduced from \cite[Th\`eor\'eme 9.7.7]{EGAIV(III)}.
 \end{proof}

 %By Noetherian induction we may assume $S$ is integral.~Further using  \cite[Tag 0550]{Stacks} and \cite[Tag 0578]{Stacks} we may assume that $X/S$ has geometrically reduced fibers. Moreover using \cite[Tag 0551]{Stacks} we may assume that there exists a surjective finite \'etale map $S' \to S$ such that the irreducible components of the generic fiber of $X'/S'$ are geometrically irreducible (here $X'$ is the base change of $X$ along $S'$).~Again by Noetherian induction we may assume that $S'$ is integral,~and that every irreducible component of $X'$ dominates $S'$.

 %Let $\{X'_{\alpha}\}_{\alpha \in \Lambda}$ be the finitely many irreducible components of $X'$ (with the reduced structure).~Note that the generic fibers of $X'_{\alpha} \to S'$ are geometrically integral.~Thus there exists an open dense subset of $S'$ such that every geometric fiber of any $X'_{\alpha}/S$ in this open dense subset is geometrically integral \cite[Th\`eor\'eme 9.7.7 (iv)]{EGAIV(III)}.~Thus replacing $S'$ by this open subset and $X'_{\alpha}$'s by their inverse images we get the required result.
 %\end{proof}

% \begin{Rem}\label{relative_base_remark}
%Suppose that in addition $X$ was a scheme over a $S$-scheme $Y$.~Then our proof of Lemma \ref{L:reduced irreducible} implies that the $Z_{i,\alpha}$ are naturally schemes over $Y'$,~the base change of $Y$ along $S'$.
% \end{Rem}

%Lemma \ref{L:reduced irreducible} immediately implies the following result.

\begin{Cor}\label{C:reduced irreducible}
Let $\{X_{\alpha}\}_{\al}$ be a bounded collection. Then the collections $\{(X_{\al})_{\red}\}_{\al}$ and
$\{X_{\alpha,i_{\al}}\}_{\al,i_{\al}}$, where $X_{\alpha,i_{\al}}$ runs over the set of all irreducible components of $X_{\al}$, are bounded.
\end{Cor}

\begin{proof}
Let $X\to S$ parameterizes collection $\{X_{\al}\}_{\al}$, and let $X'\to S'$ be as in  Lemma \ref{L:reduced irreducible}.
Then $X'_{\red}\to S'$ parameterizes collection $\{(X_{\al})_{\red}\}_{\al}$, while $\sqcup_{i=1}^m X'_i\to \sqcup_{i=1}^m S$ parameterizes
$\{X_{\alpha,i_{\al}}\}_{\al,i_{\al}}$.
\end{proof}

\begin{Lem}\label{L:simple bounded}
(a) Let $\{X_{\alpha}\subseteq \B{P}^n_{k_{\al}}\}_{\alpha \in \Lambda}$ and $\{Y_{\alpha}\subseteq \B{P}^n_{k_{\al}}\}_{\alpha \in \Lambda}$
is a bounded collection (of locally closed subschemes). Then $\{X_{\al},Y_{\al}\subseteq\B{P}^n_{k_{\al}}\}_{\al}$ and $\{X_{\al}\cap Y_{\al}\subseteq\B{P}^n_{k_{\al}}\}_{\al}$ are bounded collections.

(b) In the situation of (a), let  $\Lambda^{'} \subseteq \Lambda$ consists of all $\alpha$ such that $Y_{\alpha}\subseteq X_{\al}$.
Then the collection $\{Y_{\alpha} \subseteq X_{\alpha}\}_{\alpha \in \Lambda'}$ is bounded.

(c) Let $\{f_{\al}\}_{\al}$ be a bounded collections of morphisms such that each $f_{\al}$ is an open (resp. closed, resp. locally closed) embeddings. Then  $\{f_{\al}\}_{\al}$ be a bounded collection of open (resp. closed, resp. locally closed) embeddings (see \rn{bounded}(c)).

(d) Let $\{Y_{\al}\subseteq X_{\al}\}_{\al}$ is a bounded collection of closed (resp. open) subschemes. Then the collection
$\{X_{\al}\sm Y_{\al}\}_{\al}$ (resp. $\{(X_{\al}\sm Y_{\al})_{\red}\}_{\al}$) is bounded.

(e) Let $\{Z'_{\al},Z''_{\al}\subseteq X_{\al}\}_{\al}$ be a bounded collection such that $Z'_{\al}$ and $Z''_{\al}$ are closed subschemes of $X_{\al}$, and let $Z_{\alpha}\subseteq X_{\al}$ be the closed subscheme such that $\C{I}_{Z_{\al},X_{\al}}:=\C{I}_{Z'_{\al},X_{\alpha}}\C{I}_{Z''_{\alpha},X_{\al}}$.
Then the collection $\{Z_{\al}\subseteq X_{\al}\}_{\al}$ is bounded.

(f) Let $\{f_{\al}:X_{\al}\to Y_{\al}, Z_{\al}\subseteq Y_{\al}\}_{\al}$  be a bounded collection. Then the induced
collection $\{f_{\al}^{-1}(Z_{\al})\subseteq X_{\al}\}_{\al}$ is bounded.

(g) Let $\{f_{\al}:X_{\al}\to Y_{\al}\}_{\al}$ be a bounded collection. Then the collection $\{\ov{f_{\al}(X_{\al})}\subseteq Y_{\al}\}_{\al}$ of schematic closures of images is bounded.
\end{Lem}

\begin{proof}
(a) Let $X_0\subseteq\B{P}^n_{S_X}$ and $Y_0\subseteq\B{P}^n_{S_Y}$ parameterize collections $\{X_{\alpha}\subseteq \B{P}^n_{k_{\al}}\}_{\alpha \in \Lambda}$ and $\{Y_{\alpha}\subseteq \B{P}^n_{k_{\al}}\}_{\alpha \in \Lambda}$, respectively, and set $S:=S_X\times_{\B{Z}} S_Y$.
Then $X:=X_0\times_{S_X} S, Y:=Y_0\times_{S_Y} S\subseteq  \B{P}^n_{S}$ and $X\cap Y\subseteq  \B{P}^n_{S}$ parameterize $\{X_{\al},Y_{\al}\subseteq\B{P}^n_{k_{\al}}\}_{\al}$ and $\{X_{\al}\cap Y_{\al}\subseteq\B{P}^n_{k_{\al}}\}_{\al}$, respectively.

(b) We continue the notation of the proof of (a). Consider the subset $S_0\subseteq S$ consisting of $s\in S$ such that the embedding $(Y\cap X)_s\hra Y_s$ is an isomorphism. By \cite[Proposition 9.6.1(xi)]{EGAIV(III)}, $S_0$ is constructible. Therefore there exists a morphism $S'\to S$ of schemes of finite type over $\B{Z}$, whose image is $S_0$.
Then $Y\times_S S'\subseteq X\times_S S'\subseteq  \B{P}^n_{S'}$ parameterizes $\{Y_{\alpha} \subseteq X_{\alpha}\}_{\alpha \in \Lambda'}$.

(c) Arguing as in (b), it suffices to show that if $f:X\to Y$ is a morphism of schemes of finite type over $S$, then
the set of $s\in S$ such that $f_s$ is an open (resp. closed, resp. locally closed) embedding is constructible. But this follows from \cite[Proposition 9.6.1(viii),(ix),(x)]{EGAIV(III)}.

(d) Let $Y\subseteq X$ be a closed (resp. an open) subscheme, representing $\{Y_{\al}\subseteq X_{\al}\}_{\al}$ over $S$ (use (c)).
Then $X\sm Y$ (resp. $(X\sm Y)_{\red}$) parameterizes collection $\{X_{\al}\sm Y_{\al}\}_{\al}$ (resp. $\{(X\sm Y)_{\red,\al}\}_{\al}$).
Since  $((X\sm Y)_{\red,\al})_{\red}=(X_{\al}\sm Y_{\al})_{\red}$, the assertion follows
from \rco{reduced irreducible}.

(e) Let  $Z',Z''\subseteq X$ be closed subschemes parameterizing $\{Z'_{\al},Z''_{\al}\subseteq X_{\al}\}_{\al}$ (use (c)).
Then the closed subscheme  $Z\subseteq X$ such that $\C{I}_{Z,X}:=\C{I}_{Z',X}\C{I}_{Z'',X}$ parameterizes
$\{Z_{\al}\subseteq X_{\al}\}_{\al}$.

(f) If $(f:X\to Y, Z\subseteq X)$ parameterizes $\{f_{\al}:X_{\al}\to Y_{\al}, Z_{\al}\subseteq Y_{\al}\}_{\al}$, then
$f^{-1}(Z)\subseteq X'$ parameterizes $\{f_{\al}^{-1}(Z_{\al})\subseteq X'_{\al}\}_{\al}$.

(g) Let $f:X\to Y$ parameterizes $\{f_{\al}:X_{\al}\to Y_{\al}\}_{\al}$, and let $\ov{f(X)}\subseteq Y$ be the schematic closure of the image.
By Noetherian induction, it remains to show that there exists an open non-empty subscheme $U\subseteq S$ such that $\ov{f_s(X_s)}=\ov{f(X)}_s$ for every $s\in U$. Replacing $S$ by an open subscheme, we can assume that $S$ is integral and the generic fiber $X_{\eta}$ of $f$ is dense in $X$. Now the assertion follows from  \cite[Proposition 9.6.1(ii)]{EGAIV(III)}.
\end{proof}

\begin{Lem}\label{L:smooth}
(a) Let $\{X_{\al}\}_{\al}$ be a bounded collection. Then the collection $\{U_{\al}\subseteq X_{\al}\}_{\al}$, where $U_{\al}$ is the largest open subscheme of $X_{\al}$, smooth over $k_{\al}$, is bounded.

(b) Let $\{f_{\al}:X_{\al}\to Y_{\al}\}_{\al}$ be a bounded collection such that each $f_{\al}$ is a quasi-finite morphism between
integral schemes over $k_{\al}$. Then there exists a bounded collection $\{U_{\al}\subseteq X_{\al}\}_{\al}$, where $U_{\al}$ is an open dense  subscheme of $X_{\al}$ such that $f_{\al}|_{U_{\al}}$ is uh-\'etale (see \re{uh-etale}).
\end{Lem}

\begin{proof}
(a) Let $f:X\to S$ be a morphism parameterizing $\{X_{\al}\}_{\al}$ such that $S$ is reduced. By \cite[Tag 0529]{Stacks}, there exists an open non-empty subset $V\subseteq S$ such that $f|_V:f^{-1}(V)\to V$ is flat. Let $U\subseteq f^{-1}(V)$ be the largest open subscheme,
where $f$ is smooth. Then it follows from \cite[Tag 01V9]{Stacks} that $U_s\subseteq X_s$ is the largest subscheme of $X_s$, which is smooth over $s$, for every geometric point $s$ of $V$. Now the assertion follows by the Noetherian induction on $S$.

%Let $f:X\to Y$ be a family parameterizing $\{X_{\al}\}_{\al}$ over $S$ and let $p:X\to S$ be the projection.
%Let $V\subseteq X$ be the largest open subset such that $f|_{V}:V\to Y$ is smooth, and let $U$ be the locus of points $x\in X$ such that the fiber $f_{p(x)}:X_{p(x)}\to Y_{p(x)}$ is smooth. By Noetherian induction, it suffices to show that there exists an open subscheme $S_0\subseteq S$ such that $V_t= U_t$ for every $t\in S_0$. Note that $V\subseteq U$ and that for every generic point $\eta\in S$ we have an equality of fibers
%$V_{\eta}=U_{\eta}$. Since $V\subseteq X$ is constructible (by \cite[Proposition 17.7.11]{EGAIV(IV)}), the assertion follows by Chevalley theorem.

(b) Let $f:X\to Y$ be a family parameterizing collection $\{f_{\al}\}_{\al}$ over $S$. Using \cite[Proposition 9.6.1(vii)]{EGAIV(III)} and \cite[Th\'eor\`eme 9.7.7]{EGAIV(III)}, the set of points $t\in S$ such that $f_t:X_t\to Y_t$ is a quasi-finite morphism between geometrically integral schemes is a constructible subset. Thus, arguing as in \rl{simple bounded}(b), we can assume that these properties are satisfied for all points $t\in S$. Using Noetherian induction on $S$, and replacing $S$ by an open subset we can assume that $X$ and $Y$ are integral.
Then, by \rl{uh-etale}(a), there exists an open dense subscheme $U\subseteq X$ such that $f|_U$ is uh-\'etale. Replacing $S$ by an open subscheme,
we can assume that $U_t\subseteq X_t$ is dense for every $t\in S$. Then $U\subseteq X$ parameterizes the collection $\{U_{\al}\subseteq X_{\al}\}_{\al}$ we were looking for.
\end{proof}

The following standard fact seems to be missing the literature.

\begin{Lem} \label{L:blowup}
Let $X$ be a scheme over $S$, and $Z\subseteq X$ a closed subscheme of X such that
the normal cone $N_Z(X)$ is flat over $S$. Then $\Bl_Z(X)\to S$ commutes with
all pullbacks, that is, for all morphisms $f:S'\to S$ the natural morphism
\[
\Bl_{Z\times_S S'}(X\times_S S')\isom \Bl_Z(X)\times_S S'
\]
is an isomorphism.
\end{Lem}

\begin{proof}
Since the normal cone $N_Z(X)=\C{Spec}(\oplus_{n\geq 0}((\C{I}_{Z,X})^n/(\C{I}_{Z,X})^{n+1}))$
is flat over $S$, we conclude that every $(\C{I}_{Z,X})^n/(\C{I}_{Z,X})^{n+1}$ is flat over $S$.
By induction we therefore conclude that every quotient $\C{O}_X/(\C{I}_{Z,X})^{n}$ is flat over $S$,
which implies that $f^*(\C{I}_{Z,X})^{n}\simeq (\C{I}_{Z\times_S S',X\times_S S'})^n$ for every $n$.
Since $\Bl_Z(X)=\C{Proj}(\oplus_{n\geq 0}(\C{I}_{Z,X})^n)$ and similarly for $\Bl_{Z\times_S S',X\times_S S'}$, the assertion follows.
\end{proof}

%The next two lemma are about boundedness of blow-ups and alterations.

\begin{Cor}\label{C:blowup}
Let $\{Z_{\alpha} \subseteq X_{\alpha}\subseteq\B{P}^n_{k_{\al}}\}_{\al}$ be a bounded collection,
where each $Z_{\al}\subseteq X_{\al}$ is a closed subscheme. Then there exists $m\in\B{N}$ and a closed embedding $i_{\al}:\Bl_{Z_{\al}}(X_{\al})\hra \B{P}^m_{X_{\al}}$ over $X_{\al}$ for every $\al$ such that the collection of closed embeddings $\{i_{\al}\}_{\al\in\La}$ is bounded.
%the collection $\{\on{Bl}_{Z_{\alpha}}(X_{\alpha})\}$ is a bounded collection of closed subvarieties in some $\B{P}^N_{k_{\al}}$.~Moreover the collection of the natural morphisms $\{\on{Bl}_{Z_{\alpha}}(X_{\alpha}) \to X_{\alpha}\}$ is also bounded.
\end{Cor}

\begin{proof}
Let $Z \subseteq X \subseteq \B{P}^n_S$ be a family parameterizing $\{Z_{\alpha} \subseteq X_{\alpha}\subseteq\B{P}^n_{k_{\al}}\}_{\al}$ such that $Z \subseteq X$ is a closed subscheme and $S$ is reduced. Then $\Bl_Z(X)$ is of finite type over $S$, thus there exists an open non-empty subset $V\subseteq S$ such that $N_Z(X)$ is flat over $V$ (possibly empty). By Noetherian induction on $S$, we can restrict all schemes to $V$, thus assuming that $S$ is affine and $N_Z(X)$ is flat over $S$. But $S$ is affine, thus $X$ has an ample line bundle, hence there exists a closed embedding $i:\on{Bl}_Z(X)\hra\B{P}^m_X$ over $X$. Then, by \rl{blowup}, this closed embedding parameterizes a bounded collection of closed embeddings
$\{i_{\al}:\Bl_{Z_{\al}}(X_{\al})\hra \B{P}^m_{X_{\al}}\}_{\al}$ over $\{X_{\al}\}_{\al}$.
\end{proof}

\begin{Emp} \label{E:blowup}
{\bf Remark.} The argument of \rco{blowup} also shows that if $\{Z_{\alpha} \subseteq X_{\alpha}\}_{\al}$ is a bounded collection of closed subschemes, then the collection of projections $\{\Bl_{Z_{\al}}(X_{\al})\to X_{\al}\}_{\al}$ is bounded.
\end{Emp}

\begin{Cor}\label{boundedness_normal_cone}
Let $\{Z_{\alpha} \subseteq X_{\alpha}\}_{\al}$  be a bounded family of closed subschemes. Then the collection of projections
$\{N_{Z_{\alpha}}(X_{\alpha})\to X_{\al}\}_{\al}$ is bounded.
\end{Cor}

\begin{proof}
Let $\Bl_{Z_{\al}\times\{0\}}(X_{\al}\times\B{A}^1)_0$ be the fiber of $\Bl_{Z_{\al}\times\{0\}}(X_{\al}\times\B{A}^1)\to\B{A}^1$ over $0\in\B{A}^1$. Then $N_{Z_{\al}}(X_{\al})$ can be described as $N_{Z_{\al}}(X_{\al})=\Bl_{Z_{\al}\times\{0\}}(X_{\al}\times\B{A}^1)_0\sm\Bl_{Z_{\al}}(X_{\al})$ (compare \cite[Section 5.1]{Fu} or \cite{Va0}).
Now the result follows by a combination of Remark \re{blowup} and \rl{simple bounded}(d),(f).
%The result follows from the usual deformation to the normal cone argument \cite[Section 5.1]{Fu}.~More specifically Lemma \ref{L:simple bounded}(a) and Corollary \ref{blowup_bounded_general} imply that the collection $M_{Y_{\alpha}}X_{\alpha} \colonequals \on{Bl}_{Y_{\alpha} \times \infty}(X_{\alpha} \times_{k_{\alpha}} \P_{k_{\alpha}}^1)$ is bounded.~Lemma \ref{bounded_family_inverse_image} then implies that the fiber over $\infty$ of the natural map from $M_{Y_{\alpha}}X_{\alpha}$ to $\P_{k_{\alpha}}^1$ is bounded.~This fiber consists of $\P(N_{Y_{\alpha}}X_{\alpha} \oplus 1)$ and $\on{Bl}_{Y_{\alpha}}(X_{\alpha})$ intersecting along the hyperplane at infinity of $\P(N_{Y_{\alpha}}X_{\alpha} \oplus 1)$ identified with the exceptional divisor in $\on{Bl}_{Y_{\alpha}}(X_{\alpha})$.~Since $N_{Y_{\alpha}}X_{\alpha}$ is a complement of this hyperplane section,~the result now follows from Corollary \ref{C:reduced irreducible} and Lemma \ref{L:simple bounded}(d).
\end{proof}

\begin{Lem}\label{L:alteration}
Let $\{Z_{\alpha} \subsetneq X_{\alpha}\subseteq\B{P}^n_{k_{\al}}\}_{\al}$ be a bounded collection,
where $X_{\al}$ is integral and $Z_{\al}\subsetneq X_{\al}$ is a closed subscheme. Then there exists $m\in\B{N}$, an alteration
$f_{\al}:\wt{X}_{\al}\to X_{\al}$ and a closed embedding $i_{\al}:\wt{X}_{\al}\hra \B{P}^m_{X_{\al}}$ over $X_{\al}$ for every $\al$ such that
$\wt{X}_{\al}$ is smooth, $f_{\al}^{-1}(Z_{\al})_{\red}$ is a union of smooth divisors with normal crossings, and the collection of closed embeddings $\{i_{\al}\}_{\al\in\La}$ is bounded.
\end{Lem}

%Let $S$ be a Noetherian scheme.~Let $X \subseteq \B{P}^n_S$ be a closed subscheme and $Z \subseteq X$ a closed subscheme of $X$.~Then for any $s \in S$ such that $X_s$ is a geometrically integral and $Z_s$ is a proper closed subset of $X_s$,~there exists an alteration $\pi_s: \tilde{X}_s \to X_s$ with $\pi_s^{-1}(Z_s)$ a strict normal crossings divisor in $X_s$ satisfying the following conditions.
%
%\begin{enumerate}[(a)]
%
%\item $\tilde{X}_s$ is a geometrically integral.
%
%\item There exists an integer $m$ such that for any such $\tilde{X}_s$,~there exists an embedding $\tilde{X_s} \subseteq \B{P}^m_s$.
%
%\item The set consisting of Hilbert polynomials of $\tilde{X}_s$ (via (a)) is a finite set.
%
%\item The set consisting of the Hilbert polynomials of $\Gamma_{\pi_s}: \tilde{X}_s \subseteq \tilde{X}_s \times_{k(s)} X_s \subseteq \B{P}^{m}_s \times_{k(s)} \B{P}^n_s$ is a finite set.~Here $\Gamma_{\pi_s}$ is the graph of $\pi_s$.
%
%\item Let $\tilde{X}_{i,s} \subseteq \tilde{X}_s$ be the smooth (over $k(s)$) irreducible components of $\pi_s^{-1}(Z_s)$.~Then the set consisting of Hilbert Polynomials of $\tilde{X}_{i,s}$ as closed subscheme of $\B{P}^m_s$ (via (b)) and $\B{P}^{m}_s \times_{k(s)} \B{P}^n_s$  (via (d)) is a finite set.
%
%\item Let $X^0_s$ (resp. $\tilde{X}^0_s$) be the complement of $Z_s$ (resp. $\pi_s^{-1}(Z_s)$) in $X_s$ (resp. $\tilde{X}_s$).~Let $\pi_s^0$ be the restriction of $\pi_s$ to $\tilde{X}^0_s$.~Then the degree of $\Gamma_{\pi_s^0} \subseteq \B{P}^m_s \times_{k(s)} \B{P}^n_s$ is uniformly bounded as $s$ varies.
%
%\end{enumerate}
%
%\end{Lem}
%

\begin{proof}
Let $Z \subsetneq X \subseteq \B{P}^n_S$ be a family parameterizing $\{Z_{\alpha} \subsetneq X_{\alpha}\subseteq\B{P}^n_{k_{\al}}\}_{\al}$. By \cite[Th\'eor\`eme 9.7.7,~Corollaire 9.5.2]{EGAIV(III)}, the set of points $s \in S$ such that $X_s$ is geometrically integral and $Z_s\neq X_s$ is  a constructible subset. Thus, arguing as in \rl{simple bounded}(b), we can assume that these properties are satisfied for all points $s \in S$. By the Noetherian induction on $S$, we may replace $S$ by an open non-empty subscheme, thus assuming that assume that $S$ is integral with generic fiber $\eta$ and geometric generic fiber $\ov{\eta}$.

By theorem of de Jong (\cite[Theorem 4.1]{dJ}), there exists  an alteration
$f_{\ov{\eta}}:\wt{X}_{\ov{\eta}}\to X_{\ov{\eta}}$ such that $\wt{X}_{\ov{\eta}}$ is a smooth projective variety over $\ov{\eta}$, and $f^{-1}_{\ov{\eta}}(Z_{\ov{\eta}})_{\red}$ is a union of smooth divisors $Z_{j,\ov{\eta}}$ with normal crossings. In particular, there exists a closed embedding  $i_{\ov{\eta}}:\wt{X}_{\ov{\eta}}\hra\B{P}^m_{X_{\ov{\eta}}}$ over $X_{\ov{\eta}}$.

For every morphism $S'\to S$ we set $X':=X\times_S S'$ and $Z':=Z\times_S S'$. By standard limit theorems, the morphism $\ov{\eta}\to S$ factors as $\ov{\eta}\to S'\to S$, where $S'$ is an integral scheme of finite type over $\B{Z}$ such that $\iota_{\ov{\eta}}$ is a pullback of a closed embedding $i:\wt{X}'\hra\B{P}^m_{X'}$, and each $Z_{j,\ov{\eta}}$ is a pullback of a closed subscheme $Z_j\subseteq \wt{X}'$.
Consider the composition $f:\wt{X}'\hra \B{P}^m_{X'}\to X'$.

Replacing $S'$ by an open subscheme, one can assume that  $\wt{X}'\to S'$ is smooth, $\{Z_j\}_j$ are smooth divisors over $S'$ with normal crossings, that is, $\cap_{j\in J}Z_j$ is smooth over $S'$ of correct relative dimension, and $f^{-1}(Z')_{\red}=\cup_j Z_j$.
Furthermore, we can assume that $X'\to S'$ and $\wt{X}'\to S'$ are flat and have geometrically integral fibers.

Then for every geometric point $s'$ of $S'$, the induced morphism $f_s:\wt{X}'_s\to X'_s$ is an alteration (being a surjective morphism between varieties of the same dimension), $\wt{X}'_s$ is smooth and $f_s^{-1}(Z'_s)_{\red}$ is a union of smooth divisors with normal crossings. Now
the assertion of the lemma follows by Noetherian induction.
\end{proof}

\begin{Lem}\label{L:degree bound}
Let $\{f_{\alpha} \colon X_{\alpha} \to Y_{\alpha}\}_{\alpha \in \Lambda}$ be a bounded collection of morphisms, where
each $X_{\al}$ and $Y_{\al}$ are integral, while $f_{\al}$ is generically finite.
Then collection $\{\on{deg}(f_{\alpha})\}_{\al}$ of positive integers is bounded.
\end{Lem}

\begin{proof}
Let $f:X\to Y$ be a family parameterising $\{f_{\al}\}_{\al}$. Note that if $\eta_{\al}\in Y_{\al}$ is a generic point, then we have an equality
$\deg(f_{\al})=\dim_{k(\eta_{\al})}H^0(f_{\al}^{-1}(\eta_{\al}),~\sO_{f_{\al}^{-1}(\eta_{\al})})$. Therefore it suffices to show that there exists a constant $N$ such that for every $y \in Y$ such that $f^{-1}(y)$ is finite we have $\dim_{k(y)}H^0(f^{-1}(y),~\sO_{f^{-1}(y)})\leq N$. By Noetherian induction, we can replace $Y$ by an open subscheme (and $X$ by its preimage), thus assuming that $Y$ is irreducible. Moreover, we can assume that $f$ is generically finite, thus replacing Y by its open subscheme, thus assuming that $f$ is finite flat. In this case, the function  $y\mapsto \dim_{k(y)}H^0(f^{-1}(y),~\sO_{f^{-1}(y)})$ is constant.
\end{proof}

\begin{Lem}\label{L:ram}
Let $f:X\to S$ be a morphism between schemes of finite type over $\B{Z}$. Then there exists a constant $N$ such that for
every geometric point $s$ of $S$ the thickness of the geometric fiber $X_s$ (see \re{ramification}(a)) is at most $N$. In particular, the ramification $\ram(f)$ (see \re{ramification}(c)) is finite.
\end{Lem}

\begin{proof}
Let $f':X'\to S'$ be as \rl{reduced irreducible}, and let $N$ be the thickness of $X'$. We claim this $N$ satisfies the required property.
Indeed, by definition, we have $(\C{I}_{X'_{\red},X'})^N=0$. Hence for every geometric point $s'$ of $S'$, we have
$(\C{I}_{(X'_{\red})_{s'},X'_{s'}})^N=0$. By our assumption, $(X'_{\red})_s$ is reduced, therefore $(\C{I}_{(X'_{s'})_{\red},X'_{s'}})^N=0$,
thus the thickness of $X'_{s'}$ is at most $N$. Since every geometric fiber $X_s$ is isomorphic to some geometric fiber $X'_{s'}$, the assertion follows.
\end{proof}

%As an immediate corollary we have the following.

\begin{Cor}\label{C:ram}

(a) For every bounded family $\{X_{\alpha}\}_{\al}$, the collection of thicknesses of $X_{\alpha}$ is bounded.

(b) Let $\{f_{\alpha} \colon X_{\alpha} \to Y_{\alpha}, Z_{\al}\subseteq Y_{\al}\}_{\alpha \in \Lambda}$, where $Z_{\al}\subseteq Y_{\al}$ is a closed subscheme, be a bounded collection. Then the collection $\{\on{ram}(f_{\alpha},Z_{\alpha})\}_{\al}$ of positive integers is bounded.

(c) Let $\{f_{\alpha} \colon X_{\alpha} \to Y_{\alpha}\}_{\alpha \in \Lambda}$ be a bounded collection. Then the collection
 $\{\on{ram}(f_{\alpha})\}_{\al}$ of positive integers is bounded.

%Let $X,Y \hookrightarrow \B{P}^n_k$ be reduced closed sub-schemes of degree less than or equal to $d$ over an algebraically closed field $k$.~Let $C \hookrightarrow X \times Y \hookrightarrow \B{P}^{n^2+2n}_k$ be a reduced closed sub-scheme of degree less than or equal to $e$.
%
%Let $Z \hookrightarrow Y \hookrightarrow \B{P}^n_k $ be a reduced closed sub-scheme of degree less than or equal to $f$.~Then there exists a constant $M_{\on{ram}}$ depending only on $n,~d,~e,~f$ such that $\on{ram}(c_2,Z) \leq M_{\on{ram}}$.

\end{Cor}

\begin{proof}
(a) Let $f \colon X \to S$ be a family parameterizing collection $\{X_{\alpha}\}_{\al}$. Then the assertion
follows from Lemma \ref{L:ram} applied to $f$.

(b) By \rl{simple bounded}(f), the collection $\{f_{\alpha}^{-1}(Z_{\alpha})\}$ is a bounded. Thus the result follows from (a).

(c) follows immediately from (b).
\end{proof}

\subsection{Boundedness of collections of quasi-projective varieties}

%\addtocontents{toc}{\protect\setcounter{tocdepth}{2}}

%The starting point for us is the following theorem of Kleiman.~Let $\sP_{n,d}$ be the collection of Hilbert polynomials  of integral closed subschemes of degree less than or equal to $\dt$ of $\B{P}^n_k$ for an arbitrary algebraically closed field $k$.
%
%\begin{thm}[Kleiman's boundedness]\label{Kleiman_Boundedness}(\cite{SGA6},~Corollary 3.10)
%
%The set $\sP_{n,d}$ is a finite set.
%
%\end{thm}
%
%Combined with the existence of Hilbert schemes,~Theorem \ref{Kleiman_Boundedness} immediately implies the following.
%

The following result is central for what follows.

\begin{Thm}\label{T:boundedness_degree}
For every $n, \dt\in\B{N}$, the collection $\{X_{\al}\subseteq\B{P}^n_{k_{\al}}\}_{\al}$, where $X_{\al}\subseteq\B{P}^n_{k_{\al}}$
is closed subvariety of $\B{P}^n_{k_{\al}}$ of degree at most $\dt$, is bounded.
\end{Thm}

Theorem \ref{T:boundedness_degree} is a standard consequence of a Kleiman's boundedness Theorem (\cite[Exp. XIII,~Corollary 3.10(ii)]{SGA6}, whose proof is rather involved. For convenience of the reader, we include a much shorter proof using a method  suggested to us by Udi Hrushovski.

Let $k$ be an algebraically closed field. The following fact is standard.

\begin{Lem}\label{L:noether}
Let $X \subsetneq \B{P}^n_k$ be a closed subvariety, let $p \in \P_k^n \sm X$ be a closed point, and let $\pi_p \colon \B{P}^n_k \sm \{p\} \to \B{P}^{n-1}_k$ be the linear projection from $p$. Then $\pi_p|_X$ is a finite map, and we have an inequality $\on{deg}(\pi_p(X)) \leq \on{deg}(X)$.
\end{Lem}

%\begin{proof}
%Notice that $\pi_p|X$ is projective, because $X$ is such, and affine, because $\pi_p$ is such. Therefore $\pi_p|X$ is finite.~Since $\pi_p|_X$ is dominant onto its image,~the natural map $\sO_{\pi_p(X)} \to \pi_{p*} \sO_X$ is an inclusion. Therefore for every $n \geq 0$, we have an inequality $h^0(\pi_p(X),\sO_{\pi_p(X)}(n)) \leq h^0(X,\sO_X(n))$, implying the assertion.
%\end{proof}

\begin{Prop}\label{irreducible_component_Complete_intersection}
Let $X \subsetneq \B{P}^n_k$ be a closed subvariety.~Then $X$ is an irreducible component of a complete intersection in $\B{P}^n_k$ defined by equations of degree at most $\deg(X)$.
\end{Prop}

\begin{proof}
We set $d:=\dim(X)$, and proceed by induction on $n-d$. If $d=n-1$, then $X\subset \B{P}^n_k$ is a hypersurface, so the result is obvious.
Assume now that $d< n-1$, and let $p \in \P_k^n \sm X$ be a closed point.~Then \rl{noether} implies that $\pi_p(X) \subsetneq \B{P}^{n-1}_k$ is a closed subvariety of dimension $d$ and degree at most $\deg(X)$. Hence by induction hypothesis $\pi_p(X)$ as an irreducible component of a complete intersection $W'$ in $\B{P}^{n-1}_k$ defined by equations of degree at most $\on{deg}(\pi_p(X))\leq\on{deg}(X)$.

Denote by $W\subseteq\B{P}^n$ the closure of $\pi_p^{-1}(W')\subseteq\B{P}^n_k\sm\{p\}$. Then $W$ is a complete intersection in $\B{P}^n_k$
of dimension $d+1$, defined by equations of degree at most $\on{deg}(\pi_p(X))\leq\on{deg}(X)$, and $X\subseteq W$.

Next, applying \rl{noether} $(n-d-1)$ times, we get a linear projection $\pi_L \colon \B{P}^n_k \sm L \to \B{P}^{d+1}_k$ such that
$\pi_L|_W:W\to \B{P}^{d+1}_k$ is finite. Moreover, $\pi_L(X)\subsetneq \B{P}^{d+1}_k$ is a hypersurface of degree
$\deg(\pi_L(X))\leq\deg(X)$.

Let $W''\subseteq \B{P}_k^n$ be the closure of $\pi_L^{-1}(\pi_L(X))\subseteq\B{P}^n_k\sm L$. Then $W''$ is a hypersurface of degree
$\deg(\pi_L(X))\leq\deg(X)$, and $X\subseteq W\cap W''$.

It remains to show that $W \cap W''$ is a complete intersection.
But this follows from the fact that since $W$ is a complete intersection, while $\pi_L|_W$ is finite, for each irreducible component $W_i\subseteq W$, we have $\pi_L(W_i)\not\subseteq \pi_L(X)$, thus $W_i\not\subseteq W''$.
\end{proof}

\begin{proof}[Proof of Theorem \ref{T:boundedness_degree}]
Since the collection of hypersurfaces in $\P_{k_{\al}}^n$ of degree at most $\dt$ are bounded, the assertion follows from a combination of Proposition \ref{irreducible_component_Complete_intersection} and Corollary \ref{C:reduced irreducible}.
\end{proof}

As a consequence, we obtain an easy criterion when a collection of quasi-projective schemes is bounded.

\begin{Cor} \label{C:bounded gdeg}
Let $n$ be fixed positive integer.~Then the collection $\{X_{\al}\}_{\al}$ of reduced locally closed subschemes of $\B{P}^n_{k_{\al}}$  is bounded if and only if the collection of geometric degrees $\{\on{gdeg}(X_{\al})\}_{\al}$ are bounded.
\end{Cor}

\begin{proof}
Assume that the geometric degrees $\{\on{gdeg}(X_{\al})\}_{\al}$ are bounded. Then each $X_{\al}$ has a presentation $X_{\al}=X'_{\al}\sm X''_{\al}$, where
$X'_{\al},X''_{\al}\subseteq\B{P}^n_{k_{\al}}$ are closed reduced subschemes, and the geometric degrees $\on{gdeg}(X'_{\al})$ and
$\on{gdeg}(X''_{\al})$ are bounded. Thus, by \rl{simple bounded}(a),(d), we reduce to the case when each $X_{\al}\subseteq\B{P}^n_{k_{\al}}$ is closed. Since $\gdeg(X_{\al})$ is the sum of degrees of irreducible components of $X_{\al}$, the assertion in this case
follows from a combination of Theorem \ref{T:boundedness_degree} with \rl{simple bounded}(e) and  \rco{reduced irreducible}.

To show the converse, we have to show that if $S$ is a scheme of finite type over $\B{Z}$, and $X\subseteq\B{P}^n_S$ is a locally closed subscheme,
the geometric degrees of all reduced geometric fibers of $X\to S$ are bounded. Writing $X$ as $X_1\sm X_2$ such that $X_1,X_2\subseteq\B{P}^n_S$ are closed, we can assume that $X\subseteq\B{P}^n_S$ is a closed subscheme. Let $S',X'$ and $X'_1,\ldots,X'_m$ be as in \rl{reduced irreducible}.
Replacing $X\to S$ by $X'_i\to S'$, we can assume that all geometric fibers of $X\to S$ are integral. Replacing $S$ by a disjoint union of a constructible stratification (and $X$ by its pullback), we can assume that $X\to S$ is flat. In this case,  the degrees of all geometric fibers of $X\to S$ are locally constant, hence bounded.
\end{proof}

\begin{Cor} \label{C:gdeg mor}
For every $n,m,\dt\in\B{N}$, the following collection is bounded
\[
\{X_{\al}\subseteq\B{P}^n_{k_{\al}},X'_{\al}\subseteq\B{P}^m_{X_{\al}}\,|\, \gdeg(X_{\al})\leq \dt\text{ and }\gdeg(X'_{\al})\leq \dt\}_{\al}.
\]
\end{Cor}

\begin{proof}
By \rco{bounded gdeg} and \rl{simple bounded}(a), the collection
\[
\{X_{\al}\subseteq\B{P}^n_{k_{\al}}, X'_{\al}\subseteq\B{P}^{nm+n+m}_{k_{\al}}\,|\,\gdeg(X_{\al})\leq\dt\text{ and }\gdeg(X'_{\al})\leq\dt\}_{\al}
\]
is bounded. Therefore the assertion follows from (the proof of) \rl{simple bounded}(b).
\end{proof}

\begin{Emp} \label{E:pfapp}
\begin{proof}[Proof of assertions \re{app}(c)-(k)]
(c) By \rco{bounded gdeg} and \rl{simple bounded}(a), the collection
$\{X_{\al},Y_{\al}\subseteq\B{P}^n_{k_{\al}}\,|\,\gdeg(X_{\al})\leq\dt\text{ and }\gdeg(Y_{\al})\leq\dt\}_{\al}$ is bounded. Therefore the collection $\{(X_{\al}\cap Y_{\al})_{\red}\subseteq\B{P}^n_{k_{\al}}\,|\,\gdeg(X_{\al})\leq\dt\text{ and }\gdeg(Y_{\al})\leq\dt\}_{\al}$ is bounded (by \rl{simple bounded}(a) and \rco{reduced irreducible}). Hence $\gdeg((X_{\al}\cap Y_{\al})_{\red})$ is bounded in terms of $n$ and $\dt$ (by \rco{bounded gdeg}).

(d) Combining \rco{bounded gdeg} and \rl{simple bounded}(a),(b) we get that the collection
\[
\{Z_{\al}\subseteq X_{\al}\subseteq\B{P}^n_{k_{\al}}\,|\,\gdeg(X_{\al})\leq\dt\text{ and }\gdeg(Z_{\al})\leq\dt\}_{\al}
\]
is bounded. Now arguing as in (c), the assertion follows from \rl{simple bounded}(d).

(e) Arguing as above, the boundness of $\gdeg(f^{-1}(Z)_{\red})$ follows from \rco{gdeg mor} and \rl{simple bounded}(f),
while the boundness of $\ram(f,Z)$ from \rco{ram}(b).

The remaining assertions are proved similarly: assertion (f) follows from \rl{simple bounded}(g); assertion
(g) from \rl{degree bound} and \rco{ram}(c); assertion (h) from \rl{smooth}(a); assertion
 (i) from \rl{smooth}(b); assertion (j) from \rco{blowup}; assertion (k) from \rl{alteration}.
\end{proof}
\end{Emp}

\begin{Emp} \label{E:pfclapp}
\begin{proof}[Proof of \rcl{app}]
Both assertions follow from the explicit formulas in the proof of \rp{blowup}, using previous results.

(a) Since $V=V_{\dim(X)}$, it suffices to show that the geometric degrees of the subschemes $V_j$, $Z_j$ and $U_j$ of $\B{P}_k^n$ defined by explicit formulas \form{f}-\form{zj} are bounded in terms of $n,\dt$ and $j$. By induction on $j$,
the assertion follows from a combination of \rl{gdeg}(a) and \re{app}(b)-(f).

(b) Consider collection of quadruples $\{(k_{\al},q_{\al},X^0_{\al},C_{\al}^0)\}_{\al\in\La}$ from \rt{main}(b), satisfying the assumptions of \re{assumptions}(a),(b). For $\al\in \La$, and we denote by $X_{\al}$,  $F_{\al}$ and $Z_{\al}$ the corresponding objects, denoted by
$X, F$ and $Z$, respectively, in the proof of \rp{blowup}.

By a combination of (a), \rl{gdeg}(a) and \re{app}(b)-(f) we see that the
 geometric degrees of $\{\ov{c_1(S_{\al})},{}^{\si^{-1}}\ov{c_2(S_{\al})}\subseteq\B{P}^n_{k_{\al}}\}_{\al\in\La, S_{\al}\in \on{Irr}(F_{\al})}$ and cardinalities $\{|\on{Irr}(F_{\al})|\}_{\al\in\La}$ are bounded. Hence, by a combination of \rco{bounded gdeg}
and \rl{simple bounded}(a),(e), the collection $\{Z_{\al}\subseteq X_{\al}\subseteq\B{P}^n_{k_{\al}}\}_{\al\in\La}$ is bounded.
Therefore the assertion follows by a combination of \rco{blowup} and \rco{bounded gdeg}.
\end{proof}
\end{Emp}

The following lemma lists simple properties of geometric degrees.

\begin{Lem} \label{L:gdeg}
Let $X\subseteq\B{P}_k^n$ be a reduced subscheme, and $\ov{X}\subseteq\B{P}_k^n$ the closure of $X$.

(a) For every $\si\in\Aut(k)$, we have an equality $\gdeg({}^{\si}X)=\gdeg(X)$.

(b) There exists a closed reduced subscheme $X''\subseteq\B{P}_k^n$ such that
$X=\ov{X}\sm X''$ and we have $\gdeg(X)=\gdeg(\ov{X})+\gdeg(X'')$. In particular, we have $\gdeg(\ov{X})\leq\gdeg(X)$.

(c) Let $X_1,\ldots,X_r$ be the irreducible components of $X$. Then we have inequalities $r\leq\gdeg(X)$ and  $\gdeg(X_i)\leq\gdeg(X)$ for all $i$.
\end{Lem}
\begin{proof}
(a) is clear.

(b) By definition, there exists  closed reduced subschemes $X',X''\subseteq\B{P}^n$ such that
$X=X'\sm X''$ and $\gdeg(X)=\gdeg(X')+\gdeg(X'')$. Then $X\subseteq X'$ is an open subscheme, thus $\ov{X}$ is a union of some irreducible
components of $X'$. Then $X=\ov{X}\sm X''$ and $\gdeg(\ov{X})\leq\gdeg(X')$, which implies the assertion.

(c) Note that $\ov{X}_1,\ldots,\ov{X}_r$ are irreducible components of $\ov{X}$. Then, by (b), we have inequalities $r\leq\gdeg(\ov{X})\leq\gdeg(X)$.
Choose $X''\subseteq\B{P}^n_k$, satisfying (b). Then we have  $X_i=\ov{X}_i\sm X''$, hence $\gdeg(X_i)\leq \deg(\ov{X_i})+\gdeg(X'')\leq\gdeg(\ov{X})+\gdeg(X'')=\gdeg(X)$.
\end{proof}

\section{Norm on cycles classes and uniform boundedness}\label{uniform_Gysin_Appendix}

Let $k$ be an algebraically closed field. Using the notion of geometric degree, in this section we define an integer valued norm on the Chow group of any quasi-projective variety and study its properties.

\subsection{Definition of norm and basic properties}
Motivated by Hrushovski's Norm \cite[Notation 10.12]{Hr}, we propose the following definition.

%\addtocontents{toc}{\protect\setcounter{tocdepth}{2}}

%Throughout this section we work with embedded quasi-projective varieties over an algebraically closed field.
%~We note that the graph of a morphism between embedded quasi-projective varieties is naturally an embedded quasi-projective variety,~and we shall use this without further mention.

%\subsection{A norm on the Chow group}

%Let $k$ be an algebraically closed field.~Let $i \colon X \hookrightarrow \B{P}^n_k$ be a quasi-projective variety.~Using the notion of geometric degree we define an integer valued norm on the Chow group of any quasi-projective variety.

%\begin{Def}\label{absolute_geo_deg}
%Let $Z \subseteq X$ be a reduced closed subscheme.~We define the \textit{absolute geometric degree} (denoted by $\on{gdeg}_{a}$) of $Z$ to be
%\begin{equation}\label{abs_geo_deg_1}
%\on{gdeg}_a(Z) \colonequals \min_{i \colon X \hookrightarrow \B{P}^n_k} \on{gdeg}(i(Z)),
%\end{equation}
%\noindent here $i$ varies over all possible embeddings of $X$ inside some $\B{P}^n_k$.
%\end{Def}

\begin{Def}\label{Hrushovski_Norm}
For a subscheme $X\subseteq\B{P}^n_k$ and $\al\in A_*(X)$, we define

\begin{equation}\label{norm_definition}
|\al|\colonequals \min \limits_{\sum_j m_j [Z_j]} \{\sum_j |m_j| \on{gdeg}(Z_j)\},
\end{equation}
\noindent where the minimum is taken over all representatives $\sum_j m_j [Z_j]\in Z_*(X)$ of $\al$, where each
$Z_j$ is a closed subvariety of $X$.
\end{Def}

\begin{Emp}\label{E:dep}
{\bf Remark.} By definition, the norm $|\al|$ depends on the embedding $\iota:X\hra \B{P}^n_k$. To indicate this dependence, we will sometimes
write  $|\al|_{\iota}$ instead of $|\al|$.
\end{Emp}

%
%\begin{Def}(Hrushovski's Norm)\label{Hrushovski_Norm}
%For $[Z] \in Z^*(X)$,~define
%
%\begin{equation}\label{norm_definition}
%|[Z]|_{\sL}:=\sup \limits_{[Y]} \inf \limits_{([Z'],~[Z''])} (\on{deg}([Z'])+\on{deg}([Z''])),
%\end{equation}
%
%here $[Y]$ runs through all cycles on $X$,~and $([Z'],[Z''])$ runs through pairs of \textit{effective} algebraic cycles such that $[Z]$ is rationally equivalent to $[Z']-[Z'']$,~and both $[Z']$ and $[Z'']$ intersect $[Y]$ \textit{properly}.
%
%\end{Def}

%The following lemma lists elementary properties of the norm.

\begin{Lem}\label{L:elementary}
(a) For any $\al_1,\al_2 \in A_*(X)$ and $m_1,m_2\in\B{Z}$, we have an inequality
\[
|m_1\al_1+m_2\al_2|\leq |m_1| |\al_1|_i+|m_2||\al_2|.
\]

(b) Assume that $X\subseteq\B{P}^n_k$ is closed. Then for every closed equidimensional subscheme $Y\subseteq X$ of dimension $i$, its class $[Y]\in A_i(X)$ satisfies $|[Y]|=\deg(Y)$.
\end{Lem}

\begin{proof}
Assertion (a) is clear, so it remains to show (b). Notice that by definition, the cycle $[Y]\in Z_i(X)$ equals
$\sum_a n_a[Y_a]$, where $Y_a$ are the irreducible components of $Y$ and each $n_a\geq 0$. The assertion follows from
the fact that for every representative $\sum_{j=1}^r m_j [Y'_j]\in Z_i(X)$ of $[Y]\in A_*(X)$, we have an inequality
\[
\on{deg}(Y)=\sum_j  m_j \on{deg}(Y'_j)\leq \sum_j  |m_j| \on{deg}(Y'_j)=\sum_j  |m_j| \on{gdeg}(Y'_j).
\]
\end{proof}

%\begin{Emp}
%{\bf Question.}
%Is the norm in Definition \ref{Hrushovski_Norm} constructible in smooth projective families?
%\end{Emp}

%Before we proceed to prove uniform boundedness properties,~we extend the notion of boundedness (see Notation \ref{N:bounded}) to cycles.

%\begin{Def}\label{bounded_cycle_defn} (To CORRECT!!!)
%Let $\{X_{\alpha}\}_{\alpha \in \Lambda}$ be a bounded family of schemes.~A family of cycles $[Z_{\alpha}] \in Z_*(X_{\alpha})$ is a family of %closed subvarieties $\{Z_{j,\alpha}\}_{\alpha \in \Lambda},~1 \leq i \leq r$ such that $[Z_{\alpha}]=\sum_{i=1}^{r}m_j[Z_{j,\alpha}]$ for some %integers $r$ and $m_i$ independent of $\alpha$.
%\end{Def}

The following simple lemma is crucial for what follows.

\begin{Prop}\label{naive_Bezout}
For every reduced subschemes $X,Y \subseteq \B{P}^n_k$ we have an inequality
\begin{equation*}\label{main_bound_app_uniform}
\on{gdeg}\left ( \left (X\cap Y\right )_{\on{red}} \right)  \leq \on{gdeg}(X) \on{gdeg}(Y).
\end{equation*}
\end{Prop}

\begin{proof}
When $X$ and $Y$ are closed subschemes, the bound was proved in \cite[Example 8.4.6]{Fu} using the join construction.
In general, assume that $X= X'\sm X''$ and $Y= Y'\sm Y''$, where $X', X'', Y', Y''$ are closed reduced subschemes of $\B{P}^n_k$ such that $\on{gdeg}(X)= \on{gdeg}(X')+\on{gdeg}(X'')$ and  $\on{gdeg}(Y)= \on{gdeg}(Y')+\on{gdeg}(Y'')$. Using identity
$X\cap Y=(X'\cap Y') \sm (X''\cup Y'')$, the result follows from  \cite[Example 8.4.6]{Fu}.
\end{proof}

We will also be using the following standard lemma.

\begin{Lem}\label{degree_product}
Let $X \subseteq \B{P}_k^m$ and $Y \subseteq \B{P}_k^n$ be closed subvarieties of degrees $\dt_X$ and $\dt_Y$, respectively.~Then the degree of $X \times Y\subseteq \B{P}_k^m\times \B{P}_k^n$  is $\binom{\dim(X) +\dim(Y)}{\dim(X)} \dt_X\dt_Y$.
%~Here $d_X$ and $d_Y$ are the dimensions of $X$ and $Y$,~respectively.
\end{Lem}

\begin{proof}
The Hilbert polynomial of $X \times Y$ under the Segre embedding equals the product of the Hilbert polynomials of $X$ and $Y$. Thus the result follows.
\end{proof}
%We will need the following lemma which allows us to compare geometric degrees under different embeddings.
%
%\begin{Lem}\label{compare_geo_deg_embe}
%Let $i_1$ and $i_2$ be two embeddings of $X$ inside $\B{P}^m_k$ and $\B{P}^n_k$ respectively.~Then there exists a constant $C$ (depending only on $X$,~$m$ and $n$) such that for any cycle $[Z] \in Z^*(X)$,~$|[Z]|_{i_1} \leq C |[Z]|_{i_2}$.
%
%Moreover if $i_1$,~$i_2$ and $[Z]$ extend to a bounded family of quasi-projective varieties then $C$ can be chosen to be bounded along this family.
%\end{Lem}
%
%\begin{proof}
%Using Lemma \ref{L:elementary}(d) we may assume that $Z$ is a closed subvariety of $X$.~
%
%\end{proof}

%\begin{Cor}\label{compare_norm_arbit}
%Let $i \colon X \hookrightarrow \B{P}^n_k$ be an embedding.~Then there exists a constant $C$ (depending only on $X$ and $n$) such that for any subvariety $Z \subseteq X$,~$|[Z]| \leq \on{gdeg}(i(Z)) \leq C |[Z]|_X$.
%\end{Cor}

\begin{Lem}\label{L:pushforward}
Let $X \subseteq \B{P}^n_k$ be a subvariety, $Z \subseteq \B{P}^m_X$ a closed subvariety, and $p\colon \B{P}^m_X \to X$
the projection map.

(a) Assume that $X \subseteq \B{P}^n_k$ is closed. Then $\on{deg}(p_*([Z]) )\leq \on{deg}(Z)$.
%\footnote{Here the degree of $Z$ is computed inside $\B{P}^{mn+m+n}_k$.}

(b) For a general $X$, we have an inequality $\on{gdeg}(p_*([Z]))\leq \on{gdeg}(Z)\on{gdeg}(X)$.
\end{Lem}

\begin{proof}

(a) The class of $[Z]$ in $\B{P}^m_k \times \B{P}^n_k$ is of the form $\sum_j n_j [\B{P}^{j}_k \times \B{P}^{\on{dim}(Z)-j}_k]$, where all the $n_j$'s are non-negative \cite[Example 8.4.2]{Fu}. On the other hand, the image of $p_*([Z])$ in $A_{\on{dim}(Z)}(\B{P}^{n}_k)$ is $n_0 \leq \on{deg}(Z)$.

(b) Note that $p(Z)\subseteq X$ is a closed subvariety. We may assume that $\on{dim}(p(Z))=\on{dim}(Z)$ (since otherwise $p_*([Z])=0$),
In this case, it suffices to show the inequality
\[
\on{deg}(p|_Z)\on{gdeg}(p(Z))\leq\on{gdeg}(Z)\gdeg(X).
\]

By \rl{gdeg}, there exists a closed reduced subscheme $X''\subseteq\B{P}^n_k$ such that $X=\ov{X}\sm X''$ and $\gdeg(X)=\gdeg(\ov{X})+\gdeg(X'')$.
Let $\ov{p}$ be the projection $\B{P}^m_{\ov{X}}\to \ov{X}$, and let $\bar{Z}\subseteq\B{P}^m_{\ov{X}}$ be the closure of $Z$. Then the image $\ov{p}(\bar{Z})\subseteq\B{P}^n_k$ is closed, and we have $p(Z)=\ov{p}(\bar{Z})\sm X''$, thus
\[
\on{gdeg}(p(Z))\leq \on{deg}(\ov{p}(\bar{Z}))+\on{gdeg}(X'')\leq  \on{deg}(\ov{p}(\bar{Z}))(1+\on{gdeg}(X'')).
\]
Combining this with equality $\on{deg}(\ov{p}|_{\bar{Z}})=\on{deg}(p|_Z)$ and part (a), we conclude that
\[
\on{deg}(p|_Z)\on{gdeg}(p(Z))\leq \on{deg}(\ov{p}|_{\bar{Z}})\on{deg}(\ov{p}(\bar{Z}))(1+\on{gdeg}(X''))\leq \on{deg}(\bar{Z})(1+\on{gdeg}(X'')).
\]
Since $1+\on{gdeg}(X'')\leq \on{gdeg}(X)$ and $\on{deg}(\bar{Z})\leq\on{gdeg}(Z)$ (by \rl{gdeg}), the assertion follows.
\end{proof}

\subsection{Uniform boundness of standard operations}

In this section we study the behaviour of the norm under standard operations in intersection theory. Our ultimate goal will be to study its behaviour under Gysin pullback and hence under intersection products.

\begin{Emp} \label{proj mor}
{\bf Set up.} Let $X\subseteq\B{P}_k^n$ and $Y\subseteq \B{P}^m_X\subseteq\B{P}^{m+n+mn}_k$ be subschemes, and let
$p$ be the composition $Y\hookrightarrow\B{P}^m_X\surj X$.
%The chosen embedding $i \colon X \hookrightarrow \B{P}^n_k$ induces an embedding $i' \colon Y \hookrightarrow\B{P}^m_X\hookrightarrow \B{P}^{mn+m+n}_k$.
\end{Emp}

%Using the same notations as above
%We can also uniformly bound the geometric degrees of the inverse image.

\begin{Lem}\label{L:pullback}
 Suppose that we are in the situation of \ref{proj mor}.

(a) There exists a constant $M$ depending only on $m,n$, the geometric degrees of $X_{\red},Y_{\red}$ such that
for every subvariety $Z\subseteq X$ we have we have an inequality $\on{gdeg}((p^{-1}(Z))_{\on{red}}) \leq M \on{gdeg}(Z)$.

(b) Assume that $p$ is smooth, and let $M$ be as in (a).  Then
for every $\al\in A_*(X)$ we have we have an inequality $|p^{*}(\al)|\leq M |\al|$.

(c) Assume that $Y\subseteq \B{P}^m_X$ is closed. Then $p$ is proper, and for every $\al\in A_*(Y)$ we have an inequality $|[p_*(\al)]|\leq |\al|\on{gdeg}(X_{\red})$. Moreover, we have $|[p_*(\al)]|\leq |\al|$, if $X\subseteq\B{P}_k^n$ is closed.
\end{Lem}

\begin{proof}
(a) Note that $\left ( p^{-1}(Z) \right)_{\on{red}} = ( Y \cap \left (\B{P}^m\times Z \right) )_{\on{red}}$, where the intersection is taking place inside of $\B{P}^{mn+m+n}_k$. Hence the lemma follows from Proposition \ref{naive_Bezout} and Lemma \ref{degree_product}.

(b) By \rl{elementary}(a), it suffices to show that for every irreducible closed subvariety $Z$ of $X$, we have an inequality
$|p^{*}([Z])|\leq M\gdeg(Z)$. Since  $p$ is smooth, we have $p^{*}([Z])=[p^{-1}(Z)]$ and $p^{-1}(Z)$ is reduced. Thus the assertion follows from (a).

(c) Arguing as in (b), the assertion follows from \rl{pushforward} applied to $X_{\red}$.
\end{proof}

As a consequence, we study the dependence of a norm on an embedding.

\begin{Cor}\label{norms_different_embeddings}
Let $\iota_1\colon X \hookrightarrow \B{P}^{n_1}_k$ and $\iota_2\colon X \hookrightarrow \B{P}^{n_2}_k$ be two embeddings of $X$.
Then there exists a constant $M$ depending only on $n_1,n_2$ and the geometric degrees of $\iota_1(X)$ and $\iota_2(X)$  such that for every
$\al\in A_*(X)$,~we have $|\al|_{\iota_1} \leq M|\al|_{\iota_2}$.
\end{Cor}

\begin{proof}
Using \rl{elementary}(a), we may assume that $\al$ is the class $[Z]$ of an irreducible closed subvariety $Z$ of $X$, and we want to show the
inequality $|[\iota_1(Z)]|\leq M|[\iota_2(Z)]|$. Let
\[
\iota:(\iota_1,\iota_2)\colon X \hookrightarrow \B{P}^{n_1}_k \times \B{P}^{n_2}_k \subseteq \B{P}^{n_1n_2+n_1+n_2}_k
\]
be the diagonal embedding of $X$, and let $p_1:\iota_1(X)\times \B{P}^{n_2}\to \iota_1(X)$ be the projection.
Using identity
\[
[\iota_1(Z)]=(p_1)_*([(\B{P}_k^{n_1}\times \iota_2(Z)) \cap \iota(X)]),
\]
the result now follows from Proposition \ref{naive_Bezout} and \rl{pullback}(a),(c).
\end{proof}
%\subsubsection{Uniform bounds on Chern class operations}

Next we show the uniformly boundness of Chern class operations.

\begin{Prop}\label{P:chern}
(a) Let $X\subseteq\B{P}^n_k$ be a subscheme and let $\C{E}$ be a vector bundle on $X$. Then there exists a constant $M$ such that for any $\al\in A_*(X)$,~we have an inequality $| c_i(\C{E}) \cap \al|\leq M |\al|$.

(b) Moreover, $M$ only depends on the {\em numerical invariants of $(X\subseteq\B{P}^n_k,\C{E})$}, by which we mean that if the pair $(X\subseteq\B{P}^n_k,\C{E})$ belongs to a bounded collection $\{(X_{\al}\subseteq\B{P}^n_{k_{\al}},\sL_{\al})\}_{\al}$
\footnote{We say that a collection  $\{(X_{\al}\subseteq\B{P}^n_{k_{\al}},\C{E}_{\al})\}_{\al}$ is {\em bounded}, if there exists a locally closed embedding $X\hra \B{P}^n_S$ parameterizing $\{X_{\al}\subseteq\B{P}^n_{k_{\al}}\}_{\al}$ (see \rn{bounded}(e),(i)) and a vector bundle on $\C{E}$ on $X$
whose pullback to each $X_{\alpha}$ is $\sL_{\alpha}$.}, then there exists a constant $M$ that satisfies property (a)
for every $(X_{\al}\subseteq\B{P}^n_{k_{\al}},\sL_{\al})$.
 \end{Prop}

\begin{proof}
(a) Recall that the Chern class operations are homogeneous polynomials in the Segre class operations with integral coefficients \cite[Section 3.2]{Fu}. Hence by Lemma \ref{L:elementary}(a), it suffices to prove a similar assertion for the Segre class operations.
Let $e+1$ be the rank of $\C{E}$, let $p:P(\C{E})\to X$ be the projective bundle of lines in $\C{E}$, and let
$\C{O}(1)$ be the canonical line bundle on $P(\C{E})$. Recall (see \cite[Section 3.1]{Fu}) that Segre class operations are given by the formula
\[
s_i(\C{E})\cap \al=p_*[c_1(\C{O}(1))^{e+i}\cap p^*(\al)].
\]
Since $p$ is smooth and projective, it follows from \rl{pullback} that it suffices to show the assertion in the case when
$\C{E}$ is a line bundle $\C{L}$ and $i=1$.

Since $X$ is quasi-projective, $\sL$ can be written as a difference of two very ample line bundles (use \cite[Ex II, 5.7]{Ha}). Using Lemma \ref{L:elementary}(a) we may assume that $\sL$ is very ample, hence there exists an embedding $\iota:X \hookrightarrow \B{P}^{n'}_k$ such that $\sL$ is the pullback of $\sO(1)$. Using Lemma \ref{norms_different_embeddings}, we can there assume that $\iota$ is the inclusion $X\hra\B{P}^n_k$.

In this case, we claim that $|c_1(\C{L}) \cap\al|\leq |\al|$. Using Lemma \ref{L:elementary}(a) again, it suffices to show that for every
closed subvariety $Z\subseteq X$ we have an inequality $|c_1(\C{L}) \cap [Z]|\leq\gdeg(Z)$. Let $H \subseteq \B{P}^{n}_k$ be a hyperplane
such that $Z \nsubseteq H$. By definition, $c_1(\sL) \cap [Z]$ is the class $[H \cap Z]$ of the schematic intersection $H \cap Z\subseteq X$.
Moreover, by Bertini theorem \cite[Th\'eor\`eme 6.3(4)]{Jo}, there exists $H$ such that $H \cap Z$ is an integral scheme, if $\dim(Z)\geq 2$, and $H \cap Z$ is reduced, if $\dim(Z)=1$. In both cases, we have $|[H \cap Z]|\leq\gdeg((H \cap Z)_{\red})\leq\gdeg(Z)$ by Proposition
\ref{naive_Bezout}.

(b) The assertion follows from the fact that the argument of (a) ``can be carried out in families''.  Namely, we can assume that $S$ is affine.  Then, as in (a), for every locally closed subscheme $X\subseteq\B{P}_S$ every line bundle is a difference of two very ample bundles.
The rest of the argument follows by repeating the argument of (a) word-by-word and using Corollary \ref{C:bounded gdeg} and the uniformity assertions in Lemma \ref{norms_different_embeddings} and \rl{pullback}.
\end{proof}

%The following proposition is now an easy consequence of the definition of Chern class operations and the lemmas proved above.

%\begin{Prop}\label{P:chern}
%Let $i \colon X \hookrightarrow \B{P}^n_k$ be a quasi-projective variety and $E$ a vector bundle over $X$.~Then there exists a constant $M$ such that for any cycle $[Z] \in Z_*(X)$,~we have $|c_j(E) \cap [Z]|_i \leq M |[Z]|_i$.~Moreover,~if the triple $(X,i,E)$ belongs to a bounded collection,~then $M$ can be chosen to be bounded.

%\end{Prop}

%\begin{proof}
%By definition,~the Chern class operations are homogeneous polynomials in the Segre class operations with integral coefficients \cite[Section 3.1]{Fu}.~Hence by Lemma \ref{L:elementary} it suffices to prove a similar assertion for the Segre class operations.~For the Segre class operations the result follows by a combination of Lemmas \ref{Chern_Class_universal} and \ref{proper_pushforward_norm},~and  Corollary \ref{schematic_inverse_smooth}.
%\end{proof}

%\subsection{Uniform boundedness of refined Gysin pullback}

Finally, we show the uniform boundedness under Gysin pullbacks.

\begin{Thm}\label{uniform_bounded_Gysin_pullback}
Let $i:Y\hra X$ be a closed regular embedding of codimension $d$ of subschemes of $\B{P}^n_k$, let $X' \subseteq \B{P}^m_X$ be
a subscheme, and let $i':Y':=Y\times_X X'\hra X'$ be the pullback of $i$.

%\footnote{It suffices to assume that $i$ is a regular embedding,~but we do not need this additional generality for the purposes of this article.}

(a) Then there exists a constant $M$ such that for every cycle $\al\in A_*(X')$, the refined pullback $i^*(\al)\in A_{*-d}(Y')$ satisfies
 $|i^*(\al)| \leq M |\al|$.

(b) Moreover, $M$ only depends on {\em the numerical invariants of $(Y\overset{i}{\hra}X\subseteq\B{P}^n_k, X' \subseteq \B{P}^m_X)$.}\footnote{As in \rp{chern}, it means that if a diagram $(Y\overset{i}{\hra}X\subseteq\B{P}^n_k, X' \subseteq \B{P}^m_X)$ belongs to a bounded collection $\{(Y_{\al}\overset{i_{\al}}{\hra}X_{\al}\subseteq\B{P}^n_{k_{\al}}, X'_{\al} \subseteq \B{P}^m_{X_{\al}})\}_{\al}$,
 then there exists a constant $M$ that satisfies property (a) for every diagram
 $(Y_{\al}\overset{i_{\al}}{\hra}X_{\al}\subseteq\B{P}^n_{k_{\al}}, X'_{\al} \subseteq \B{P}^m_{X_{\al}})$.}
\end{Thm}

\begin{proof} We will carry out the proof in 6 steps.

\vskip 4truept

\noindent{\bf Step 1.} By \rl{elementary}(a), it suffices to show that there exists a constant $M$, only depending on the numerical invariants of
$(Y\overset{i}{\hra}X\subseteq\B{P}^n_k, X' \subseteq \B{P}^m_X)$, such that for every closed subvariety $V'\subseteq X'$, we have inequality
$|i^*([V'])|\leq M\gdeg(V')$.
%
%Moreover, if  $(Y\overset{i}{\hra}X\subseteq\B{P}^n_k, X' \subseteq \B{P}^m_X)$ belongs to a bounded collection, then there exists $M$ which satisfies \form{ref} for every member of the collection.

\vskip 4truept

\noindent{\bf Step 2.} Assume that the closed embedding $i'|_{V'}:i'^{-1}(V')\hra V'$ is an isomorphism. In this case we claim that we have an equality
\begin{equation} \label{Eq:triv case}
i^*([V']_{X'})=c_d(N_Y(X)|_{Y'}) \cap [i'^{-1}(V')]_{Y'},
\end{equation}
where $N_Y(X)|_{Y'}$ is the pullback to $Y'$ of the normal bundle of $Y$ in $X$. Indeed, consider Cartesian diagram

\begin{equation*}
\xymatrix{ i'^{-1}(V') \ar[r]_{\sim}^{i'|_{V'}} \ar[d]_{q} & V' \ar[d]^-{p} \\
                 Y' \ar[r]^-{i'} \ar[d] & X' \ar[d] \\
                 Y \ar[r]^-{i} & X.}
\end{equation*}

Then using the self intersection formula \cite[Corollary 6.3]{Fu}, compatibility of Gysin pullbacks and with proper pushforward \cite[Theorem 6.2(a)]{Fu} and projection formula for Chern classes \cite[Theorem 3.2(c)]{Fu}, we have an equality
\[
i^*([V']_{X'})=i^*(p_*([V']_V))=q_*(i^*([V']_{V'}))=q_*(q^*(c_d(N_Y(X)|_{Y'})) \cap [i'^{-1}(V')]_{i'^{-1}(V')})=
\]
\[
=c_d(N_Y(X)|_{Y'}) \cap q_*([i'^{-1}(V')]_{i'^{-1}(V')})=c_d(N_Y(X)|_{Y'}) \cap [i'^{-1}(V')]_{Y'}.
\]

Moreover, if $(Y\overset{i}{\hra}X\subseteq\B{P}^n_k, X' \subseteq \B{P}^m_X)$ belongs to a bounded collection, then
the vector bundle $N_Y(X)|_{Y'}$ on $Y'$ belongs a bounded collection (by Corollary \ref{boundedness_normal_cone} and \rl{simple bounded}(f)). Since $|[i'^{-1}(V')]|\leq\gdeg(i'^{-1}(V'))=\gdeg(V')$, the assertion of Step 1 now
follows from a combination of \form{triv case} and Proposition \ref{P:chern}.

\vskip 4truept

From now on we assume that  $i'|_{V'}:i'^{-1}(V')\hra V'$ is not an isomorphism.

\vskip 4truept

\noindent{\bf Step 3.} Assume that $X\subseteq\B{P}^n$ is a closed subscheme, $d=1$ and $X'\to X$ is an open embedding with dense image.

\vskip 4truept

To simplify the notation, we will view $Y$ and $X'$ as a subschemes of $X$. Then the schematic intersection $Y'\cap V'\subseteq V'$ is a Cartier divisor, and we have an equalities
$i^*([V'])=i'^*([V'])=[Y'\cap V']\in A_*(Y')$ (see \cite[Ex 2.6.5 and Corollary 6.3]{Fu}) and $i'_*i'^*([V'])=c_1(\C{O}(Y'))\cap [V']$ (see \cite[Propositions 2.6(d) and 6.1(c)]{Fu}).

\vskip 4 truept

(a) Assume first that $X'=X$ (hence also $Y'=Y$ and $i'=i$). Set $V:=V'$. Then by \rl{elementary}(b), we have an equality
$|[Y\cap V]|=\deg(Y\cap V)=|i_*([Y\cap V])|$, which implies that $|i^*([V])|=|c_1(\C{O}(Y))\cap [V]|$.

Moreover, if $(Y\subseteq X\subseteq \B{P}^n)$ belongs to a bounded collection, then the line bundle $\C{O}(Y)$ on $Y$ belongs to a bounded collection (by \rl{simple bounded}(c)), so in this case the assertion of Step 1 follows from Proposition \ref{P:chern}.

\vskip 4truept

(b) In general, let $V\subseteq X$ be the closure of $V'$. It suffices to show an inequality
\begin{equation} \label{Eq:ineq gysin}
|i^*([V'])|\leq |i^*([V])|\gdeg(X'_{\red}).
\end{equation}
Indeed, the assertion of Step 1 would then follow from an inequality $\gdeg(V)\leq\gdeg(V')$ (see \rl{gdeg}(c)), Corollaries \ref{C:reduced irreducible},   \ref{C:bounded gdeg} and case (a), shown above.

Let $Z_1,\ldots,Z_r$ be all irreducible components of $Y\cap V$. Then each
$Z'_j:=Z_j\cap X'$ is either an irreducible components of $Y'\cap V'$ or empty. Choose a closed reduced subscheme $X_2\subseteq\B{P}^n_k$ such that
$X'=X\sm X_2$ and $\gdeg(X'_{\red})=\gdeg(X_{\red})+\gdeg(X_2)$. Then we have
$1+\gdeg(X_2)\leq \gdeg(X'_{\red})$ and $Z'_j:=Z_j\sm X_2$ for all $j$.

Recall the cycle $[Y\cap V]\in Z_*(Y)$ equals $\sum_{j=1}^r n_j[Z_j]$ for some  $n_j\geq 0$.
Hence we have an equality of cycles $[Y'\cap V']=\sum_{j=1}^r n_j[Z'_j]\in Z_*(Y')$, therefore
an inequality
\[
|i^*([V'])|\leq\sum_{j=1}^r n_j\gdeg(Z'_j)\leq \sum_{j=1}^r n_j\gdeg(Z_j)+ (\sum_{j=1}^r n_j) \gdeg(X_2).
\]
Using the equality $\sum_{j=1}^r n_j\gdeg(Z_j)=\deg(Y\cap V)=|i^*([V])|$ (compare \rl{elementary}(b)), we conclude that
$(\sum_{j=1}^r n_j)\leq |i^*([V])|$, implying \form{ineq gysin}.

\vskip 4truept

\noindent{\bf Step 4}. Consider Cartesian diagram

\begin{equation*}
\xymatrix{ Y'' \ar[r]^-{i''} \ar[d]_-{q} & X'' \ar[d]^-{p} \\
                 Y' \ar[r]^-{i'} \ar[d] & X' \ar[d] \\
                 Y \ar[r]^-{i} & X,}
\end{equation*}
where $X'' \colonequals \on{Bl}_{Y'}(X')$, and $p \colon X'' \to X'$ is the projection map.

Let $V''\subseteq X'' $ be the strict transform of $V'$. Note that $p$ is proper and that $i$ (resp. $i''$) is a regular immersion of codimension $d$ (resp. $1$). Since $p_*([V''])=[V']$, it follows  by compatibility of Gysin pullback with proper pushforward \cite[Theorem 6.2(a)]{Fu} and excess intersection formula \cite[Theorem 6.3]{Fu}, that  we have an equality
\begin{equation}\label{Eq:uniform gysin}
i^*([V'])=i^*(p_*([V'']))=q_*(i^*([V'']))=q_* \left (c_{d-1}(\C{E}) \cap (i'')^*([V'']) \right ),
\end{equation}
%
%Here $\sO_{X''}(1)$ is the restriction of $\sO_{Y''/Y'}(1)$ to $X''$.
where $\C{E}:= q^*(N_Y(X)|_{Y'})/N_{Y''}(X'')$.

\vskip 4truept

\noindent{\bf Step 5.}
Let $\ov{X}\subseteq \B{P}^n_k$ be the closure of $X$, and let  $\ov{X}',\ov{Y}'\subseteq  \B{P}^m_{\ov{X}}$ be the closures of $X'$  and $Y'$, respectively. We set  $\ov{X}'' \colonequals \on{Bl}_{\ov{Y}'}(\ov{X}')$, and let $\ov{Y}''\subseteq\ov{X}''$ be the exceptional divisor.
Then we have a Cartesian diagram

\begin{equation*}
\xymatrix{ Y'' \ar[r]^-{i''} \ar[d] & X'' \ar[d]^-{j} \\
                 \ov{Y}'' \ar[r]^-{\ov{i}''} & \ov{X}'',}
\end{equation*}
where $\ov{i}''$ is a regular embedding of codimension one, and $j$ is an open embedding with dense image. In particular,
we get to the situation of Step 3.
\vskip 4truept

\noindent{\bf Step 6.} We claim that there exist constants $M_1,M_2,M_3$, which only depend on the numerical invariants of $(Y\overset{i}{\hra}X\subseteq\B{P}^n_k, X' \subseteq \B{P}^m_X)$, such that for each closed subvariety $V'\subseteq X'$ such that $i'|_{V'}:i'^{-1}(V')\hra V'$ is not an isomorphism, we have
\[
%\begin{equation*} \label{Eq:gysin}
|i^*([V'])|\leq M_1 |(i'')^*([V''])|\leq M_1M_2 \gdeg(V'')\leq  M_1M_2M_3 \gdeg(V').
\]

Indeed, assume that a diagram $(Y\overset{i}{\hra}X\subseteq\B{P}^n_k, X' \subseteq \B{P}^m_X)$ belongs to a bounded collection. Then it follows from a combination of Lemma \ref{L:simple bounded}(f),(g) and Corollary \ref{C:blowup} that the full Cartesian diagrams of Steps 4 and 5 belong to
bounded collections. In particular, the existence of $M_2$ follows from Steps 3 and 5.

Next, it follows from Corollary \ref{boundedness_normal_cone} that vector bundle $\C{E}$ on $Y'$, defined in Step 4, belongs to a bounded collection. Therefore the existence of $M_1$ follows from equality \form{uniform gysin} together with Proposition \ref{P:chern}, Lemma \ref{L:pullback}(c) and \rco{bounded gdeg}.

Finally, since $V''$ is an irreducible component of $p^{-1}(V')$, we get an inequality $\gdeg(V'')\leq \gdeg(p^{-1}(V')_{\red})$ (by \rl{gdeg}(b)),
so the existence of $M_3$ follows from \rco{bounded gdeg} and Lemma \ref{L:pullback}(a).
\end{proof}

\begin{Cor} \label{C:bound gysin}
Let $X\subseteq \B{P}^n_k$ and $Y\subseteq \B{P}^m_X$ be smooth subvarieties, let $Z\subseteq X$ be a closed reduced subscheme,
and let  $f$ be a composition $Y\overset{i}{\hra} \B{P}^m_X\overset{p}{\surj} X$. Then there exists a constant $M$ depending on the
on $n,m$ and geometric degrees of $X,Y$ and $Z$ such that for every $\al\in A_*(Z)$, the refined pullback $f^*(\al)\in A_*(f^{-1}(Z))$ satisfies
 $|f^*(\al)| \leq M |\al|$.
\end{Cor}

\begin{proof}
Since $f^*=i^*\circ p^*$, the assertion follows from a combination of \rl{pullback}(b), (the proof of) \rco{gdeg mor} and Theorem \ref{uniform_bounded_Gysin_pullback}.
\end{proof}

\begin{Cor}\label{behaviour_norm_intersection}
Let $X\subseteq\B{P}_k^n$ be a smooth closed subscheme. Then there exists a constant $M$,
 depending on $n$ and the geometric degree of $X$, such that
for all $\al,\beta\in A_*(X)$, we have an inequality $|\al\cap\beta| \leq M |\al| |\beta|$.
\end{Cor}

\begin{proof}
Let $\Delta:X\hra X\times X\subseteq \B{P}_k^n\times\B{P}_k^n$ be the diagonal embedding. By definition, we have $\al\cap\beta=\Delta^*(\al\times \beta)$. Thus using Theorem \ref{uniform_bounded_Gysin_pullback} and Corollary \ref{norms_different_embeddings},
it suffices to show inequality $|\al\times\beta|\leq M|\al||\beta|$, where $M$ depends only on $n$ and the degree of $X$.
As before,~using Lemma \ref{L:elementary}(a) we may assume that $\al=[V]$ and $\beta=[W]$ for some closed subvarieties
$V,W\subseteq X$. The result now follows from Lemma \ref{degree_product}.
\end{proof}

We need the following estimate for the norm of a composition of correspondences.

\begin{Emp} \label{E:ring}
{\bf Composition of correspondences.} (a) Let $X_1$, $X_2$ and $X_3$ be $d$-dimensional smooth proper schemes over $k$.
Then to every $u \in A_d(X_1 \times X_2)$ and $v \in A_d(X_2 \times X_3)$ one associates the composition
$v \circ u:= p_{13*}(p_{12}^*(u) \cap p_{23}^*(v))\in A_d(X_1 \times X_3)$. Here as usual $p_{ij}$ is the projection from $X_1 \times X_2 \times X_3$ to $X_i \times X_j$.

(b) In particular, when $X_1=X_2=X_3$, the above composition equips $A_d(X \times X)$ with the structure of a ring with the graph of identity as the unit. Furthermore, the assignment $\al\mapsto H^i(\al)$ defines an action of $A_d(X \times X)$ on $H^i(X,\ql)$
for all $\ell\neq\on{char}(k)$ and all $i\in\B{N}$.
\end{Emp}
%~Corollary \ref{behaviour_norm_intersection} and Lemma \ref{degree_product} immediately implies that following.

\begin{Cor}\label{Hrushovski_Norm_Correspondence}(\cite[Lemma 10.17(3)]{Hr})
For every $n,\dt\in\B{N}$, there exists a constant $N$, depending on $(n,\dt)$ such that for all $d\leq n$,

$\bullet$ all triples $d$-dimensional smooth closed subvarieties $X_1,X_2,X_3 \subseteq \B{P}^n_k$ (over an arbitrary algebraically closed field $k$) of degree $\dt$;

$\bullet$ and all pairs $u\in A_d(X_1 \times X_2)$ and $v\in A_d(X_2 \times X_3)$,

~we have inequality
\begin{equation}\label{modified_norm_0}
|v\circ u| \leq N |v|\cdot |u|.
\end{equation}
Here the norms are taken with respect to the Segre embedding of the product varieties.
\end{Cor}

\begin{proof}
Using definition of $v\circ u$ (see \re{ring}(a)), the assertion follows from a combination of Corollary \ref{behaviour_norm_intersection} with
\rl{pullback}.
\end{proof}

\begin{Def}\label{modified_norm}
(a) For  $n,\dt\in\B{N}$, we denote by $N(n,\dt)$ the smallest possible integer $N$ satisfying the assertion of Corollary \ref{Hrushovski_Norm_Correspondence}.

(b) For $d\leq n$, a pair $X_1,X_2\subseteq\B{P}^n_k$ of $d$-dimensional smooth closed subvarieties of degree $\dt$ and a class $\al\in A_d(X_1 \times X_2)$, we define a {\em renormalized norm} of $||\al||$ of $\al$ by the formula $||\al|| \colonequals N(n,\dt)|\al|$.
%~where the norm on the right is the one defined in (\ref{norm_definition}) with respect to the Segre embedding $X_1 \times X_2\subseteq \P_k^n \times \P_k^n\subseteq \P_k^{n^2+2n}$.
\end{Def}

%
%\begin{equation}\label{modified_norm}
%||[Z]|| \colonequals M|[Z]|,
%\end{equation}

\begin{Emp} \label{reform}
{\bf Remark.} In the notation of Definition \ref{modified_norm}, the inequality (\ref{modified_norm_0}) of Corollary \ref{Hrushovski_Norm_Correspondence} can now be rewritten in the form
\begin{equation}\label{Eq:mult}
||v\circ u|| \leq ||v||\cdot ||u||.
\end{equation}
\end{Emp}

%\subsection{Norm of a correspondence bounds the spectral radius}
%In this section we use sub-multiplicativity of the norm to bound the norm of a self-correspondence by the spectral radius of its action on cohomology.~This is done using reduction to finite fields,~and then using Katz-Messing (see \cite[Theorem 1.4]{Shu19b} for a similar argument).

\subsection{A bound of the spectral radius}
In this section we apply the submultiplicativity of the renormalized norm \form{mult}
to obtain a bound a spectral radius of an action of a cycle on a cohomology. First we treat of the case of varieties over finite fields,
following \cite{Hr} very closely and using Katz-Messing \cite{KM} (see \cite[Theorem 1.4]{Shu19b} for a similar argument). Then we deduce the general case by standard spreading out argument.

%Let $\ell$ be any prime invertible in $k$.~We fix an embedding $\tau \colon \ql
%\hookrightarrow \B{C}$.~Thus given an endomorphism of a $\ql$-vector space,~we may consider the eigenvalues as complex numbers (via $\tau$).~Moreover,~we may also define the spectral radius of any such endomorphism as the maximum in absolute value of its complex eigenvalues.

%In this section we use sub-multiplicativity of the norm to bound the norm of a self-correspondence by the spectral radius of its action on cohomology.~This is done using reduction to finite fields,~and then using Katz-Messing (see \cite[Theorem 1.4]{Shu19b} for a similar argument).

\begin{Emp} \label{E:purity2}
{\bf Spectral radius.}
Let $X$ be a smooth proper variety of dimension $d$ over an algebraically closed field $k$ and let $\al\in A_d(X\times X)$.
For every $\ell\neq\on{char}(k)$ and a field embedding $\tau:\ql\hra\B{C}$ one associates an endomorphism $\tau(H^i(\al))$ of the $\B{C}$-vector space $H^i(X,\ql)\otimes_{\ql}\B{C}$. Hence one can consider its spectral radius $\rho(\tau(H^i(\al))\in\B{R}$, that is, maximum of norms of its eigenvalues.

Using the identity $\rho(A)=\limsup_n |\Tr(A^n)|^{\frac{1}{n}}$, it follows from \re{purity} that the spectral radius $\rho(H^i(\al)):=\rho(\tau(H^i(\al))$ is independent of $\ell\neq\on{char}(k)$ and $\tau$.
\end{Emp}

\begin{Prop} \label{P:spectral}(\cite[Proposition 11.11]{Hr})
Let $k$ be an algebraically closed field, let $X\subseteq\B{P}^n_k$ be a smooth closed subvariety over $k$ of dimension $d$, let
$\al\in A_{d}(X \times X)$, and let $i\in\B{N}$. Then the spectral radius of $H^{i}(\al)$ (see \re{purity2}) is bounded above by the $||\al||$ (see Definition \ref{modified_norm}).
%Here the norm is the modified one (see Equation \ref{modified_norm}) with respect to $\sL \boxtimes \sL$.
\end{Prop}

\begin{proof}
{\bf Finite field case.} First we are going to show the assertion in the case when $k$ is an algebraic closure of a finite field $\F_q$. Replacing $q$ by its power, we can assume that $X$  is defined over $\F_q$. Let $\on{F}_{X} \colon X \to X$ be the (geometric) Frobenius with respect to $\F_q$.
%~Then $H^{i}([C])$ and $\on{F}_{X}^*$ commute as endomorphisms of $H^{i}(X,\ql)$.
By \cite[Theorem 2]{KM},~there exists a polynomial $P_{i}[T] \in \B{Q}[T]$ such that the endomorphism $P_{i}(\on{F}_X^*)$ of $H^j(X,~\ql)$ is zero for $j \neq i$ and identity for $j=i$. We choose $m\in\B{N}$ such that $(mP_i)(T)\in\B{Z}[T]$.

Let $\al^{\circ n} \in A_{d}(X \times X)$ be the $n$-th power of $\al$ in the ring  $A_{d}(X \times X)$ (see \re{ring}),
let $[\Gm_X]\in A_{d}(X \times X)$ be the class of the graph of $F_X$, and let $(mP_i)([\Gm_X])\in A_{d}(X \times X)$
be the element obtained from $[\Gm_X]$ using the ring structure on $A_{d}(X \times X)$.

By repeated application of equation \form{traces},~we get inequality
\begin{equation}\label{trace_norm}
\on{Tr} \left( H^i(\al)^{n}, H^{i}(X,\ql)\right)= \frac{(-1)^{i}}{m}  \al^{\circ n} \cdot (mP_{i})([\Gm_X])\in\B{Q}.
\end{equation}

Hence, by Corollaries \ref{behaviour_norm_intersection} and \ref{Hrushovski_Norm_Correspondence} (in the form of (\ref{Eq:mult}), there exists a constant $M$, such that for all
$n\in\B{N}$ we have an inequality
\begin{equation}\label{trace_norm_1}
|\on{Tr} \left( H^i(\al)^{n},H^i(X,\ql) \right)|\leq   M |(mP_{i})([\Gm_X])|\cdot ||\al||^n.
\end{equation}
Since the spectral radius of $H^{i}(\al)$ equals $\limsup_n |\on{Tr} \left( H^i(\al)^{n},H^{i}(X,\ql)\right)|^{\frac{1}{n}}$,
the proposition in this case follows from inequality (\ref{trace_norm_1}).

\vskip 4truept
{\bf General case.} We claim that the general case follows from the finite field case, shown above.
Indeed, fix a prime $\ell$, invertible in $k$. By \rl{elementary}(a), we can assume that $\al$ is the class $[C]$ of a closed irreducible subvariety $C\subseteq X\times X$. We can assume that $k$ is a geometric generic point of a scheme $S$  of finite type over $\B{Z}$, that
that $X$ is a geometric generic fiber of a closed subscheme $\wt{X}\subseteq\B{P}^n_S$ and $C$ is the geometric generic fiber of a closed subscheme $\wt{C}\subseteq\wt{X}\times_S\wt{X}$. Moreover, shrinking $S$, if necessary, we can assume that $S$ is connected, $\ell$ is invertible in $S$, $\wt{X}$ is smooth over $S$ and $\wt{C}$ is flat over $S$.

For every geometric point $\bar{s}$ of $S$ we denote by $\wt{X}_{\bar{s}}$ and $\wt{C}_{\bar{s}}$ the fibers of $\wt{X}$ and $\wt{C}$ over $\bar{s}$, respectively. Since every closed point of $S$ is defined over a finite field, to show the reduction, it suffices to show that the modified norm $||[\wt{C}_{\bar{s}}]||$ and the spectral radius of the action of $H^{i}([\wt{C}_{\bar{s}}])$ on $H^{i}(\wt{X}_{\bar{s}},\ql)$ do not depend on $\bar{s}$. The assertion about $||[\wt{C}_{\bar{s}}]||$ follows from equality $|[\wt{C}_{\bar{s}}]|=\deg (\wt{C}_{\bar{s}})$ (Lemma \ref{L:elementary}(b)) and
observation that $\deg (\wt{C}_{\bar{s}})$ and $\deg (\wt{X}_{\bar{s}})$ do not depend on $\bar{s}$, because $\wt{X}$ and $\wt{C}$ are flat over $S$. The assertion about the spectral radius follows from \rcl{indep} below.
\end{proof}

\begin{Cl} \label{C:indep}
The characteristic polynomial of $H^i([\wt{C}_{\bar{s}}])$ is independent of $\bar{s}$.
\end{Cl}
\begin{proof}
It suffices to show that for every $n\in\B{N}$ the trace $\Tr(H^i([\wt{C}_{\bar{s}}])^n, H^i(\wt{X}_{\bar{s}},\ql))$ is independent of $\bar{s}$.
To see this it suffices to construct a global section $a(n)\in H^0(S,\ql)$ such that for every geometric point
$\bar{s}$ of $S$, the pullback $a(n)_{\bar{s}}\in H^0(\bar{s},\ql)=\ql$ of $a(n)$ equals $\Tr(H^i([\wt{C}_{\bar{s}}])^n, H^i(\wt{X}_{\bar{s}},\ql))$.
Indeed, the independence of $\bar{s}$ then would follow from the fact that $S$ is connected, thus $a(n)$ is constant.

The construction of $a(n)$ is standard: Since the projection $p:\wt{X}\to S$ is proper and smooth, the derived push-forward
$R^ip_*(\ql)$ is a local system on $S$, and we have a natural isomorphism  $R^ip_*(\ql)_{\bar{s}}\simeq H^i(\wt{X}_{\bar{s}},\ql)$
for all $\bar{s}$. It therefore suffices to construct an endomorphism $\C{H}^i(\wt{C}):R^ip_*(\ql)\to R^ip_*(\ql)$, whose geometric fiber
at $\bar{s}$ is $H^i(\wt{C}_{\bar{s}})$.

Let $\pi:\wt{C}\to S$ be the projection, and $(c_1,c_2):\wt{C}\hra \wt{X}\times\wt{X}$ be the inclusion. Then we define
$\C{H}^i(\wt{C})$ to be the composition $R^ip_*(\ql)\overset{(c_1)^*}{\lra} R^i\pi_*(\ql)\overset{(c_2)_*}{\lra} R^ip_*(\ql)$,
where $(c_1)^*$ is induced morphism $Rp_*(\ql)\to R(p\circ c_1)_*(\ql)=R\pi_*(\ql)$, while $(c_2)_*$ is defined by a manner
similar to \re{end}, using the fact that $p$ is smooth, while $\pi$ is flat.
\end{proof}

\begin{Cor}\label{C:spectral_twisted}(\cite[Corollary 11.12]{Hr})
Let $k$ be an algebraically closed field of characteristic $p>0$, $q$ a power of $p$, $X\subseteq\B{P}^n_k$ a smooth closed subvariety over $k$ of dimension $d$. Then for every $\al\in A_{d}(X \times X^{(q)})$ and $i\in\B{N}$ the spectral radius of the endomorphism $F^{*}_{X,q} \circ H^{i}(\al)$ of $H^{i}(X,\ql)$ (see \re{purity}(b)) is bounded above by $q^{\frac{i}{2}}||\al||$.
\end{Cor}

\begin{proof}
{\bf Finite field case.} First we are going to show the assertion in the case when $k$ is an algebraic closure of a finite field $\F_q$.
Using \rl{elementary}(a), we may assume that $\al$ is the class $[C]$ of an irreducible subvariety $C\subseteq  X \times X^{(q)}$ and that closed subvarieties $X\subseteq\B{P}^n_k$ and $C\subseteq  X \times X^{(q)}$ are defined over $\F_{q^r}$. Then the Frobenius twist $X^{(q^r)}$ is naturally isomorphic to $X$ and the composition $F_X:X\overset{\on{F}_{X,q^r}}{\lra} X^{(q^r)}\simeq X$ is the geometric Frobenius endomorphism of $X$ over $\F_{q^r}$.

The class $[C]  \in A_{d}(X \times X^{(q)})$ gives rise to classes $[C^{(q^i)}] \in A_{d}(X^{(q^{i})} \times X^{(q^{i+1})})$. Moreover, each $X^{(q^{i})} \times X^{(q^{i+1})}$ come equipped with natural embeddings into $\B{P}^n_k \times \B{P}^n_k$,~and $||[C^{(q^i)}]||=||[C]||$ with respect to these embeddings.

We claim that

\begin{enumerate}

\item the endomorphism $\left(F^{*}_{X/k,q} \circ H^{i}([C])\right)^r$ of $H^{i}(X,\ql)$ equals
\[
\on{F}_{X}^* \circ  \left ( H^{i}([C^{(q^{r-1})}]) \circ \cdots \circ H^{i}([C^{(q)}])\circ H^{i}([C]) \right );
\]

\item endomorphisms $\on{F}_{X}^*$ and $H^{i}([C^{(q^{r-1})}]) \circ \cdots \circ H^{i}([C^{(q)}])\circ H^{i}([C])$ of $H^{i}(X,\ql)$ commute.

\end{enumerate}

Indeed, assertion (1) follows from equalities

\begin{equation}\label{P:spectral_twisted_1}
\on{F}_{X^{(j)}/k,q}^* \circ H^{i}([C^{(q^{j})}])=H^{i}([C^{(q^{j-1})}]) \circ \on{F}_{X^{(j-1)}/k,q}^*
\end{equation}
of  endomorphisms of $H^{i}(X^{(q^i)},\ql)$, equality
\[
\on{F}_X=\on{F}_{X/k,q^r}=\on{F}_{X^{(q^{r-1})}/k,q}\circ\ldots\circ \on{F}_{X^{(q)}/k,q}\circ \on{F}_{X/k,q}
\]
of endomorphisms  of $X$ and induction, while assertion (2) follows from the same equalities together with identifications $X^{(q^{j+r})}\simeq X^{(q^{j})}$ and $C^{(q^{j+r})}\simeq C^{(q^{j})}$.

Now the result easily follows. Indeed, it suffices to show that spectral radius of  $\left(F^{*}_{X/k,q} \circ H^{i}([C])\right)^r$ is at most
$q^{\frac{ir}{2}}||[C]||^r$. Using claims (1),(2) and purity (\cite{De}) it suffices to show that the spectral radius of
\[
H^{i}([C^{(q^{r-1})}]) \circ \cdots \circ H^{i}([C^{(q)}])\circ H^{i}([C])=H^{i}([C^{(q^{r-1})}] \circ \cdots \circ[C^{(q)}]\circ [C])
\]
is at most $||[C]||^r$. But this follows from a combination of Proposition~\ref{P:spectral}, inequality (\ref{Eq:mult}) and
equalities $||[C^{(q^i)}]||=||[C]||$.

\vskip 4truept
{\bf General case.} Again, we are going to deduce to the finite field case, shown above. We can assume that $k$ is a geometric generic point of a scheme $S$  of finite type over $\B{F}_p$, and
 $X$ is a geometric generic fiber of a closed subscheme $\wt{X}\subseteq\B{P}^n_S$. We let $\wt{X}^{(q)}$ be the pullback
 $\wt{X}\times_{S,F_{S,q}} S$ with respect to absolute $q$-Frobenius morphism $F_{S,q}:S\to S$. Since $F_{S,q}$ is functorial in $S$, for every geometric point $\bar{s}$ of $S$, the geometric fiber $(\wt{X}^{(q)})_{\bar{s}}$ of $\wt{X}^{(q)}$ is canonically isomorphic to $(\wt{X}_{\bar{s}})^{(q)}$, and the
geometric fiber over $\bar{s}$ of the relative Frobenius morphism $F_{\wt{X}/S,q}:\wt{X}\to \wt{X}^{(q)}$ is identified with $F_{X/\bar{s},q}$.
In particular, the geometric generic point of $F_{\wt{X}/S,q}$ is identified with $F_{X/k,q}:X\to X^{(q)}$. Furthermore, we can assume that $C$ is a geometric fiber of a closed subscheme $\wt{C}\subseteq\wt{X}\times_S\wt{X}$, $\wt{X}$ is smooth over $S$, and $\wt{C}$ is flat over $S$.
To finish the argument, we repeat the last paragraph of the proof of Proposition \ref{P:spectral} and use \rcl{indep}.
\end{proof}

\begin{Cor} \label{C:bound trace}
In the situation of \rco{spectral_twisted}, there exists a constant $M$ depending on $n$ and $\gdeg(X)$ such that
the trace $\Tr(F^*_{X,q} \circ H^{i}(\al))\in\B{Z}$ (see \re{purity}) satisfies inequality
\[
|\Tr(F^*_{X,q} \circ H^{i}(\al))|\leq M|\al| q^{i/2}.
\]
\end{Cor}

\begin{proof}
By \re{purity}, the Betti number $h^i(X):=\dim_{\ql}H^i(X,\ql)$ in independent of $\ell\neq\on{char}(k)$. Using Corollary \ref{C:spectral_twisted},  identity $||\al||=N(n,\deg(X))|\al|$ (see Definition \ref{modified_norm}) and inequality $|\Tr(F^*_{X,q} \circ H^{i}(\al))|\leq \rho(F^*_{X,q} \circ H^{i}(\al))h^i(X)$, the assertion follows from a combination of \rco{bounded gdeg} and \rl{betti} below.
\end{proof}

%We will need boundedness of Betti numbers,~which is an immediate consequence of standard constructibility arguments and independence of $\ell$ for Betti numbers of smooth projective varieties \cite[Corollary 1]{Katz-Messing} (EXPLAIN!!!)

\begin{Lem}\label{L:betti}
Let $\{X_{\alpha} \subseteq \B{P}^n_{k_{\alpha}}\}$ be a bounded collection of closed  smooth subvarieties. Then for every $i\in\B{N}$
the collection of Betti numbers $\{h^i(X_{\al})\}_{\al}$ is bounded.
%Then for any prime $\ell$ invertible in $k_{\alpha}$,~the function given by Betti numbers $\on{dim}_{\ql}(H^i(X_{\alpha},\ql))$ is bounded.
\end{Lem}

\begin{proof}
Let $X\subseteq\B{P}^n_S$ be a family representing $\{X_{\alpha} \subseteq \B{P}^n_{k_{\alpha}}\}$. Using constructibility the set of $s\in S$ such that $X_s$ is smooth, we can assume as before that the projection $f:X\to S$ is smooth. We can fix a prime $\ell$ and assume that $\ell$ is invertible in $S$. By the proper base change theorem, for every geometric point $\ov{s}$ of $S$ we have an isomorphism
$H^i(X_{\ov{s}},\ql)\simeq  R^if_*(\ql)_{\ov{s}}$. Since $R^if_*(\ql)$ is constructible, the assertion follows.
\end{proof}

%\begin{Emp} \label{E:main}
%{\bf Remark.} It suffices to show the conclusion of \rt{main} for quadruples $(k,q,X^0,C^0)$, where $k$ is the algebraic closure of $\fq$.
%Indeed, let $(k,q,X^0,C^0)$ be a tuple as in \rt{main}. Arguing as in the last paragraph of the proof Corollary \ref{P:spectral_twisted}, we can assume that $k$ is a geometric generic point of a scheme $S$  of finite type over $\B{F}_p$, and
% $X^0$ is a geometric generic fiber of a closed subscheme $\wt{X}^0\subseteq\B{P}^n_S$, $C$ is a geometric fiber of a closed subscheme $\wt{C}\subseteq\wt{X}\times_S\wt{X}$, and $F_{X^0/k,q}:X^0\to (X^0)^{(q)}$ is a geometric generic fiber of a relative Frobenius morphism $F_{\wt{X}/S,q}:\wt{X}^0\to (\wt{X}^0)^{(q)}$.
%\end{Emp}

%\newpage

%\Addresses

%\end{document}

%\end{document}

\end{document}